\documentclass[a4,11pt]{amsart}
\usepackage{hyperref}
 
\usepackage{amssymb,amsmath,amsthm,txfonts, verbatim,bm,mathtools}
\usepackage{cite}

\usepackage{color}
\usepackage{esint}

\newtheorem{thm}{Theorem}[section]
\newtheorem{lem}[thm]{Lemma}
\newtheorem{cor}[thm]{Corollary}
\newtheorem{prop}[thm]{Proposition}
\theoremstyle{definition}
\newtheorem{defn}[thm]{Definition}
\theoremstyle{remark}
\newtheorem{rem}[thm]{\textbf{Remark}}

 \makeatletter
    
    \@addtoreset{equation}{section}
  \makeatother

\makeatletter
      \def\@makefnmark{%
         \leavevmode
            \raise.9ex\hbox{\check@mathfonts
                \fontsize\sf@size\z@\normalfont%
                            \@thefnmark}%
       }
      \makeatother

\newcommand{\dd}{\textrm{d}}

\begin{document}

\title[]{Stability of   Hill's spherical vortex}
\author[]{K.Choi}
\date{}
\address[K. Choi]{Department of Mathematical Sciences, Ulsan National Institute of Science and Technology, 50 UNIST-gil, Ulsan, 44919, Republic of Korea}
\email{kchoi@unist.ac.kr}

\date{\today}

\maketitle

\begin{abstract}  
We study   stability of a spherical vortex  introduced by M.  Hill in 1894, which is an  explicit solution of the three-dimensional  incompressible Euler equations.   The flow is  axi-symmetric with no swirl,  the vortex core is simply a  ball sliding on the axis of symmetry with a constant speed, and the vorticity in the core is proportional to the distance from the symmetry axis. We use  the variational setting introduced by A. Friedman and B. Turkington (Trans. Amer. Math. Soc.,
1981), which produced  a maximizer of the kinetic energy under constraints on vortex strength, impulse, and circulation.  We match   the set of maximizers with the Hill's vortex via the uniqueness result of C. Amick and L. Fraenkel (Arch. Rational Mech. Anal., 1986). The matching process is done by an approximation near  exceptional points (so-called metrical boundary points) of the vortex core. As a consequence,  the  stability   up to a translation is obtained by using a concentrated compactness method.
\end{abstract}

\tableofcontents

\vspace{15pt}

\section{Introduction}


\subsection{Hill's spherical vortex: Hill (1894) } \ \\

The  three-dimensional incompressible Euler equations are written by\begin{equation}\begin{split}\label{euler_velocity}\partial_t u+( u\cdot \nabla)u+\nabla P&=0,\\\mbox{div}\, u&=0,\quad x\in\mathbb{R}^3,\quad t>0,\end{split} \end{equation} where $u(x,t)\in\mathbb{R}^3$ is the fluid velocity and $P(x,t)\in \mathbb{R}$ is the  pressure.
The Hill's spherical vortex, which was discovered 
in 1894 \cite{Hill}, 
represents an axi-symmetric flow without swirl whose compactly supported vorticity is proportional to the distance from the symmetry axis. The vortex \textit{core}, which means the support of the vorticity, is an unit ball which    slides the axis   in a constant speed forever without changing its shape or size. 
 More  precisely, 
we write the Euler equations in vorticity vector $\omega$ form 
\begin{equation}\begin{split}\label{eq_general_euler}
\partial_t \omega+( u\cdot \nabla)\omega&=( \omega\cdot\nabla ) u, \quad  
\mbox{curl}\,{u} = \, {\omega}, \quad  x\in\mathbb{R}^3,\quad t>0. 
\end{split} \end{equation} Here, the fluid velocity $u$ can be recovered from its vorticity $\omega$ via the 3d Biot-Savart law $u=\nabla\times (-\Delta)^{-1}\omega$, which makes the fluid at rest at infinity for compactly supported and bounded vorticities. When the velocity of a  flow is axi-symmetric without swirl, the vorticity admits its angular component $\omega^\theta$ only. 
By choosing the $x_3$-axis as the axis of symmetry and by setting the \textit{relative} vorticity  \footnote{In this paper, we use the terminology ``\textit{relative} vorticity", which appeared in \cite{NoSe}, even if  it  is not standard.} $\xi$ (in the cylindrical coordinates) $$\xi(r,z)=\frac {\omega^\theta(r,z)} {r}, \quad r=\sqrt{x_1^2+x_2^2},\quad z=x_3,$$ the symmetry transforms \eqref{eq_general_euler} into the active scalar equation
 \begin{equation}\begin{split}\label{3d_Euler_eq_intro}
 {\partial_t}\xi+ u\cdot \nabla\xi&=0,
 \quad  x\in\mathbb{R}^3,\quad t>0,\\  
\end{split}
\end{equation} where the   axi-symmetric velocity is determined by  
the axi-symmetric  Biot-Savart law 
 $u =\mathcal{K}[\xi]$ 
 introduced later in \eqref{form_u}. 
 In this setting, 
  the Hill's vortex $\xi_H$ is simply defined by
\begin{equation}\label{defn_hill_intro}
 {\xi_H(x)}={1}_{B }(  x),
\end{equation}
 where
$B$ is the unit ball in $\mathbb{R}^3$ centered at the origin. 
It has been well-known that
$$\xi(t,x)=\xi_H(x+t u_\infty )={1}_{B }(  x+t u_\infty )$$ is a traveling wave solution of
\eqref{3d_Euler_eq_intro} where   
$u_\infty$ is the constant velocity $$  u_\infty
=-W_H e_{x_3}, \quad W_H=(2/15).$$
It produces the unique weak solution $u(t)=\mathcal{K}[\xi(t)]$ of \eqref{euler_velocity}.   The  velocity  $u(t)$ lies on $C^{\alpha}(\mathbb{R}^3),\,0<\alpha<1$ because
the corresponding vorticity vector $\omega=(r\xi)e_\theta$ lies on $(L^1\cap L^\infty)(\mathbb{R}^3)$. 
For  more detail, we refer
to Section \ref{sec_prelim} in this paper or the original paper \cite{Hill}, the classical  textbooks \cite{Lamb}, \cite{Bat}, the modern textbook \cite{Saff}.  In fluid mechanics, it is important to study such a localized vortex moving without changing shape or size because it might help to explain  transport  of mass, momentum and energy in  large scale at a flow of  high Reynolds number.\\


 

 In this paper, we are interested in  stability of the Hill's vortex in axi-symmetric perturbations. Since 
such a traveling vortex  can be easily observed experimentally, e.g. when an ink is dropped in another fluid\cite{harper_moore_1968}, smoke is ejected from a tube \cite[p44]{VanDyke82}, or a bubble rises in a liquid \cite{LAKER2004473}, it is natural to expect (or ask questions on) its  stability in longer times. 
Expecting that 
the vortex \eqref{defn_hill_intro}     maximizes the kinetic energy among other axi-symmetric patch-type functions $\zeta$ having the same impulse condition
$$\int_{\mathbb{R}^3} r^2 \zeta dx=\int_{\mathbb{R}^3} r^2 \xi_Hdx,$$
Benjamin\cite[Section I]{Ben76}  in 1976 suggested 
variational principles   for a broad class of  steady vortex rings and  inferred their   stability up to a translation. 
Saffman also suggested in his textbook \cite[footnote in p25]{Saff}  that one can employ conservation of mass and momentum to produce  nonlinear   stability in an $L^1$ and $L^2$ norm. However,  to the best of our knowledge, there is still no   rigorous proof for such stability.
 Wan's paper  \cite{Wan88} in 1988 contains  an orbital stability statement\footnote{It was mentioned as a corollary \cite[Corollary (H)]{Wan88}  \textit{without} a written proof. The metric used in 
the statement contains a non-invariance quantity, which makes the corollary incorrect. In particular, it fails     when one simply compares two Hill's vortices with different radii. 
 The original statement of 
\cite{Wan88} 
 will be reviewed 
 in Remark \ref{rk_wan}  while
  a counterexample 
  will be presented in Remark \ref{rk_wan_pf}.}  in  patch-type axi-symmetric perturbations as a corollary. The statement 
  has most to do with our result in the sense that the variational principles used for both results 
 are suggested by \cite{Ben76} and  Friedman-Turkington  \cite{FT81}. \\
 
Most of the  other  existing literature regarding on  stability/instability issue
 focus  on linearized (or approximated) response to a patch-type  perturbation to its boundary and/or related numerical computations.
 Moffatt-Moore \cite{MM78} in 1978
 (also see \cite{Bliss}) analysed  an approximate evolution equation for the patch boundary. Roughly speaking, the perturbation can produce a thin spike from the rear stagnation point when the initial vortex is either a prolate spheroid or an oblate spheroid. It might be understood that the irrotational flow outside the sphere tends to ``sweep" the perturbation as mentioned in \cite{MM78}. The  situation is validated in a nonlinear setting by Pozrikidis \cite{pozrikidis_1986}  in 1986 numerically (also see  \cite{Pro_El} for an spectral approach). For non axi-symmetric perturbations, we refer to   \cite{Fukuya}, \cite{Rozi}.
Lastly,   investigations  by short-wavelength stability analysis  can be found  in 
\cite{Lifschitz95},\cite{RoFu}, \cite{HaHi}.\\

 Our main result (Theorem \ref{thm_hill_gen}) says that
 the vortex is nonlinearly  stable (up to a translation performed in the axis) in axi-symmetric perturbations which are allowed to be a non-patch type.  Simply speaking, 
  the amount swept by the irrotational flow outside the core  can be controlled uniformly in all time by the initial difference. The key idea is to 
 make a bridge between the existence result of  
\cite{FT81} based on variational method and  the uniqueness result of Amick-Fraenkel \cite{AF86} in order to apply the concentrated compactness method of Lions \cite{Lions84a} into a maximizing sequence.

 \subsection{Main Theorems \ref{thm_hill}, \ref{thm_hill_gen}: stability of Hill's vortex}\label{subsec_main_thm}\ \\


By using the cylindrical coordinate system $(r,\theta,z)$, 
we say that a scalar function $f:\mathbb{R}^3\to\mathbb{R}$ is axi-symmetric
if it has the form of $f(x)=f(r,z)$, 
 and  
a subset $A\subset\mathbb{R}^3$ is  axi-symmetric if the characteristic function
${1}_A:\mathbb{R}^3\to\mathbb{R}$ is axi-symmetric. Here is our main result for patch type data. 

\begin{thm}\label{thm_hill}  
The Hill's vortex 
 is  stable up to a translation in the sense that  
 for  $\varepsilon>0$,
there exists $\delta>0$ such that 
for any    axi-symmetric measurable 
 subset  $A_0\subset \mathbb{R}^3$  satisfying  
\begin{equation}\label{assump_uniq_pat}
A_0\subset \{0\leq r<R\}\quad\mbox{for some}\quad R<\infty
\end{equation} 
 and  
  \begin{align*}
 \int_{A_0\triangle B}(1+r^2)\,dx
\leq \delta,
\end{align*} the  corresponding
   solution $\xi(t)= {{1}}_{A_t}$  of \eqref{3d_Euler_eq_intro} for the initial data
   $\xi_0=1_{A_0}$
   satisfies 
 \begin{align*}\label{conclu_orb_patch}
 \inf_{\tau\in\mathbb{R}}\left\{ 
 \int_{A_t\triangle B^\tau}(1+r^2)\,dx
\right\}
\leq \varepsilon \quad \mbox{for all}\quad t\geq0,
\end{align*}  
where
 
$$B^{\tau}:=\{x\in\mathbb{R}^3\,|\, |x-\tau e_{z}|<1\}$$
is the unit ball   centered at $(0,0,\tau)$. 
Here, the symbol $ \triangle  $ means the symmetric difference. 
\ \\

\end{thm}

The above theorem   is a particular case of  the following stability theorem allowing non patch-type data: 
\begin{thm}\label{thm_hill_gen}  
 For  $\varepsilon>0$, 
there exists $\delta>0$ such that 
for 
any non-negative axi-symmetric  function
$\xi_0$ satisfying
\begin{equation}\label{assump_uniq}
\xi_0, \, r\xi_0\in L^\infty(\mathbb{R}^3)
\end{equation} and
\begin{align*}
\|\xi_0-\xi_H\|_{L^1\cap L^2(\mathbb{R}^3)}
+\|r^2(\xi_0-\xi_H)\|_{L^1(\mathbb{R}^3)} 
\leq \delta,
\end{align*} the corresponding
 solution $\xi(t)$ of \eqref{3d_Euler_eq_intro} 
 for the initial data $ \xi_0$ 
satisfies
 \begin{align*}\label{conclu_orb_patch_gen}
 \inf_{\tau\in\mathbb{R}}\left\{ 
\|\xi(\cdot+\tau e_{z},t)-\xi_H\|_{L^1\cap L^2(\mathbb{R}^3)}
+\|r^2(\xi(\cdot+\tau e_{z},t)-\xi_H)\|_{L^1(\mathbb{R}^3)}  
\right\}
\leq \varepsilon \quad \mbox{for all}\quad t\geq0.
\end{align*} 
Here, $\|\cdot\|_{L^1\cap L^2}$ means $\|\cdot\|_{L^1}+\|\cdot\|_{L^2}$.\\

\end{thm} 
\begin{rem} \label{rem_uniq_sol} 
In general, the Euler equations in velocity form \eqref{euler_velocity} admits non-unique weak solutions for  initial data $u_0\in L^2(\mathbb{R}^3)$  by \cite{MR2564474},  \cite{MR2838398}. However, when the flow is axi-symmetric without swirl,  we can consider the 
simpler equation 
\eqref{3d_Euler_eq_intro}
 instead. Then for any
    axi-symmetric initial data 
\begin{equation*}
0\leq \xi_0\in 
(L^1\cap L^2)(\mathbb{R}^3)\quad\mbox{with}\quad
r^2\xi_0\in   L^1(\mathbb{R}^3),
\end{equation*}
  existence  and   uniqueness of a weak solution  is guaranteed by  imposing    the  extra condition \eqref{assump_uniq} on the relative vorticity $\xi_0$ by  Ukhovskii-Yudovich \cite{UI} (also see \cite{Raymond}, \cite{Danchin}).    By the same reason, the assumption \eqref{assump_uniq_pat} is added in Theorem \ref{thm_hill}. We revisit the issue  in detail in Subsection 
\ref{subsec_exist_weak_sol} (see  Lemma \ref{lem_exist_weak_sol} and  Remark \ref{rem_exist_weak}).
\end{rem}

\begin{rem} \label{c_infty_smooth} 
Theorem \ref{thm_hill_gen} deals with   non-patch type solutions near the Hill's vortex.
It gives some advantage in the following sense:
For any $\delta>0$, there is a   $C^\infty$-smooth initial compactly supported axi-symmetric relative vorticity $\xi_0$ satisfying  the assumption of Theorem \ref{thm_hill_gen} whose compact support lies away from the axis $\{r=0\}$. 
 Then, we have the  initial vorticity $\omega_0(x)=r\xi_0(r,z)e_\theta(\theta)$ lying on $C_c^\infty(\mathbb{R}^3)$
   which gives the global-in-time $C^\infty$  solution $u(t)$   of \eqref{euler_velocity} by \cite{SY}, \cite[Theorem 2.4]{Raymond}
   since the axi-symmetric initial velocity $u_0:=\nabla\times(-\Delta)^{-1}\omega$ has no swirl and  lies on $H^m(\mathbb{R}^3)$ for any integer $m>0$.    
  As a result, the smooth solution $u(t)$ is close in our sense (up to a translation) to the flow of the Hill's vortex for all time.
\end{rem}

\begin{rem}\label{rem_hill_scaling}
By the scaling invariance of the Euler equations, we have a   family of Hill's vortices of two parameters in the following sense:\\
For any given $0<\lambda,\,a<\infty$, we define   $\xi_{H(\lambda,a)}$ by
\begin{equation}\label{defn_hill_gen_intro}
 \xi_{H(\lambda,a)}(x)=\lambda{1}_{B_a}(x)
\end{equation}
 where $B_a:=\{x\in\mathbb{R}^3\,|\, |x|< a\}$.
Then 
\begin{equation}\label{hill_intro_scaling}
\xi(t,x)=\lambda {1}_{B_a }(  x -tW_{H(\lambda, a)}e_{x_3} )
\end{equation}
 is a traveling wave solution of
\eqref{3d_Euler_eq_intro}  with its traveling speed
\begin{equation}\label{speed_intro_gen}
  W_{H(\lambda, a)}=W_H\cdot ({\lambda a^2}) =\frac{2}{15}{\lambda a^2}
\end{equation}  (see Subsection \ref{subsec_def_hill} for more detail).
The above theorems work for each fixed $\lambda,a\in(0,\infty)$ by the scaling.
 \end{rem}
 
 \begin{rem}\label{rem_negative}
 Dropping the non-negativity assumption on the initial relative vorticity $\xi_0$ in Theorem \ref{thm_hill_gen} seems   non-trivial in the sense that we do not exclude a possibility that  small negative part of $\xi_0$ might spoil the distribution of  other positive part much at time infinity. Indeed, 
our variational method  uses the  fluid  impulse $\int_{\mathbb{R}^3}r^2\xi(t,x)\,dx$ conserved  in time as the main reference quantity. Together with the non-negativity on $\xi$, the conservation of impulse  plays a role of attraction or cohesion toward the symmetry axis $\{r=0\}$. However, without assuming the sign condition,
 we do not expect any global-in-time bound  on 
$\int_{\mathbb{R}^3}r^2\xi^+(t,x)\,dx$ and $\int_{\mathbb{R}^3}r^2\xi^-(t,x)\,dx$. There  might be continued leakage of positive  part and negative part of $\xi$   from the core of the Hill's vortex. On the other hand, dropping the axis-symmetry assumption is more challenging. We do not even know global existence of solutions without the symmetry.

\end{rem}
  \begin{rem}\label{rk_wan}
  The   paper \cite{Wan88} proved that the Hill's vortex  $\xi_H$ is a nondegenerate local maximum of the kinetic energy   under certain constraints. It    mentioned an orbital stability of the   vortex as  a corollary \cite[Corollary (H)]{Wan88} without a written  proof. The metric used in the corollary has the form 
\begin{equation*}
 d(\xi_1,\xi_2)=\int  r^2 |\xi_1-\xi_2| dx+\left|\int z r^2 \xi_1  dx-\int z r^2 \xi_2  dx\right|. 
\end{equation*}   We note that the term
$\int z r^2 \xi(t,x)  dx$   is not conserved in general   when $\xi(t)$ is a solution of \eqref{3d_Euler_eq_intro} while the impulse
$\int  r^2 \xi(t,x) dx$ is preserved. The   statement   \cite[Corollary (H)]{Wan88} 
  considers   patch-type initial data and 
says    
    that for $\varepsilon>0$, there is $\delta>0$ such that
  if $A_0$ is an axi-symmetric bounded subset of $\mathbb{R}^3$ satisfying
   \begin{align*}
 \int_{A_0\triangle B}r^2 dx+\left|\int_{A_0 } zr^2 dx -\int_{B } zr^2dx \right|
\leq \delta,
\end{align*} 
then the  corresponding 
    solution $\xi(t)= {{1}}_{A_t}$ of \eqref{3d_Euler_eq_intro} for the initial data $\xi_0=1_{A_0}$ 
    satisfies 
 \begin{align}\label{conclu_orb_patch_remark}
\inf_{\tau\in\mathbb{R}}\left( \int_{A_t\triangle B^\tau}r^2dx+\left|\int_{A_t } zr^2 dx -\int_{B^\tau } zr^2dx \right|\right)
\leq \varepsilon \quad \mbox{for all}\quad t\geq0.
\end{align} 
However, when $A_0$ has a different impulse, i.e. when $$\int_{A_0} r^2 dx\neq   \int_{B} r^2 dx,$$
the statement fails in general. 
The precise verification is postponed until Remark \ref{rk_wan_pf}.
Heuristically, 
the impulse part  $ \int_{A_t\triangle B^\tau}r^2dx$ is minimized when two sets $A_t$ and $B^\tau$ share the same center.  However, in the case,  the non-invariance part $\left|\int_{A_t } zr^2 dx -\int_{B^\tau } zr^2dx \right|$ can grow linearly in time due to the weight $z$ in the integrand (see \eqref{linear_sp}).
It shows that the quantity in \eqref{conclu_orb_patch_remark} cannot be  
 small for \textit{large} time. 
  \end{rem}
\begin{rem}\label{filamentation}
It  is natural to ask where the core of the perturbed solution  should be at each time. Indeed, the current result merely says that the perturbed one is  close to the Hill's vortex up to \textit{some} time-dependent $z-$translation $\tau(t)$. In this direction, one can estimate the core position $\tau(\cdot)$ by a bootstrap argument  (see \cite{CJ_hill}). Roughly speaking, the perturbed solution travels in a similar speed of the Hill's vortex.
\end{rem}

\subsection{Ideas of proof}\ \\


The  
stability up to   a \textit{translation} (or an \textit{orbital} stability in general)  is a proper notion of nonlinear stability of the Hill's vortices. For instance, imagine the vortex \eqref{hill_intro_scaling}  for $\lambda= 1$ and $  0<|a-1|\ll 1$. The core travels with the speed \eqref{speed_intro_gen}  and eventually becomes disjoint  from  the other travelled core of the Hill's vortex for $\lambda=1,  a=1$. In that sense, a translation is necessary when comparing sliding vortices in longer times. We also refer to the explanation  \cite[p605]{Const_survey} for the case of Kirchhoff ellipses   asking stability up to a rotation.\\

To obtain such a stability, we follow the strategy via the variational method based on vorticity due to the idea of Kelvin \cite{kelvin1880} and Arnold \cite{Arnold66} (also see the book \cite{AK98}). The variational setting for vortex rings  we use comes from  \cite{FT81}. More specifically, for given $\lambda,\mu,\nu\in(0,\infty)$, we consider an  admissible function $\xi$ which is a characteristic function of strength $\lambda$
$$\xi=\lambda 1_A$$ for some axi-symmetric $A\subset\mathbb{R}^3$
(i.e.  a patch-type data) and whose impulse satisfies the exact condition
$$\frac 1 2\int_{\mathbb{R}^3}r^2\xi\,dx=\mu$$ while
 its total circulation(or its mass) is asked to hold the less strict condition $$\int_{\mathbb{R}^3}\xi\,dx\leq \nu.$$
  (see \eqref{defn_adm}) 
 ({cf}. for vortex pairs, see \cite{Tu83}, \cite{Burton88}, \cite{AC2019}). In this class of admissible functions, we pursue maximizing the kinetic energy $$\int_{\mathbb{R}^3} |u|^2 dx,$$ where $u=\mathcal{K}[\xi]$ is the corresponding velocity field of the relative vorticity $\xi$. The approach prescribing impulse was first suggested by
  \cite{Ben76} (also see Burton \cite{Burton_03}).
  \\ 

Under this setting, Friedman and Turkington in their paper \cite[Theorem 2.1]{FT81}   showed that there exists a maximizer $
\xi$ of the energy 
 satisfying 
\begin{equation}\label{eq_W_intro}
\begin{aligned}
&\xi=\lambda f_H\left(\Psi\right),\quad \Psi:=\mathcal{G}[\xi]-\frac{1}{2}Wr^2-\gamma,  
\end{aligned}
\end{equation}  for some 
$W>0$ and $\gamma\geq 0$, where the vorticity function $f_H$ is defined by
\begin{equation}\label{defn_hill_vor_fct}
f_H(s)=\begin{cases} &1,\quad s>0\\ &0,\quad s\leq 0,\end{cases}
\end{equation} and 
  the stream function $\mathcal{G}[\xi]$ is defined later in 
\eqref{form_psi}. The constants $W, \gamma$ represent
 the propagation speed and the flux constant, respectively.
The existence was obtained by a limiting argument ($\beta\to 0$) via a penalized energy functional 
\begin{equation}\label{defn_pen_en}
E_\beta[\xi]:=
\mbox{(kinetic energy of the fluid induced by $\xi$)}
-\beta\lambda\int \left(\frac{\xi}{\lambda}\right)^{1+1/\beta} dx,
\end{equation} where the definition of kinetic energy can be found in  
\eqref{defn_en}
(also see Remark \ref{rem_ft81}).\\

 On the other hand, 
Amick and Fraenkel in their paper \cite[Theorem 1.1]{AF86} showed that any $\xi$ satisfying \eqref{eq_W_intro} with $\gamma=0$ is  the Hill's vortex 
 \eqref{defn_hill_gen_intro} with certain radius $a>0$ (see \eqref{def_a})  up to a translation in $z-$variable. This uniqueness was proved by adapting the moving plane method (refer to\cite{Serrin71},
\cite{GNN}). 
The main difficulty comes from  discontinuity of the vorticity function $f_H$ in \eqref{defn_hill_vor_fct} (also see Remark \ref{rem_idea_uniqAF}).
\\

Before explaining the key ideas of our proof,  we  may assume $\lambda=\nu=1$
without loss of generality. This is done by the scaling argument (e.g. see \eqref{scaling}). Now the impulse $\mu\in(0,\infty)$ is the only free parameter. To obtain stability of the Hill's vortex whose impulse is exactly the constant $\mu$, we connect those two classical results \cite{FT81} and \cite{AF86}  mentioned above by showing the following two statements:\\

\noindent 1. 
For any impulse $\mu>0$, \textit{every} maximizer satisfies \eqref{eq_W_intro} for some speed $W>0$ and some flux constant $\gamma\geq 0$ (see Theorem \ref{thm_max_is_ring}). It implies that every maximizer is a steady vortex ring related to the vorticity function $f_H$ \eqref{defn_hill_vor_fct}. \\
2. 
For any \textit{small} impulse $\mu>0$, every maximizer is the corresponding
Hill's vortex \eqref{defn_hill_gen_intro}   up to a translation (Theorem \ref{thm_uniq}).\\

When proving 
statement 1 above
(Theorem \ref{thm_max_is_ring}), 
the discontinuity of $f_H$ is the main obstacle. Indeed, for a maximizer $\xi_{max}=1_A$ with some axi-symmetric $A\subset\mathbb{R}^3$,   by using variational principles, 
we     have to find some constants $W,\gamma$ satisfying
\begin{equation}\label{intro_claim}
A=\{\Psi>0\},
\end{equation} where the adjusted stream function $\Psi$ is defined in   \eqref{eq_W_intro}.
We note that for any choice of constants $W, \gamma$, the set
$\{\Psi>0\}$ is open so that  its  boundary $\partial \{\Psi>0\}$ is well-defined.
Once the claim \eqref{intro_claim} is proved, such constants $W,\gamma$ are
uniquely and explicitly determined 
by any two points on the boundary $\partial A$ because
$\Psi=0$  there  (see \eqref{def_W_gamma_uniq}).\\

However, 
we merely know that $A$ is measurable (i.e. $A$ is defined up to \textit{almost everywhere})  until proving the claim \eqref{intro_claim}
 for some $W, \gamma$. It makes {topological} boundary $\partial A$ unusable. Instead,   we
 use the   \textit{metrical}\footnote{\label{footlabel1}
 These terminologies ``\textit{metrical}"  and ``\textit{exceptional}" can be found in  \cite[p78]{Croft}  and \cite[p765]{Szenes}, respectively (also see \cite{Kol}).} boundary (see Definition \ref{defn_excep}) of the measurable set $A$, whose element is called  an \textit{exceptional} 
 point. Near an exceptional point of $A$,
there are non-trivial parts of $A$ and $A^c$ in the neighborhood (See  Lemma \ref{lem_exist_one}) so that  
  we can construct an approximation to the Dirac  function at the point from the inside of $A$ and from the outside of $A$ (see Remark \ref{rem_intui} for more detail). It is the key idea to solve the former statement (Theorem \ref{thm_max_is_ring}). 
As a minor difficulty, the restriction $$\xi(x)\in\{0,1\}$$ 
on 
admissible functions $\xi$
 prevents us from adding general $L^\infty$-perturbations into the maximizer  $\xi_{max}$.
 To overcome, we  
extend  the class of admissible functions in advance so that non patch-type functions
$$0\leq\xi(x)\leq 1$$ are allowed (see \eqref{defn_prime_class}).
Then we can freely  add small negative perturbations supported on $A$ and small positive perturbations supported on the complement $A^c$  into   $\xi_{max}$ .\\

The statement 2 above (Theorem \ref{thm_uniq}) is decomposed into three steps:\\ 
\begin{equation*}
\mu\ll 1 \quad\Rightarrow_1\quad \int \xi_{max}\, dx<1\quad\Rightarrow_2\quad \gamma=0 \quad \Rightarrow_3\quad \xi_{max}=\xi_{H(1,a)}\quad \mbox{up to a translation,}
\end{equation*} where the radius $a>0$ satisfies $\mu= (4/15)\pi   a^5  $.
To prove the first step (Proposition \ref{prop_small_mu}),
we use the lower bound  of the traveling speed $W$ depending only on the impulse $\mu$ (Proposition \ref{prop_cpt_supp}) (also see \cite{FT81} or see \cite[p42]{FB74}). Together with an elementary estimate of stream functions (Lemma \ref{lem_est_psi_axis}), it says that the total mass should be strictly smaller than the prescribed number $\nu=1$.  The second step is done 
at Lemma \ref{lem_pos_gam} by using variational principles. As noted in 
Remark \ref{rem_small_mu}, 
this strategy is essentially contained in \cite[Remark 5.2]{FT81} (\text{cf}. for vortex pairs, see \cite[Remark 2.6]{AC2019}).  For the third step, the uniqueness  result due to\cite[Theorem 1.1]{AF86} is used. In Section \ref{sec_uniq_proof_hill}, we carefully verify the setting and all the assumptions 
of \cite[Theorem 1.1]{AF86}. 
In sum,  we make a bridge between the existence result
\cite{FT81} and the uniqueness result  \cite{AF86}, which gives 
 the characterization of the set of maximizers as a single orbit of the corresponding Hill's vortex (Theorem \ref{thm_uniq}).\\
 
  The next part is to prove  compactness  (Theorem \ref{thm_cpt}), which means that
there is a limiting function (up to a subsequence) for any maximizing sequence, and the limit is a maximizer.  
   As already used in the case of 2d vortex pairs by 
Burton-Nussenzveig Lopes-Lopes Filho
  \cite{BNL13} (and more recently by \cite{AC2019}), we
exploit the concentrated compactness lemma  of Lions \cite[Lemma I.1]{Lions84a}. 
Roughly speaking, the lemma says that for any sequence $\{f_n\}$ of non-negative functions on the whole space, if each function of the sequence has the unit mass, then  (up to a subsequence) one of the three following cases happens\footnote{For the precise statements,  see Lemma \ref{lem_concent}. }: \begin{enumerate}
\item (compactness) The most mass is confined in a large ball up to a translation. 
\item (vanishing) For any large ball, $f_n$ has negligible mass on the ball as $n$ goes to infinity.
\item (dichotomy) There is a real number $\mu$ between 0 and 1 such that each $f_n$ is decomposed  into two functions 
 having disjoint supports, and the supports are getting farther and farther away from each other. In particular, both the supports has non-trivial mass close to $\mu$ and $(1-\mu)$, respectively.
\end{enumerate}
In this context, our goal is to exclude both  the case of vanishing and the case of dichotomy.  There is a  difficulty when avoiding the dichotomy case since sub-additivity of maximum energy   in $\mu>0$  is not known.  In particular, since we are not allowed to use  $L^\infty$-bound of the initial data $\xi_0$, we carefully use the particular form of the  two functions $\xi_i=\xi 1_{\Omega_i}$ from dichotomy (e.g. see \eqref{dic_key}). It is important that the function class in the assumptions need to be large enough to cover our stability statement.
As a consequence,
the proof (Section \ref{sec_cpt}) gets a bit technical.    \\

Lastly, the orbital stability is obtained 
 once we  combine  compactness (Theorem \ref{thm_cpt}) together with   uniqueness (Theorem \ref{thm_uniq}).  Indeed, we can use
the classical contradiction argument. 
  Loosely speaking, let's suppose that there is a sequence of   solutions obtained from perturbed  initial data, but at a certain positive time, it  does not converge to the Hill's vortex up to any translation. Then by the energy conservation, we can use the compactness result to obtain a limit function which is a maximizer. Then by the uniqueness result, the limit should be the Hill's vortex up to a translation, which is a contradiction.

 \begin{rem}
If one can extend the uniqueness result
 (Theorem \ref{thm_uniq})
 beyond the class of Hill's vortices 
  (i.e. $\gamma>0$ case in \eqref{eq_W_intro}),
then we can follow the same line of our
proof in Subsection \ref{subsec_proof} of Theorem \ref{thm_hill_gen} to establish a similar stability for the vortex rings in certain (probably smaller) function class. However, 
  such a uniqueness statement  seems out of reach under the current technique (e.g. moving plane method \cite{GNN}). For instance, while one may prove that the core set $A$
   is connected and symmetric (about the plane $\{z=c\}$ for some $c\in\mathbb{R}$) 
    by following the  approach
from  \cite{Ben76} (also see \cite{CF80}) and
from \cite{BNL13} respectively,
 it is  not strong enough to guarantee the uniqueness. 
\end{rem} 

\subsection{Other vortices and waves}\ \\

The Hill's vortex can be considered as the limiting fattest case of the one parameter family 
from Fraenkel \cite{Fraenkel70} and Norbury  \cite{Norbury72}. 
Except the limiting case, Fraenkel-Norbury's solutions are genuine  \textit{rings} in the sense that their cores are away from  the symmetry axis forming torus-type vortices. The Helmholtz's rings of small cross-section \cite{Helm1858} are related to the opposite limiting thinnest case. We refer to \cite{FB74} for the existence of general vortex rings via   a stream function method (also see   \cite{Ni80}
 and references therein). \\
 
It is interesting  that there are explicit spherical vortices  \textit{with} swirl traveling in a constant speed (e.g. see \cite{Hicks}, \cite{Moffatt}). We refer to the recent preprint  \cite{Abe_swirl} and references therein for diverse vortices with swirl. Their stability are generally open.\\

 As a 2d analogue to Hill's spherical vortex, there is an explicit traveling vortex pair (or ``dipole") introduced  by 
 H. Lamb \cite[p231]{Lamb} in 1906 and,  independently, by S. A. Chaplygin in 1903 \cite{Chap1903}, \cite{Chap07} 
  (cf. \cite{MV94}).
For a simple presentation, 
let us consider the case when the vortex core is the unit disk in $\mathbb{R}^2$ and the traveling speed is exactly equal to 1. Then the vortex   $\omega_{L}$ 
has the form
 in the polar coordinates $(r,\theta)$:
\begin{equation}\label{lamb_dipole}
\omega_{L}(x)=\left\{
\begin{aligned}
\left(-\frac{2c_0}{J_0(c_0)}\right) J_1(c_0r)\sin\theta,\quad r\leq 1,\\
0,\quad r>1,
\end{aligned}
\right. 
\end{equation} where   $J_{m}(r)$ is the $m$-th order Bessel function of the first kind. The constant $c_0>0$ is the first zero point of $J_1$ and  $J_0(c_0)<0$.  In particular, it satisfies
\begin{equation*}\label{lamb_dipole_vor}
\omega_{L}=c_0^2 f_L(\psi_L-W_L x_2)\quad\mbox{in}\quad\mathbb{R}^2_+,\quad \psi_L=(-\Delta_{\mathbb{R}^2})^{-1}\omega_L, \quad W_L=1,
\end{equation*} where the vorticity function $f_L$ is defined by 
\begin{equation}\label{defn_lamb_vor_fct}
f_L(s)=\begin{cases} &s,\quad s>0\\ &0,\quad s\leq 0.\end{cases}
\end{equation}
Then 
$$\omega(t,x)=\omega_L(x-W_Lte_{x_2})$$ is a traveling wave solution of the 2d Euler equations
\begin{equation*}
\begin{aligned}
\partial_t \omega+u\cdot \nabla \omega=0,\quad  u&=K*\omega \quad \textrm{in}\quad \mathbb{R}^{2}\times (0,\infty),\\
\end{aligned}
 \end{equation*}
with the 2d Biot-Savart kernel $K(x)=(2\pi)^{-1}x^{\perp}|x|^{-2}$, $x^{\perp}=
(-x_2,x_1)$. 
 We note that 
the vorticity function  $f_L$  in  \eqref{defn_lamb_vor_fct}
 is smoother than $f_H$  in  \eqref{defn_hill_vor_fct} of the Hill's vortex.  The fact that  $f_L$  is Lipschitz   helps to study  certain properties of the dipole including the stability question. The orbital stability of the dipole \eqref{lamb_dipole_vor} was recently obtained by   \cite{AC2019}. We mention Burton's work
\cite{Burton96},  \cite{Burton05b} for other properties of the dipole.
 A similar stability for broader class of vortex pairs (not including the particular case \eqref{lamb_dipole}) was proved by \cite{BNL13}. In this paper, we    follow a similar structure of \cite{BNL13}, \cite{AC2019}   to obtain the stability for the Hill's vortex. For more properties of general vortex pairs,   we refer to \cite{Norbury75}, \cite{Tu83}, \cite{Burton88}. 
 For general dimension $N\geq2$, existence (and uniqueness) of a  vortex generalizing \eqref{lamb_dipole} and \eqref{defn_hill_intro} was proved by Burton-Preciso \cite{Burton_Preciso}. 
 \\
 
There are stability results 
 for   other exact solutions of the 2d Euler equations. For instance, we  see \cite{WP85}, \cite{SV09}, \cite{Cao19} for a circular   patch,  \cite{Denisov17}  for a rectangular patch in the infinite strip.  As   simple applications of stability for a circular patch, we mention \cite{CJ_winding}, \cite{Choi_distance} for   winding number estimates and \cite{CJ_perimeter} for   growth in perimeter. Surprisingly, there is an asymptotic stability result \cite{BesMas} for a 2d Couette flow (cf. 
 \cite{BGM} even for 3d Couette flow).\\
  
  An orbital stability of solitary waves  appeared first in \cite{Ben72}, \cite{Bona75} for the   generalized  KdV equation. We also refer to \cite{MartelMerle01}, \cite{Merle01}, \cite{MartelMerle02} for   instability and   blow-ups  (also see the survey paper \cite{Tao09}). For general dispersive equations, we refer to \cite{CL82}, \cite{Bona87}, \cite{GSS}.  \\
  
  Such a stability up to a translation  even occurs for inviscid/viscous shocks in   conservation laws. 
    In $L^2$-setting in one dimensional space, we mention \cite{Leger}, \cite{ChoiV} for the scalar case,  \cite{LegerV}, \cite{Kang-V-NS17}, \cite{Kang-V-NS20}  for   systems  such as
  compressible Euler/Navier-Stokes systems  
 and \cite{ckkv2019}, \cite{ckv2020} for certain Keller-Segel type systems. We also  refer to the classical paper \cite {MR0101396} for asymptotic stability up to a translation  in $L^1$-setting (also see \cite{MR2217605} and references therein).  \\

 In the rest of the paper, in Section \ref{sec_prelim}, we introduce preliminary materials. Then, Theorems \ref{thm_cpt}, \ref{thm_uniq} (compactness, uniqueness) are presented in Subsection \ref{subsec_cpt_uniq}. By assuming them for a moment, the proof of our main stability result (Theorems \ref{thm_hill}, \ref{thm_hill_gen}) is given in Subsection \ref{subsec_proof}. The rest of the paper (Sections \ref{sec_exist}-\ref{sec_uniq_proof_hill})  is devoted to prove Theorems \ref{thm_cpt}, \ref{thm_uniq}.

 \color{black}
\section{Preliminaries}\label{sec_prelim}

We introduce relevant  mathematical background of the Hill's vortex (more generally  steady vortex rings). 
For more detail, we  also refer to \cite[Section 2]{FB74}, \cite[Section 1]{AF86}, \cite[Section 1]{FT81}. 


\subsection{Axi-symmetric Biot-Savart law} \label{subsec_axi-sym}\ \\

A  vector field $ u
$ is  called \textit{axi-symmetric} 
if 
it has the form of   $$ u(x)=u^r(r,z)e_r(\theta) +u^\theta(r,z)e_\theta(\theta)+u^z(r,z)e_z,
$$ for $$e_r(\theta)=(\cos\theta,\sin\theta,0),
 \quad e_\theta(\theta)=(-\sin \theta, \cos \theta,0),
\quad e_z=(0,0,1),$$ 
where $(r,\theta,z)$ is the cylindrical coordinate to the Cartesian coordinate $x=(x_1,x_2,x_3)$, i.e. $x_1=r\cos\theta,\, x_2=r\sin\theta,\,x_3=z.$
If $u^\theta\equiv0$,   then   $ u$ is called axi-symmetric \textit{without} swirl.
The divergence-free condition for $ u$ can be written as $$ 
  \partial_r(ru^r)+\partial_z(ru^z)=0.$$ 
Then, 
there exists 
an axi-symmetric stream function  $ \psi=\psi(r,z)$ such that  
$$ u =\nabla \times\phi,\quad \phi= \left(\frac \psi r e_\theta\right).$$
By denoting the vorticity vector field 
 $${\omega} :=\nabla\times  u=
 (\partial_z u^r-\partial_r u^z)e_\theta(\theta)=  \omega^\theta e_\theta(\theta),
 $$ 
the stream $\psi$ satisfies 
 $$ -\frac 1 {r^2} {\mathcal{L}\psi}=\xi
 $$  
for the relative vorticity $\xi:=r^{-1}{\omega^\theta}$
  with   the operator 
\begin{equation*}\label{defn_L}
\mathcal{L}=
r\frac{\partial}{\partial r}\left(\frac{1}{r}\frac{\partial}{\partial r}\right)+\frac{\partial^2}{\partial z^2}=
\frac{\partial^2}{\partial r^2}-\frac{1}{r}\frac{\partial}{\partial r}+ \frac{\partial^2}{\partial z^2}. 
\end{equation*}
From $-\Delta \phi=\omega^\theta e_\theta,$  we may assume   
$$
\phi
=\frac{1}{4\pi|x|}*_{\mathbb{R}^3}(\omega^\theta e_\theta).
$$
Then, by the axial symmetry (e.g. see \cite[Section 1]{FT81}), we have
\begin{equation}\label{form_psi} \psi(r,z)=
\int_\Pi G(r,z,r',z')\xi(r',z')r'dr'dz'
=:\mathcal{G}[\xi](r,z),\quad(r,z)\in \Pi,\\
\end{equation} for the the half space $$ \Pi=  \{(r,z)\in\mathbb{R}^2\,|\, \,r>0\}$$ and 
for the Green function \begin{equation*}
G(x,y)=
G(r,z,r',z')=\frac{rr'}{2\pi}\int_{0}^\pi\frac{\cos\vartheta}{\sqrt{
r^2+r'^2-2rr'\cos\vartheta+(z-z')^2}}d\vartheta,\end{equation*}  where $y_1^2+y_2^2=r'^2, \, y_3=z'$ for $y=(y_1,y_2,y_3)\in\mathbb{R}^3$. We note that $G(x,y)=G(y,x)$  and $G$ is  axi-symmetric for each variable. 
  Following \cite{Sverak_lecture}, we write 
\begin{equation}\label{form_F}
  G(r,z,r',z') =\frac{\sqrt{rr'}}{2\pi}\,F(s)
\end{equation}
by setting 
$$s=\frac{(r-r')^2+(z-z')^2}{rr'},\quad F(s)=\int_0^\pi\frac{\cos\vartheta}{\sqrt{2(1-\cos\vartheta)+s}}d\vartheta.$$
We observe that the function $F>0$ is strictly decreasing in $s>0$. 
The  function $F$ (and $G$) can be estimated by using  
  complete elliptic integrals of the first and second kind (e.g. see \cite[Lemma 3.3]{FT81}, \cite[Section 19]{Sverak_lecture}). Indeed, by the asymptotic behavior which can be found in \cite[Lemma 2.1, Remark 2.2]{GV2015}, \cite[Lemmas 2.7,  2.8]{FengSverak}, \cite[Section 2.3]{Do}  we have 
\begin{equation}\label{log_beha}
 F(s)=\frac{1}{2}\log \frac 1 s +\log 8-2+{O}\left(s\log \frac  1 s\right)\quad\mbox{as}\quad s\to 0 
\end{equation}
and
$$F(s)=\frac{\pi}{2}\frac{1}{s^{3/2}}+{O}(s^{-5/2})\quad\mbox{as}\quad  s\to\infty.$$ 
Thus we have 
\begin{equation*}
F(s)\lesssim_\tau   \frac 1 {s^\tau}, \quad 0<\tau\leq 3/2,\quad s>0
\end{equation*}  and
\begin{equation}\label{est_F}
G(r,z,r',z')\lesssim_\tau  
\frac{(rr')^{\tau+\frac 1 2}}{\left(\sqrt{|r-r'|^2+|z-z'|^2}\right)^{2\tau}},
 \quad 0<\tau\leq 3/2,\quad  (r,z),\,(r',z')\in\Pi.
\end{equation}
We set the Biot-Savart law $$ \mathcal{K}[\xi]:=\nabla\times\left(\frac 1 r \mathcal{G}[\xi]e_\theta\right)=\nabla\times\left(\frac 1 r \psi e_\theta\right).$$ Then we recover the axi-symmetric velocity
\begin{equation}\label{form_u}
 u=\mathcal{K}[\xi]=\frac{-\partial_z\psi}{r} e_r+\frac{\partial_r\psi}{r}e_z.
\end{equation}
 
In Subsection \ref{subsec_est_stream}, 
we present  a natural class for $\xi$ where the above formal computations can be valid. 

\subsection{Vorticity equation}\label{subsec_axi_euler}\ \\
 

Once we assume that the flow $u$ in \eqref{eq_general_euler} is axi-symmetric without swirl, 
 we can derive  the following active scalar equation   for  the relative vorticity $\xi=\omega^\theta/r$:  
\begin{equation}\begin{split}\label{3d_Euler_eq}
 {\partial_t}\xi+ u\cdot \nabla\xi&=0,\quad  u =\mathcal{K}[\xi], \quad  x\in\mathbb{R}^3,\quad t>0,\\ 
\xi|_{t=0}&=\xi_0,\quad x\in\mathbb{R}^3. 
\end{split}
\end{equation} 
In this paper, we consider  only non-negative axi-symmetric weak solutions $\xi(t)$ 
to \eqref{3d_Euler_eq}
preserving the 
quantities (1)-(4) below (by noting $dx=2\pi rdrdz$):\\

\begin{enumerate}
\item     Impulse
 $$P[\xi] =\frac{1}{2}\int_{\mathbb{R}^3}r^2\xi  dx= \pi\int_\Pi r^3\xi drdz.$$
 \item Kinetic energy
 \begin{equation}\label{defn_en}\begin{split}
E[\xi]&=
\frac{1}{2}\int_{\mathbb{R}^3} \xi \mathcal{G}[\xi] dx
= \pi\iint_{\Pi\times\Pi}  G(r,z,r',z')\xi(r',z')\xi(r,z)rr'dr'dz' drdz.
\end{split}
\end{equation} 
 \item    Circulation (or total mass)
  $$ \Gamma[\xi]=\int_{\mathbb{R}^3} \xi dx=2\pi\int_\Pi r\xi drdz.$$
  \item Vortex strength
  $$\Lambda[\xi]=\mbox{ess\,sup}_{x\in\mathbb{R}^3}\xi(x).$$
\end{enumerate}
For the existence of such weak solutions, we see Lemma \ref{lem_exist_weak_sol} in Subsection \ref{subsec_exist_weak_sol}. 
In addition, we  have   conservation of any
 $L^p-$norm   $$\|\xi\|_{L^p(\mathbb{R}^3)},\quad p\in[1,\infty]$$ and
 conservation of  mass contained in any  level sets
$$ \int_{\{x\in\mathbb{R}^3\,|\,a<\xi<b\}}\xi \,dx,\quad 0<a<b<\infty.$$

\subsection{Steady vortex rings}\label{subsec_vortex_ring}\ \\


 We call an axi-symmetric motion without   swirl with vanishing velocity at infinity a \textit{steady vortex ring} if the support of the vorticity is a bounded set (so-called the vortex \textit{core}), and if the vortex  moves  at a constant speed along the axis without any change in its size and shape.
In other words, we are interested in 
axi-symmetric solutions $\xi$ 
to \eqref{3d_Euler_eq}
 of the form
\begin{equation*}\label{eq_vortex_ring}
\begin{aligned}
\xi(x,t)=\xi_{vr}(x+t u_\infty )
\end{aligned}
\end{equation*} for some compactly supported $\xi_{vr}$
and
for some constant velocity
$u_\infty=-We_z$. 
We set 
$$  U =\mathcal{K}[\xi_{vr}]+u_\infty$$ 
 so that the pair  $(\xi_{vr},
 U)$
satisfies \color{black} the   stationary equation 
\begin{equation*}\begin{split}
& U\cdot\nabla \xi_{vr} =0,\quad x\in\mathbb{R}^3,\\
& U\to u_\infty \quad \mbox{as}\quad |x|\to \infty.
\end{split}
\end{equation*}
By denoting $ \psi_{vr}=\mathcal{G}[\xi_{vr}]$ and by using \eqref{form_u}, we get 
$$ U
=\frac{-\partial_z\psi_{vr}}{r} e_r+\left(\frac{\partial_r\psi_{vr}}{r}-W\right)e_z.
 $$
Thus the above stationary equation
 can be written as  
\begin{equation*}
\begin{split}
&\frac{\partial(\xi_{vr},\Psi_{vr})}{\partial(r,z)}=
\frac{\partial \xi_{vr} }{\partial z}\frac{\partial \Psi_{vr} }{\partial r}
-
\frac{\partial \xi_{vr} }{\partial r}\frac{\partial \Psi_{vr} }{\partial z}
=0
\end{split}
\end{equation*} for   the adjusted stream function 
\begin{equation}\label{defn_adj_str}
\Psi_{vr}=\psi_{vr}-\frac{1}{2}Wr^2-\gamma
\end{equation}
 for any flux constant $\gamma$.
In this sense, we shall seek a (time-independent) solution  $\xi$ to (by dropping the subscript $``vr"$) $$\xi=f(\Psi)$$ for some non-decreasing function $f:\mathbb{R}\to\mathbb{R}$ (so-called a \textit{vorticity function})
satisfying $f(s)=0$ for $s\leq 0$, $f(s)>0$ for $s>0$. Due to
 $$-\frac 1 {r^2} {\mathcal{L}\Psi}= -\frac 1 {r^2} {\mathcal{L}\psi}=\xi, $$ 
the above system  is reduced to a semilinear elliptic equation in the half-space $\Pi$:
\begin{equation}\label{jac}
\begin{split}
& -\frac 1 {r^2}{\mathcal{L}\Psi}=f(\Psi),\quad r>0, \quad z\in\mathbb{R},\\
&\Psi(0,z)=-\gamma,\quad z\in\mathbb{R},\\
\frac{1}{r}\partial_r\Psi&\rightarrow -W, \quad \frac{1}{r}\partial_z\Psi\rightarrow 0\quad \textrm{as}\quad\sqrt{r^2+z^2}\to\infty.
\end{split}
\end{equation}

\subsection{Hill's spherical vortex revisited}\label{subsec_def_hill}\ \\

The stream function $\psi_H=\mathcal{G}[\xi_H]$ of the Hill's vortex $\xi_H=1_{\{|x|\leq1\}}$ \eqref{defn_hill_intro} is explicitly written by 
\begin{equation*}
\psi_{H}(x)=\left\{
\begin{aligned}
&\,\frac{1}{2}Wr^2\left(\frac{5}{2}-\frac{3}{2} {|x|^2} \right),\quad & |x|\leq 1,\\
&\,\frac{1}{2}Wr^2\frac{1}{|x|^3},\quad &|x|>1,
\end{aligned}
\right. 
\end{equation*} where
the traveling speed $W$ is set by ${2}/{15}$. 
The corresponding axi-symmetric velocity $u=(u_H^re_r+  u_H^ze_z)$ is obtained via
$$ u_H^r=-\frac{\partial_z\psi_H}{r},\quad u_H^z=\frac{\partial_r\psi_H}{r}$$
(see  \eqref{form_u}).
As noted in 
\cite{AF86}, $\psi_H$ is simply    obtained by solving the following O.D.E. problem for $\eta(|x|)=:\psi(x)/r^2$:
\begin{equation*}\label{ode_hill_unit}
\begin{split}
&\eta \in C^1[0,\infty): \mbox{strictly decreasing},\\
& -\frac 1 {t^4} (t^4 \eta')'=1, \quad 0<t<1,\\
 & -\frac 1 {t^4} (t^4 \eta')'=0, \quad t>1,\\
 & \eta(t)|_{t=1}=\frac{1}{2}W,\quad \eta(t)|_{t=\infty}=0.\\
\end{split}\end{equation*}
 Its relevant physical quantities are 
 \begin{equation}\label{phy_quan_hill}
\Lambda[\xi_H]=1,\quad  \Gamma[\xi_H]
=\frac{4}{3}\pi,\quad
P[\xi_H]
= \frac{4\pi }{15}, \quad E[\xi_H]=\frac{8\pi}{15\cdot 21} .
 \end{equation}
The adjusted stream function
$\Psi_H:=\psi_H-(1/2)Wr^2$ (i.e.  \eqref{defn_adj_str} with  $\gamma=0$)
 solves
\eqref{jac} for $\gamma=0$ and for the vorticity function 
$f=f_H=1_{(0,\infty)}$ (see \eqref{defn_hill_vor_fct}).\\

As mentioned in Remark \ref{rem_hill_scaling}, 
 for any given $\lambda>0$ and $a>0$, we set
\begin{equation}\label{defn_hill_gen}
 \xi_{H(\lambda,a)}(x)=\lambda{1}_{B_a}(x),
\end{equation}
where $B_a$ is the ball centered at the origin with radius $a>0$. 
In other words, we use the scaling
$$\xi_{H(\lambda,a)}(x)=\lambda\xi_H\left(\frac x a\right).$$ The corresponding stream function $\psi_{H(\lambda,a)}$ is obtained by
\begin{equation}\label{defn_hill_gen_st}
\psi_{H(\lambda,a)}=\lambda a^4\psi_H\left(\frac x a\right),
\end{equation}
and the corresponding traveling speed is set by
\begin{equation}\label{def_a}
 W_{H(\lambda,a)} =\frac {2}{15}\lambda a^2. 
\end{equation}
By the scaling, we see 
\begin{equation}\label{comp_imp}
\Lambda[\xi_{H(\lambda,a)}]=\lambda,\quad \Gamma[\xi_{H(\lambda,a)}] 
=\frac{4}{3}\pi\lambda a^3,
\quad P[\xi_{H(\lambda,a)}] 
= \frac{4\pi }{15}\lambda a^5,\quad  E[\xi_{H(\lambda,a)}]=\frac{8\pi}{15\cdot 21}\lambda^2 a^7.
\end{equation}

 
\begin{rem}\label{rk_wan_pf}
  As promised in Remark \ref{rk_wan} of Subsection \ref{subsec_main_thm}, here we verify that the orbital stability statement \cite[Corollary (H)]{Wan88} fails by an example.
  Indeed, we can simply set the example
  $$A_0= {B_{a}} $$ for $0<|a-1|<1/2$.
  In other words, we set the initial data $\xi_0=1_{A_0}=1_{B_a}=\xi_{H(1,a)}.$
 Then the solution has the form $$\xi(t,x)=1_{B_a}(x-t{\widetilde{W}} e_{x_3})$$ with the  traveling speed   ${\widetilde{W}} := W_{H(1,a)} =(2/15)a^2$. We note that the speed $\widetilde{W}$ is  different from the speed  $W_{H(1,1)}=(2/15)$ of the Hill's vortex $\xi_H=\xi_{H(1,1)}$ on the unit ball. 
  Then we can check the quantity in the left-hand side of \eqref{conclu_orb_patch_remark} admits a   positive lower bound  of order $1$ for large time. 
  More precisely, we first 
   observe  $$\xi(t)=1_{A_t}\quad\mbox{with}\quad A_t=B_a^{(t{\widetilde{W}} )}:=\{x\in\mathbb{R}^3\,|\, |x-(t{\widetilde{W}} ) e_z|<a\}, \quad t>0.$$  
 For the impulse part in \eqref{conclu_orb_patch_remark},  for any $t>0$ and for any $\tau\in\mathbb{R}$, if 
\begin{equation}\label{disjoint_time}
|t {\widetilde{W}} -\tau|\geq (1+a),
\end{equation} 
   then
the two sets $A_t$ and $B^\tau$ become disjoint so that we get $$\int_{A_t\triangle B^\tau}r^2dx
 =\left(\int_{A_t }r^2dx+\int_{  B^\tau}r^2dx\right)
 =\left(\int_{A_0 }r^2dx+\int_{  B}r^2dx\right) 
 =\frac{8\pi}{15} (a^5+1)\sim 1.
 $$ 
 On the other hand, for the non-invariance part in  
\eqref{conclu_orb_patch_remark}, 
 if $$|t {\widetilde{W}} -\tau|<(1+a),$$ then we can compute
 \begin{align*}\left|\int_{A_t } zr^2 dx -\int_{B^\tau } zr^2dx \right|
 &= \left|\left(\int_{A_0 } zr^2 dx +t{\widetilde{W}} \int_{A_0 } r^2 dx\right) -\left(\int_{B  } zr^2dx+\tau\int_{B  } r^2dx \right) \right|\\
&= \left| 
t{\widetilde{W}} \int_{A_0 } r^2 dx  - 
\tau\int_{B  }  r^2dx   \right|=\frac{8\pi}{15}
\left| 
a^5t{\widetilde{W}}   - 
\tau    \right|\\&\geq  \frac{8\pi}{15}
\left(
\left| 
a^5t{\widetilde{W}}   - 
t{\widetilde{W}}   \right|-\left| 
 t{\widetilde{W}}   - 
\tau    \right| \right) \geq  \frac{8\pi}{15}
\left(
\frac 2 {15} a^2|a^5-1|t-(1+a)\right)\\
&\geq  \frac{8\pi}{15}(a+1)
\left(
\frac 2 {15}a^2|a-1|t-1\right),
 \end{align*} which implies  for   
$t\geq 2(
(2/15)a^2|a-1|)^{-1}$, 
\begin{equation}\label{linear_sp}
 \left|\int_{A_t } zr^2 dx -\int_{B^\tau } zr^2dx \right|\geq  \frac{8\pi}{15}(a+1)
\left(
\frac 1 {15}a^2|a-1|t\right)\gtrsim |a-1|t\gtrsim 1.
\end{equation} 
In sum, we have verified that 
for any
$t\geq 2(
(2/15)a^2|a-1|)^{-1}$ 
(whether \eqref{disjoint_time} holds or not), 
 \begin{align*}\inf_{\tau\in\mathbb{R}}\left( \int_{A_t\triangle B^\tau}r^2dx+\left|\int_{A_t } zr^2 dx -\int_{B^\tau } zr^2dx \right|\right)
\gtrsim 1. 
\end{align*} However, the metric at the initial time has the estimate 
    \begin{align*}
 \int_{A_0\triangle B}r^2 dx+\left|\int_{A_0 } zr^2 dx -\int_{B } zr^2dx \right|= \left|\int_{A_0}r^2 dx-\int_{ B}r^2 dx\right|\sim  |a^5-1| \sim |a-1|,
\end{align*} which we can make arbitrarily small by taking the limit $a\to 1$. 
Hence, the statement  \cite[Corollary (H)]{Wan88} (or see Remark \ref{rk_wan}) cannot be true for this example.
 \end{rem}
\subsection{Notations} \ \\

We collect notations used in this paper.


$$\|f\|_{p}:=\|f\|_{L^p}=\|f\|_{L^p(\mathbb{R}^3)},
\quad p\in[1,\infty],$$
$$\mbox{(weighted }L^1\mbox{-space)}\quad L^1_w:=\{f:\mbox{measurable}\,|\, \|r^2f\|_{1}<\infty \}, \quad\mbox{where}\quad 
 \|r^2f\|_{1}:=\int_{\mathbb{R}^3}(x_1^2+x_2^2)|f(x)|\,dx,
 $$
   $$\int\,dx:=\int_{\mathbb{R}^3}\,dx,\quad\iint\,dydx:=\int_{\mathbb{R}^3}\int_{\mathbb{R}^3}\,dydx,$$
 $$\Pi:=\{(r,z)\in\mathbb{R}^2\,|\,  \, r>0\},\quad
 \overline\Pi=\{(r,z)\in\mathbb{R}^2\,|\,  \, r\geq 0\},$$
 $$\int\,drdz:=\int_{\Pi}\,drdz,\quad\iint\,dr'dz'drdz:=\int_{\Pi}\int_{\Pi}\,dr'dz'drdz.$$ We note   $$ dx=2\pi rdrdz$$ when restricted to axi-symmetric integrands. For instance, for axi-symmetric $f$,
 $$\|f\|_1=2\pi \int|f(r,z)|rdrdz.$$ For $R>0$, we define
\begin{equation}\begin{split} \label{defn_ball}
(\mbox{disks in  }\Pi) \quad B_R(r,z):&=\{(r',z')\in\Pi\,|\, |(r',z')-(r,z)|<R\},\\
(\mbox{balls in }\mathbb{R}^3)\quad B_R(x):&=\{x'\in\mathbb{R}^3\,|\, |x'-x|< R\},\\
(\mbox{tori in  }\mathbb{R}^3)\quad T_R(r,z):&=\{x'\in\mathbb{R}^3\,|\, |(r',z')-(r,z)|<R\}
 \end{split}\end{equation} when $$
  x_1'^2+x_2'^2=r'^2, \quad x_3'=z',\quad\mbox{for}\quad x'=(x_1',x_2',x_3').$$
We note that $T_R(r,z)$ is not a ball in $\mathbb{R}^3$ in general but a torus obtained by revolving 
$  B_R(r,z)$ with respect to the axis of symmetry. \ \\ 



We denote by $BUC(\overline{\mathbb{R}^d})$ the space of all bounded uniformly continuous functions in $\mathbb{R}^d$ and by $C^{\alpha}(\overline{\mathbb{R}^d})$ for $\alpha\in(0,1)$ the space of all uniformly
 H\"older continuous functions of the exponent $\alpha$ in $\mathbb{R}^d$. For an integer $k\geq 0$, $BUC^{k+\alpha}( \overline{\mathbb{R}^d})$ means the space of all $\phi\in BUC(\overline{\mathbb{R}^d})$ such that $\partial_{x}^{l}\phi\in BUC(\overline{\mathbb{R}^d})\cap C^\alpha(\overline{\mathbb{R}^d})$ for $|l|\leq k$.
 For  the half-space $\overline\Pi$, we define $BUC(\overline\Pi), C^\alpha(\overline\Pi), BUC^{k+\alpha}(\overline\Pi)$ in the same way as above. \\
 
\subsection{Elementary estimates of stream functions}\label{subsec_est_stream}\ \\


Before finishing the preliminary section, we  collect some elementary estimates for axi-symmetric   $\xi$.
First we study the decay rate of the stream function near $r=0$ and $ r=\infty$  thanks to the estimate \eqref{est_F} of the kernel $G$. These are essentially contained in \cite[Lemma 3.4]{FT81}.

\begin{lem}
\label{lem_est_stream}
For   axi-symmetric   $\xi\in \left(L^1_w\cap L^2\cap L^1\right)(\mathbb{R}^3)$, the   stream function $\psi=\mathcal{G}[\xi]$ satisfies  
\begin{align}
&|\psi(r,z)|
\lesssim  r\left(   \|r^2\xi\|_{1}+
   \|\xi\|_{L^1\cap L^2} \right), \quad (r,z)\in\Pi,\label{est_psi}  \\
   &|\psi(r,z)|
\lesssim_\delta r^{-1+\delta}\left(   \|r^2\xi\|_{1}+
   \|\xi\|_{L^1\cap L^2} \right), \quad (r,z)\in\Pi,\quad 0<\delta\leq 1.\label{est_psi_bdd}  
   \end{align} 
\end{lem}
\begin{proof}

We estimate
$$|\psi(r,z)|\leq 
 \int_\Pi
G(r,z,r',z') 
|\xi(r',z')|r'dr'dz'
=\int_{t<r/2}+\int_{t\geq r/2}=:I+II,
$$ for 
$  t
=\sqrt{(r-r')^2+(z-z')^2}.$
For the term $I$, we take $p\geq 2$ and $\tau=1/(2p)$ in \eqref{est_F}. Applying   H\"older's inequality
with $(1/p)+(1/p')=1$
 implies that 
\begin{equation*}\begin{split}
I&\lesssim \int_{t<r/2}
\frac{(rr')^{\frac 1 2 +\frac 1 {2p}}}{t^{1/p}}
 |\xi(r',z')|r' dr'dz' \lesssim  \left(\int_{t<r/2}  
 \frac{(rr')^{(1+p)/2}}{t}
 r' dr'dz'\right)^{1/p}  \|\xi{1}_{\{t<r/2\}}\|_{p'}
 \\ &\lesssim r^{(3/p) +1}  \|\xi{1}_{\{t<r/2\}}\|_{p'},
\end{split}\end{equation*}
where 
the last inequality follows from $r'\sim r$ when $t<r/2$.
 Since
$1\leq 2r'/r$ for $t<r/2,$
 \begin{equation*}\begin{split}
 \|\xi{1}_{\{t<r/2\}}\|_{p'}&\leq  \|\xi{1}_{\{t<r/2\}}\|_{1}^{1-(2/p)} \|\xi\|_{2}^{2/p}\lesssim r^{-2+(4/p)} \left(\|r^2\xi\|_{1}+  \|\xi\|_{2}\right).
\end{split}\end{equation*} 
 Thus we get
$$
I\lesssim 
r^{-1+(7/p)
} \left(\|r^2\xi\|_{1} + \|\xi\|_{2}\right).
$$
For the term $II$, we take $
\tau=3/2$ in \eqref{est_F}. Since $r'\leq 3t$ for $t\geq r/2$, we obtain
\begin{equation*}\begin{split}
II&\lesssim \int_{t\geq r/2}
\frac{(rr')^2}{t^3}
|\xi(r',z')|r'dr'dz'
\lesssim
r\|\xi\|_1.
\end{split}\end{equation*} 
Similarly, we get
$
II\lesssim
 r^{-1}\|r^2\xi\|_{1}. 
$
Thus, for any $\vartheta\in[0,1]$, we have
 \begin{equation*}\begin{split}
II\lesssim
r^{2\vartheta-1} \left(\|\xi\|_1+ \|r^2\xi\|_{1} \right).  
\end{split}\end{equation*}
By combining the estimates of $I$ and $II$, we get the estimate
\begin{equation*}
|\psi(r,z)|\leq 
C_pr^{-1+(7/p)}\left(\|r^2\xi\|_{1}+ \|\xi\|_2\right)+ C r^{2\vartheta-1} \left(\|\xi\|_1+ \|r^2\xi\|_1 \right), \quad 2\leq p<\infty,\quad 0\leq \theta\leq 1.
\end{equation*}
Thus, \eqref{est_psi} and \eqref{est_psi_bdd}  follow.

\end{proof}
As a consequence, the energy defined in \eqref{defn_en}  is well-defined with the following estimates.
\begin{lem}

\label{lem_est_energy} For   axi-symmetric   $\xi, \xi_1,\xi_2\in \left(L^1_w\cap L^2\cap L^1\right)(\mathbb{R}^3)$, we have 
\begin{equation}\label{est_E} 
|E[\xi]|\leq E[|\xi|]
\lesssim \left(   \|r^2\xi\|_{1}+
   \|\xi\|_{L^1\cap L^2}\right)  
\|r^2\xi\|^{1/2}_{1}  \|\xi\|^{1/2}_{1},
\end{equation}
\begin{equation}\label{est_iint} 
\left|\int_\Pi\int_\Pi G(r,z,r',z')\xi_1(r,z)\xi_2(r',z') rr' dr'dz'drdz\right|
\lesssim \left(   \|r^2\xi_1\|_{1}+
   \|\xi_1\|_{L^1\cap L^2} \right) 
\|r^2\xi_2\|^{1/2}_{1} \|\xi_2\|^{1/2}_{1}, 
\end{equation}
\begin{align}\label{est_E_diff} 
&\left|E[\xi_1]-E[\xi_2]\right|
\lesssim \left(   \|r^2(\xi_1+\xi_2)\|_{1}+
   \|\xi_1+\xi_2\|_{L^1\cap L^2} \right) 
\|r^2(\xi_1-\xi_2)\|^{1/2}_{1} \|\xi_1-\xi_2\|^{1/2}_{1}. 
\end{align}
 
\end{lem}
\begin{proof}

We use \eqref{est_psi} to estimate
\begin{equation*}\begin{split}
&\int_\Pi\left(\int_\Pi G(r,z,r'z')|\xi_1(r,z)|rdrdz\right)|\xi_2(r',z')|r'dr'dz' 
\lesssim
\left(   \|r^2\xi_1\|_{1}+
   \|\xi_1\|_{L^1\cap L^2}  \right) \int_\Pi r^2  |\xi_2(r,z)|drdz.  
\end{split}\end{equation*} 
Since we have  $$\|r^2\xi\|_{L^1(\Pi)}\sim \|r\xi\|_1\leq \|r^2\xi\|^{1/2}_{1}\|\xi\|^{1/2}_{1},$$ the estimate \eqref{est_iint} holds. 
 The estimate \eqref{est_E} follows from 
\eqref{est_iint}. 
By the symmetry $G(r,z,r',z')=G(r',z',r,z)$
and by setting $\tilde{\xi}=\xi_1-\xi_2$, we obtain (by suppressing the measure $rr'drdzdr'dz'$)
\begin{align*}
\frac 1 \pi\left(E[\xi_1]-E[\xi_2]\right)
&=\iint G(r,z,r',z')\xi_1(r,z)\xi_1(r',z')
-\iint G(r,z,r',z')\xi_2(r,z)\xi_2(r',z')
 \\
&=\iint G(r,z,r',z')\tilde{\xi}(r,z)\xi_1(r',z')
+\iint G(r,z,r',z')\xi_2(r,z)\tilde{\xi}(r',z')
\\
&=\iint G(r,z,r',z')\tilde{\xi}(r,z)\left(\xi_1(r',z')+\xi_2(r',z')\right).
\end{align*} 
Applying \eqref{est_iint} implies \eqref{est_E_diff}.
\end{proof}

 
Next we show the energy defined in \eqref{defn_en} is equal to the kinetic energy $\frac 1 2 \int |
u|^2 dx$.

\begin{lem}
\label{lem_en_iden_origin}
For axi-symmetric   $\xi\in \left(  L^1_w\cap L^\infty\cap L^1\right)(\mathbb{R}^3)$,   
the stream function $\psi=\mathcal{G}[\xi]$ is continuous  on $\overline\Pi$,
\begin{align}
   &\psi(r,z)\to 0\quad\mbox{as}\quad |(r,z)|\to\infty,\label{psi_conv}
   \end{align}  
  $$\psi\in H^2_{loc}(\Pi),\quad \mbox{and}\quad -\frac{1}{r^2}\mathcal{L\psi }=\xi\quad \mbox{a.e.}$$Moreover,
  the axi-symmetric velocity $
\mathcal{K}[\xi]$ is continuous on $\mathbb{R}^3$ and lies on $  L^2(\mathbb{R}^3)$  with 
 \begin{align}\label{iden_en}
 E[\xi]=
\frac 1 2 \|\mathcal{K}[\xi]\|_2^2=
 \frac{1}{2}\int_{\mathbb{R}^{3}}\frac{1}{r^2}
\left(|\partial_r\psi|^2+|\partial_z\psi|^2\right) 
  dx.
\end{align}  
\end{lem}

\begin{proof}

By setting 
$\omega(x)=r\xi(r,z) e_\theta(\theta)$, 
 we observe 
\begin{equation}\label{vor_l1}
\int |\omega|\,dx=\int r|\xi |^{1/2}|\xi|^{1/2}\,dx\leq
\|r^2\xi\|_{1}^{1/2}  \|\xi\|_{1}^{1/2}
\end{equation} 
  and 
$$\int |\omega|^2\,dx=\int r^2|\xi|^2\,dx\leq \|\xi\|_\infty\int r^2|\xi|\,dx <\infty.$$ It implies   
  $\omega\in L^p(\mathbb{R}^3)$ for any $p\in[1,2]$.   
  By setting 
\begin{equation}\label{rel_phi_}
\phi=\frac\psi r e_\theta,
\end{equation} we have
  $\phi =({4\pi|x|})^{-1}* \omega$ (see Subsection \ref{subsec_axi-sym}). 
  This representation 
implies
 $\nabla \phi\in L^{q}(\mathbb{R}^3)$ for any $q\in (3/2,6]$, and
 $\phi\in L^{q'}(\mathbb{R}^3)$ for any $q'\in (3,\infty)$  by the Hardy-Littlewood-Sobolev inequality (e.g. see \cite[p354]{Stein93}).
 Hence
 \begin{equation}\label{phi_w_1_q}
 \phi\in W^{1,q''}(\mathbb{R}^3)\quad\mbox{for any}\quad q''\in (3,6].
 \end{equation}
We also observe
$\omega\in L^\infty_{loc}(\mathbb{R}^3)$ since
 $\sup_{x\in U}|\omega
(x)|=\sup_{x\in U}r|\xi(r,z)|\leq C_U\|\xi\|_{\infty}$ for any bounded set $ U
\subset\mathbb{R}^3.$ 
Together with
\eqref{phi_w_1_q},
we get $\phi\in W_{loc}^{2,6}(\mathbb{R}^3)$ which implies
$\phi$ and $\nabla\phi$ are continuous on $\mathbb{R}^3$ by the Sobolev embedding (e.g. see \cite[p284]{Evans_book}). 
  Hence \eqref{phi_w_1_q} implies 
  $\phi \in BUC^{(1/2)}(\overline{\mathbb{R}^3})$ by the Morrey's inequality (e.g. see \cite[p280]{Evans_book}).      
 From  \eqref{rel_phi_}, 
 we have 
\begin{equation}\label{psi_from_phi}
 \psi(r,z)=r  \phi_2(r e_{x_1}+z e_{x_3}),\quad (r,z)\in \Pi,
\end{equation} 
 where $\phi(x)=\phi_{1}(x) e_{x_1}+\phi_{2}(x)e_{x_2} +\phi_3(x) e_{x_3}$.   Since $\phi_2$ is continuous on $\mathbb{R}^3$,   the stream function $\psi$ is continuous 
 on $\overline\Pi$ 
by defining $\psi|_{r=0}=0.$\\ 

 In particular, the velocity 
$\mathcal{K}[\xi]=(\nabla\times \phi )$ 
  is continuous on $\mathbb{R}^3$ with 
  $\mathcal{K}[\xi]
\in L^2(\mathbb{R}^3).$ 
 By integration by parts,
we have
\begin{align*}2 E[\xi]&=\int r\xi\frac \psi r  dx
= \int\omega\cdot \phi dx
=-\int\Delta\phi\cdot \phi dx
=\int\nabla\times\left(\nabla\times\phi \right)\cdot \phi dx\\&=\int \left(\nabla \times\phi \right)\cdot\left(\nabla \times\phi \right) dx=\int \frac  1 {r^2}\left(|\partial_r\psi|^2+|\partial_z\psi|^2\right)   dx
=
 \|\mathcal{K}[\xi]\|_2^2.\end{align*} 
The integration by parts we did in the above is justified in the following way:\\ We take
  a radial function
 $\varphi\in C^{\infty}_{c}(\mathbb{R}^3)$ satisfying  $\varphi(x)=1$ for $|x|\leq 1$ and $ \varphi(x)=0$ for $|x|\geq 2$, and set the cut-off function on $\mathbb{R}^3$ by $\varphi_M(x)=\varphi(x/M)$ for any $M>0$. Then we have
\begin{align*}\int\nabla\times\left(\nabla\times\phi \right)\cdot (\varphi_M\phi) dx&=\int \left(\nabla \times\phi \right)\cdot\left(\nabla \times(\varphi_M\phi) \right) dx 
\\
&=\int \varphi_M\left(\nabla \times\phi \right)\cdot\left(\nabla \times\phi \right) dx 
+\int \left(\nabla \times\phi \right)\cdot\left((\nabla \varphi_M)\times \phi\right) dx. 
\end{align*}
Since we know
 $ \omega=\nabla\times\left(\nabla\times\phi \right)\in L^1(\mathbb{R}^3)$, $\phi \in L^\infty(\mathbb{R}^3),$ and $
\nabla\times\phi \in L^2(\mathbb{R}^3), 
$  it is enough to show that the last integral   above vanishes as $M\to \infty$. We simply compute
$$\left|\int \left(\nabla \times\phi \right)\cdot\left((\nabla \varphi_M)\times \phi\right) dx\right|\lesssim \|\nabla \phi\|_{2}\|\nabla\varphi_M\|_{4}\|\phi\|_4\lesssim M^{-1/4} \|\nabla \phi\|_{2} \|\phi\|_4\to 0\quad\mbox{as}\quad M\to\infty.$$
Thus \eqref{iden_en} follows.\\

To show \eqref{psi_conv}, 
 let $\epsilon>0$. By \eqref{est_psi_bdd} in Lemma \ref{lem_est_stream}, there exists $R>0$ such that
 $$\sup_{r\geq R}|\psi(r,z)|<\epsilon.$$ On the other hand, from
  $\phi\in L^{q'}(\mathbb{R}^3)$ for any $q'\in (3,\infty)$ and
  $\phi \in BUC^{(1/2)}(\overline{\mathbb{R}^3})$, we get
  $$\lim_{|x|\to\infty} |\phi(x)|=0.$$ 
It implies, from $|\psi|=r|\phi|$,  there exists $Z>0$ such that
   $$\sup_{r\leq R, |z|\geq Z}|\psi(r,z)|\leq R\cdot \sup_{x_1^2+x_2^2\leq R^2, |x_3|\geq Z}|\phi(x)|<\epsilon.$$
   Thus we get 
   $$\sup_{r^2+z^2\geq R^2+Z^2}|\psi(r,z)|\leq  \epsilon,$$ which implies
    \eqref{psi_conv}.\ \\

Since   $\phi$ lies on $H^2_{loc}(\mathbb{R}^3)$,
we have  $-\Delta \phi =\omega$ {a.e. in} $ \mathbb{R}^3$ by the   elliptic regularity theory, which implies $$\psi\in H^2_{loc}(\Pi)\quad \mbox{and}\quad -\frac{\mathcal{L\psi }}{r^2}=\xi\quad \mbox{a.e. in }\,\Pi.$$



   \end{proof}
 
 \section{Variational problem with compactness and uniqueness}
 \subsection{Variational setting: Friedman-Turkington (1981)}\ \\

For $0<\mu,\nu,\lambda<\infty$, we set the space of admissible functions
\begin{align}\label{defn_adm}
{\mathcal{P}}_{\mu,\nu,\lambda}=\left\{\xi\in L^{\infty}(\mathbb{R}^3)\ \middle|\ \xi=\lambda{{1}}_A \mbox{  for   some axi-symmetric   
  } A\subset\mathbb{R}^3\,,\ 
\frac 1 2 \|r^2\xi\|_{1}=\mu,\,\,
\|\xi\|_1\leq \nu
 \right\},
\end{align} and we
 study the variational problem of  maximizing the energy $E$ on ${\mathcal{P}}_{\mu,\nu,\lambda}$.
In the rest of the paper, we set \begin{align}\label{var_prob}
{\mathcal{I}}_{\mu,\nu,\lambda}
= \sup_{\xi\in {\mathcal{P}}_{\mu,\nu,\lambda}}E[\xi]   
\end{align} and denote $
{\mathcal{S}}_{\mu,\nu,\lambda}$  by  the set  of maximizers of \eqref{var_prob}, i.e.
\begin{equation}\label{var_prob_max}
{\mathcal{S}}_{\mu,\nu,\lambda}=\{
\xi\in  {\mathcal{P}}_{\mu,\nu,\lambda}\,|\, E[\xi]=
 {\mathcal{I}}_{\mu,\nu,\lambda}
\}.
\end{equation}
We note that any $z$-directional translation
of   $\xi\in {\mathcal{S}}_{\mu,\nu,\lambda}$  
lie on the same set ${\mathcal{S}}_{\mu,\nu,\lambda}$. 

 \subsection{Theorems \ref{thm_cpt}, \ref{thm_uniq}: compactness and uniqueness of the set of maximizers}\label{subsec_cpt_uniq}\ \\

We introduce the following compactness theorem and uniqueness theorem, whose proofs will be given in later sections.  By assuming these theorems for a moment, we will produce our main result (Theorems \ref{thm_hill_gen} and \ref{thm_hill}) in Subsection \ref{subsec_proof}.
 
\begin{thm}\label{thm_cpt}[Compactness of  maximizing sequence]
Let $0<\mu,\nu, \lambda<\infty$. 
    Let  $\{\xi_n\}_{n=1}^\infty$ be a sequence of non-negative axi-symmetric functions in $\mathbb{R}^3$ and let $\{a_n\}_{n=1}^\infty$ be  a sequence of positive numbers such that
$$a_n\to 0\quad \mbox{as}\quad n\to\infty,$$
$$\limsup_{n\to\infty}\|\xi_n\|_{{1}}\leq \nu,\quad
\lim_{n\to\infty}\int_{\{x\in\mathbb{R}^3\,|\,|\xi_n(x)-\lambda|\geq a_n\}}\xi_n\,dx=0,\quad \color{black}
\lim_{n\to\infty}\frac 1 2\|r^2\xi_n\|_{1}
  =\mu,  $$ 
  $$\sup_n\|\xi_n\|_2<\infty,\quad  \mbox{and}\quad \quad \lim_{n\to\infty}E[\xi_n]= {\mathcal{I}}_{\mu,\nu,\lambda}.$$
\color{black}  
 Then there exist    a subsequence $\{\xi_{n_k
 }\}_{k=1}^\infty$, a
 sequence $\{c_{k}\}_{k=1}^\infty\subset \mathbb{R}$, and a function $\xi\in {\mathcal{S}}_{\mu,\nu,\lambda} $ such that
 \begin{equation}\label{conclu_cpt} 
\|r^2\left(\xi_{n_k}(\cdot+c_ke_z)-\xi\right)\|_{1}\to 0\quad  \mbox{as}\quad k\to \infty.
 \end{equation}
In particular, the set ${\mathcal{S}}_{\mu,\nu,\lambda} $ is non-empty.
\end{thm}

 \begin{thm}\label{thm_uniq}
 [Uniqueness]   There exists a   constant $M_1>0$ such that for any constants $0<\mu,\nu, \lambda<\infty$ satisfying $\mu\nu^{-5/3}\lambda^{2/3}\leq M_1$, \begin{equation*}
 \mathcal{S}_{\mu,\nu,\lambda}= \{\xi_{H(\lambda,a)}(\cdot+c e_z)\,|\, c\in \mathbb{R}\}, \end{equation*}  where $\xi_{H(\lambda, a)}$ is the Hill's vortex for the vortex strength constant $\lambda$ with the radius $a=a(\lambda,\mu)>0$ solving the equation $\mu= (4/15)\pi \lambda a^5  $. \end{thm}

\begin{rem}\label{rem_m1}
The optimal constant of $M_1$ satisfying the above theorem can be explicitly computed even if we do not need the exact value   in the sequel. Indeed, we consider the Hill's vortex $\xi_H=\xi_{H(1,1)}$ of unit strength on the unit ball. By a direct computation (or see \eqref{phy_quan_hill} in Subsection \ref{subsec_def_hill}), we know
$$\xi_H\in \mathcal{P}_{(4/15)\pi,(4/3)\pi,1}.$$ 
 On the other hand, by Theorem \ref{thm_uniq}, for any  $\nu>0$ satisfying
 $(4/15)\pi\nu^{-5/3}\leq M_1$, we obtain
\begin{equation}\label{state_uniq_m1}
\xi_H\in \mathcal{S}_{(4/15)\pi,\nu,1}. \end{equation}  Since the admissible class
$\mathcal{P}_{(4/15)\pi,\nu,1}$ is increasing in $\nu>0$, we conclude that 
\eqref{state_uniq_m1} holds if and only if $\nu\geq (4/3)\pi$, which 
  was  inferred by \cite[p21]{Ben76} in 1976.
 Hence, we can set 
$$M_1= \frac 4 {15}\pi \left(\frac 4 3 \pi\right)^{-5/3},
$$ and it is sharp.
\end{rem}


\subsection{Existence and uniqueness of global weak solutions}\label{subsec_exist_weak_sol}\ \\
 



We consider 
the case when
 the Euler  equations  admits the active scalar transport equation form \eqref{3d_Euler_eq}. Therefore existence and uniqueness of solutions $\xi(t)$ can be studied analogously as the two-dimensional case. We refer to
\cite{UI}, \cite{Maj86}, \cite{Raymond}, \cite{SY}, \cite{ChaeKim97}, \cite{CI},
  \cite{Danchin}, \cite{AHK2010}, \cite{Jiu15} in various settings.\\
   
 For our stability result, we  just  need a weak solution preserving the quantities listed in Subsection \ref{subsec_axi_euler}. Since our main interest lies not on existence of such weak solutions but on stability of them, we only briefly explain the existence (and the uniqueness) here.
We simply  take 
initial data $\xi_0$ regular enough in order to have existence and uniqueness of the corresponding weak solution $\xi(t)$ with desirable conservations. More precisely we consider a  non-negative axi-symmetric initial data
$\xi_0\in (L^1_w\cap L^{\infty}\cap L^{1})(\mathbb{R}^{3})$
satisfying $ (r\xi_0) \in L^\infty(\mathbb{R}^3)$. Such  a regularity  
 guarantees    existence and uniqueness of the corresponding weak solution of the Euler equations \eqref{euler_velocity}  (see Remark \ref{rem_exist_weak} after Lemma \ref{lem_exist_weak_sol}). This axi-symmetric solution  satisfies \eqref{3d_Euler_eq} both in its weak form and   in the renormalized sense  of DiPerna-Lions  \cite{DL89} (refer to \cite[Definitions  2, 3]{NoSe} for precise notions of  such solutions).
 As a result, we obtain all conservations listed in Subsection \ref{subsec_axi_euler} by using \cite[Theorems 1, 2, 3]{NoSe} (also see \cite[Section 5.1]{AC2019} for 2d case). We state the result in the form of a lemma below. 
\begin{lem}
\label{lem_exist_weak_sol}
For any  non-negative axi-symmetric   
$\xi_0\in(L^1_w\cap L^{\infty}\cap L^{1})(\mathbb{R}^{3})$
satisfying 
$(r\xi_0)\in L^\infty(\mathbb{R}^3),$
 there exists a unique    weak solution $\xi\in 
 L^\infty(0,\infty; (L^1_w\cap L^{\infty}\cap L^{1})(\mathbb{R}^{3}))$ of \eqref{3d_Euler_eq}
 for the initial data $\xi_0$ such that\\
 $$\xi(t)\geq0\,:\,\mbox{ axi-symmetric},$$
   \begin{equation}\label{conserv_euler} 
\begin{aligned}
||\xi(t)||_{q}&=||\xi_0||_{q},\quad 1\leq q\leq \infty,  \\
||r^2\xi(t)||_{1}&=||r^2\xi_0||_{1},  \\
E[\xi(t)]&=E[\xi_0],\qquad \textrm{for all}\ t> 0, 
\end{aligned}
\end{equation}  
 and,   for any $0<a<b<\infty$ and for each $t> 0$,
\begin{equation}\label{conserv_parti}
 \int_{\{x\in\mathbb{R}^3\,|\,a<\xi(x,t)<b\}}\xi(x,t)\,dx=\int_{\{x\in\mathbb{R}^3\,|\,a<\xi_0(x)<b\}}\xi_0(x)\,dx.\ \\
\end{equation}  
 \ \\
\end{lem}

 \begin{rem}\label{rem_exist_weak}
 For axi-symmetric data $\xi_0\in (L^1_w\cap L^{\infty}\cap L^{1})(\mathbb{R}^3)$
with  $ (r\xi_0) \in L^\infty(\mathbb{R}^3)$, 
the initial velocity 
$u_0:=\mathcal{K}[\xi_0]$ lies on $L^2(\mathbb{R}^3) $ by Lemma \ref{lem_en_iden_origin}, and  
the initial vorticity 
$\omega_0:=(r\xi_0)e_\theta$ lies on $(L^1\cap L^\infty)(\mathbb{R}^3)$ (e.g. see the estimate \eqref{vor_l1}). In this setting, existence with uniqueness of a weak solution
can be obtained similarly as in the two-dimensional case. For instance, if one use the earlier paper  \cite{UI},  then
the existence of a weak solution for such data can be found in  \cite[Theorem 4.1]{UI}. 
Thanks to the transport  structure for $\xi$ in  
\eqref{3d_Euler_eq}, 
we can show the norm $\|\omega\|_{L^\infty(0,T;L^\infty(\mathbb{R}^3))}$ is finite for any finite $T>0$ (e.g. see   the \textit{a priori} estimate (1.26) on p56 in \cite{UI}). Then, the uniqueness is obtained by \cite[Theorem 2.2]{UI}. 
 Or equivalently,  one may simply use  \cite[Theorem 3.3]{Raymond} (also see \cite[Theorem 1]{Danchin}) for both existence and uniqueness.
\end{rem}

 \begin{rem}\label{rem_uniq_sol_detail}
 The assumption 
for initial data in Lemma \ref{lem_exist_weak_sol}
   might be weakened if one does not ask  \textit{uniqueness} of solutions. 
   To have a clear presentation toward nonlinear stability which is our main goal, here we do not seek such a generalization. For readers interested in the existence issue (without asking uniqueness) with desirable conservations \eqref{conserv_euler},   \eqref{conserv_parti},    we refer to  \cite{NoSe} and references therein.
  \color{black} 
 \end{rem}
 \begin{rem}\label{exist_sols_vel}
Regarding on smooth solutions $u(t)$ of
the velocity form \eqref{euler_velocity} of the three-dimensional Euler equations, it is still an open problem whether it can develop a finite-time blow-up from a $C^\infty$ initial data. In particular, it looks very difficult to convince a blow-up via direct numerical experiments as explained in \cite{VaVi}.
For classical solutions $u(t)$ of \eqref{euler_velocity}, local-in-time existence and uniqueness  in Sobolev space $H^s$ for $s>5/2$  
 \cite{MR271984}, \cite{MR0481652}
and in H\"{o}lder space  $C^{1,\alpha}$ for $\alpha>0$  \cite{MR1544733}, \cite{Gunther}
have been known. 
Very recently,   \cite{Elgindi19} showed that  the latter case for small $\alpha>0$ admits a finite-time blow-up. It is interesting that the  initial velocity  $u_0\in C^{1,\alpha}$   developing  a   blow-up in \cite{Elgindi19} is axi-symmetric  without swirl while  the corresponding initial relative vorticity $\xi_0(:=\omega^\theta_0/r)$ doe not lie on $L^\infty(\mathbb{R}^3)$. More precisely, it is not bounded on the axis $\{r=0\}$.
As noted  in \cite[p10]{Elgindi19}, 
in order to have  a $C^\infty$ vorticity  $\omega=\omega^\theta e_\theta$, it is a necessary condition that $\omega^\theta$ vanishes  linearly (so $\xi$ is at least bounded)  on the axis $\{r=0\}$ 
 (also see \cite{LiuWang}).
 Lastly, regarding on  weak solutions $u(t)$ of \eqref{euler_velocity}, even uniqueness 
fails for any solenoidal vector field $u_0\in L^2(\mathbb{R}^3)$ by \cite{MR2564474},  \cite{MR2838398}
(also see \cite{MR1231007}, \cite{MR1476315}, \cite{MR2600877}).
 \end{rem}
 
 \subsection{Proof of  nonlinear stability (Theorems \ref{thm_hill_gen},  \ref{thm_hill}) }\label{subsec_proof}\ \\
 
 Now we are ready to prove Theorem \ref{thm_hill_gen} by assuming      Theorem \ref{thm_cpt}(compactness)  and Theorem \ref{thm_uniq}(uniqueness).
\begin{proof}[Proof of Theorem \ref{thm_hill_gen}]



 
We recall that  $\xi_H=\xi_{H(1,1)}={1}_B$ is the  the Hill's vortex
\eqref{defn_hill_gen}, where $B$ is the unit ball centered at the origin.   
Let us suppose that the conclusion
   of Theorem \ref{thm_hill_gen} 
were false. 
Then there exist    a constant $\varepsilon_0>0$ and a sequence  $\{\xi_{0,n}\}_{n=1}^\infty$  of non-negative axi-symmetric functions,
and  a sequence  $\{t_n\}_{n=1}^\infty$ of non-negative numbers   such that, for each $n\geq1$, we have
$\xi_{0,n}, \, (r\xi_{0,n})\in L^\infty(\mathbb{R}^3),$
\begin{align}\label{init_conv}
\|\xi_{0,n}-\xi_H\|_{L^1\cap L^2}+\|r^2\left(\xi_{0,n}-\xi_H\right)\|_{1} 
\leq \frac 1 {n^2},
\end{align}  and
 \begin{align}\label{con_conv}
 \inf_{\tau\in\mathbb{R}}\left\{ 
\|\xi_{n}(t_n,\cdot+\tau e_z)-\xi_H\|_{L^1\cap L^2}+\|r^2\left(\xi_{n}(t_n,\cdot+\tau e_z)-\xi_H\right)\|_{1}  
\right\}
\geq \varepsilon_0.
\end{align} 
where  $\xi_{n}(t)$  is the  global-in-time weak solution of \eqref{3d_Euler_eq}  for the initial data $ \xi_{0,n}$  obtained by
    Lemma \ref{lem_exist_weak_sol}.\\ 
  
  
We set $\mu_0=(4/ {15}) \pi$, $\lambda_0=1$ and  fix any   $\nu_0>0$ satisfying
$\mu_0\nu_0^{-5/3}\lambda_0^{2/3}\leq M_1$,
where $M_1$ is the constant in Theorem \ref{thm_uniq}.
Then  the theorem says   \begin{equation}\label{orbit}  \mathcal{S}_{\mu_0,\nu_0,\lambda_0}= \{\xi_{H }(\cdot+c e_z)\,|\, c\in \mathbb{R}\}. \end{equation}    
  By \eqref{init_conv} and the estimate \eqref{est_E_diff},
  $$\lim_{n\to \infty} E[\xi_{0,n}]=E[\xi_{H}]=\mathcal{I}_{\mu_0,\nu_0,\lambda_0}.$$
  We write $\xi_{n}=\xi_{n}(t_n)$ by suppressing $t_n$.
  Thus, by the conservations \eqref{conserv_euler}, we get
\begin{equation}\begin{split}\label{conserv_}
& \lim_{n\to\infty} \frac  1 2 \|r^2\xi_{n}\|_{1}=\frac 1 2  \|r^2\xi_{H}\|_{1}=\mu_0,\quad
 \lim_{n\to\infty}\|\xi_n \|_{{1}}= \|\xi_H\|_1\leq \nu_0,\\&
  \lim_{n\to\infty}\|\xi_n \|_{{2}}= \|\xi_H\|_2<\infty,\quad  \lim_{n\to\infty}E[\xi_n]=\mathcal I_{\mu_0,\nu_0,\lambda_0}.
\end{split}\end{equation}
 We claim \begin{equation}\label{claim_outside}
\lim_{n\to\infty}\int_{\{x\in\mathbb{R}^3\,|\,|\xi_n(x)-\lambda_0|\geq 1/n\}}\xi_n\,dx=0.
\end{equation}  
To prove, we observe that
the conservation \eqref{conserv_parti} (together with \eqref{conserv_euler} for $q=1$) implies $$\int_{\{|\xi_n(x)-\lambda_0|\geq 1/n\}}\xi_n\,dx=\int_{\{|\xi_{0,n}(x)-\lambda_0|\geq 1/n\}}\xi_{0,n}\,dx=:I_n.$$
By setting ${D(n)}=\{x\in\mathbb{R}^3\,|\,|\xi_{0,n}(x)-\lambda_0|\geq 1/n\}$ and by recalling 
$\xi_H=1_B=\lambda_0{1}_B$, we observe
$$\|\xi_{0,n}-\xi_H\|_1\geq \|\xi_{0,n}-\xi_H\|_{L^1( {D(n)}\cap B)}=\int_{{D(n)}\cap B}|\xi_{0,n}(x)-\lambda_0|\,dx \geq \int_{{D(n)}\cap B} \frac 1 n\,dx=
\frac 1 n|{{D(n)}\cap B}|.$$ Thus we estimate 
 \begin{align*}
I_n&=\|\xi_{0,n}\|_{L^1({D(n)})}=\|\xi_{0,n}\|_{L^1({D(n)}\cap B)}+\|\xi_{0,n}\|_{L^1({D(n)}\cap B^c)}\\ &\leq \|\xi_{0,n}-\xi_H\|_{L^1({D(n)}\cap B)}+\|\xi_H\|_{L^1({D(n)}\cap B)}+\|\xi_{0,n}-\xi_H\|_{L^1({D(n)}\cap B^c)}\\
&\leq\|\xi_{0,n}-\xi_H\|_{L^1( B)}+\|\xi_H\|_{L^1({D(n)}\cap B)}+\|\xi_{0,n}-\xi_H\|_{L^1(B^c)}\\
&\leq\|\xi_{0,n}-\xi_H\|_{1}+|{{D(n)}\cap B}| \leq(n+1)  \|\xi_{0,n}-\xi_H\|_{1}\leq \frac{n+1}{n^2}\to 0\quad \mbox{as} \quad n\to \infty,
\end{align*} which shows \eqref{claim_outside}.\\




Now we apply  Theorem \ref{thm_cpt} to the sequence $\{\xi_n\}_{n=1}^\infty$ with the choice $a(n)=1/n$ to obtain
a subsequence (still denoted by $\{\xi_n\}$ after reindexing),
 $\{c_{n}\}_{n=1}^\infty\subset \mathbb{R}$, and  
 $\xi\in\mathcal S_{\mu_0,\nu_0,\lambda_0} $ such that  
\begin{equation}\label{conv_weight}
\|r^2\left(\xi_{n}(\cdot+c_ne_z)-\xi\right)\|_{1}\to 0 \quad\mbox{as}\quad n\to \infty.
\end{equation}
By \eqref{orbit}, we know $\xi=\xi_H(\cdot+ce_z)$ for some $c\in\mathbb{R}$.   We may assume $c=0$  by shifting $c_n$ by the constant $c$.\\ 
 
By uniform boundedness in $L^2$ from \eqref{conserv_},  the sequence $\{\xi_n(\cdot +c_n e_z)\}$ subsequently converges weakly in $L^2(\mathbb{R}^3)$, and the weak limit agrees with  $\xi_H$ by \eqref{conv_weight}. Thus,
   convergence of the norm
$$\lim_{n\to\infty}\|\xi_n(\cdot+c_ne_z) \|_{{2}}=\lim_{n\to\infty}\|\xi_n \|_{{2}}= \|\xi_H\|_2$$ from \eqref{conserv_} gives the strong convergence  in $L^2$ (still denoted by $\{\xi_n\}$) 
\begin{equation}\label{conv_l2}
 \xi_{n}(\cdot+c_ne_z)\to \xi_H\quad\mbox{in}\quad L^2(\mathbb{R}^3)\quad\mbox{as}\quad n\to \infty.
\end{equation}
In particular, $ \xi_{n}(\cdot+c_ne_z)\to \xi_H$ in $L^1(U)$ for any bounded $U\subset \mathbb{R}^3$ by using   H\"older's inequality.
  Since $\mbox{spt} \,\xi_H=B$ and
  \begin{align*}
 \|\xi_{n}(\cdot+c_ne_z)-\xi_H\|_{L^1(\mathbb{R}^3)}&=\|\xi_{n}(\cdot+c_ne_z)-\xi_H\|_{L^1(B)}+\|\xi_{n}(\cdot+c_ne_z)\|_{L^1(B^c)}\\
 &=\|\xi_{n}(\cdot+c_ne_z)-\xi_H\|_{L^1(B)}+
 \|\xi_{n}(\cdot+c_ne_z)\|_{L^1(\mathbb{R}^3)}-\|\xi_{n}(\cdot+c_ne_z)\|_{L^1(B)},
\end{align*} 
the convergences \eqref{conserv_} and \eqref{conv_l2} imply
  \begin{align*}
 \limsup_{n\to\infty}\|\xi_{n}(\cdot+c_ne_z)-\xi_H\|_{L^1(\mathbb{R}^3)}& \leq  \lim_{n\to\infty}
 \|\xi_{n}(\cdot+c_ne_z)\|_{L^1(\mathbb{R}^3)}- \liminf_{n\to\infty}\|\xi_{n}(\cdot+c_ne_z)\|_{L^1(B)}\\ &
\leq  
 \|\xi_H \|_{L^1(\mathbb{R}^3)}-\|\xi_H\|_{L^1(B)}=0.
\end{align*}  

In sum, we have
 $\xi_{n}(\cdot+c_ne_z)\to \xi_H$ in $(L^1_w\cap L^2\cap L^1)(\mathbb{R}^3)$, which contradicts to 
\eqref{con_conv} because 
  \begin{align*}
0& 
= \lim_{n\to\infty}
\left\{||\xi_n(\cdot+c_ne_z)-\xi_H||_{L^1\cap L^2} +||r^2\left(\xi_n(\cdot+c_ne_z)-\xi_H\right)||_{1}\right\} \\ 
&\geq  \liminf_{n\to\infty}\left(\inf_{\tau\in\mathbb{R}}
\left\{||\xi_n(\cdot+\tau e_z)-\xi_H||_{L^1\cap L^2} +||r^2\left(\xi_n(\cdot+\tau e_z)-\xi_H\right)||_{1}\right\} \right)
\geq \varepsilon_0.
\end{align*} 
\color{black}
\end{proof}
\begin{proof}[Proof of Theorem \ref{thm_hill_gen}]
 
For axi-symmetric set $ {A_0}$ satisfying 
\eqref{assump_uniq_pat}, if we set the initial data
$\xi_0=1_{A_0}$, then the data satisfies \eqref{assump_uniq}. Since
$$\int_{A_0\Delta B}1 \,dx=\|\xi_0-\xi_H\|_{1}=\|\xi_0-\xi_H\|_{2}^2,\quad
\int_{A_0\Delta B}r^2 \,dx=\|r^2(\xi_0-\xi_H)\|_{L^1(\mathbb{R}^3)}, $$
 we obtain Theorem \ref{thm_hill}   by applying Theorem \ref{thm_hill_gen} into the unique   weak solution $\xi(t) =1_{A_t}$ obtained from Lemma \ref{lem_exist_weak_sol}.
\end{proof}

It remains to show Theorems \ref{thm_cpt} and \ref{thm_uniq}.
In Section \ref{sec_exist}, as a warm-up section, we revisit the existence result of a maximizer for \eqref{var_prob} due to \cite{FT81}. In fact, we show that such a maximizer maximizes the energy in a slightly larger class. Then, 
 in Section \ref{sec_exist_vortex}, we prove that every maximizer gives a steady vortex ring by constructing a sequence of admissible perturbations. In Section \ref{sec_cpt}, we obtain Theorem \ref{thm_cpt} via 
 concentrated compactness due to \cite{Lions84a}.  Lastly, in Section \ref{sec_uniq_proof_hill}, we apply the uniqueness result of \cite{AF86} to prove Theorem \ref{thm_uniq}.


%

\section{Existence and properties of maximizers}\label{sec_exist}
In this section, our goal is to show the existence of a maximizer (Theorem \ref{thm_exist_max})  below, which will be used  in Section \ref{sec_cpt} when proving 
Theorem \ref{thm_cpt}. 

 \subsection{Variational problem in larger spaces} \label{subsec_other_cl}\ \\
 
Before stating the existence theorem, we introduce
some other spaces of admissible functions.
For $0<\mu,\nu,\lambda<\infty$, we set the following spaces of admissible functions
\begin{equation}\begin{split}\label{defn_prime_class}
&\mathcal{P}'_{\mu,\nu,\lambda}=\left\{\xi\in L^{\infty}(\mathbb{R}^3)\ \middle|\ \xi:\mbox{axi-symmetric}, 0\leq \xi\leq \lambda,
\,\frac 1 2 \|r^2\xi\|_{1}=\mu, 
\,\|\xi\|_{1}\leq \nu\    \right\},\\
&\mathcal{P}''_{\mu,\nu,\lambda}=\left\{\xi\in L^{\infty}(\mathbb{R}^3)\ \middle|\ \xi:\mbox{axi-symmetric}, 0\leq \xi\leq \lambda,
\,\frac 1 2 \|r^2\xi\|_{1}\leq \mu,\, 
\|\xi\|_{1}\leq \nu\    \right\}.
\end{split}\end{equation}  
\begin{rem}
We observe the set relations:
\begin{equation}\label{set_rel}\mathcal{P}''_{\mu,\nu,\lambda}\supset  \mathcal{P}'_{\mu,\nu,\lambda}\supset {\mathcal{P}}_{\mu,\nu,\lambda},   \end{equation} and note that
$\mathcal{P}''_{\mu,\nu,\lambda}$ is closed under the weak-$L^2$ topology. i.e.
if $\{\xi_n\}$ is a
sequence in $\mathcal{P}''_{\mu,\nu,\lambda}$ and if
 $\xi_n\rightharpoonup \xi$ in $L^2(\mathbb{R}^3)$ for some $\xi\in L^2(\mathbb{R}^3)$ as $n\to \infty$, then
 the weak-limit $\xi$ lies on $\mathcal{P}''_{\mu,\nu,\lambda}$.
\end{rem}

As in \eqref{var_prob}, 
\eqref{var_prob_max}, we set  the variational problems
\begin{align}\label{var_prob_extra}
\mathcal{I}'_{\mu,\nu,\lambda}=\sup_{\xi\in \mathcal{P}'_{\mu,\nu,\lambda}}E[\xi],\quad\mathcal{I}''_{\mu,\nu,\lambda}=\sup_{\xi\in 
\mathcal{P}''_{\mu,\nu,\lambda}}E[\xi]
\end{align}
and denote $\mathcal{S}'_{\mu,\nu,\lambda},\,\mathcal{S}''_{\mu,\nu,\lambda}$   the sets of maximizers of 
\eqref{var_prob_extra}, respectively.\\

\begin{thm}\label{thm_exist_max}
For $0<\mu,\nu,\lambda<\infty$, we have   
   $$  {\mathcal{I}}_{\mu,\nu,\lambda}= \mathcal{I}'_{\mu,\nu,\lambda} =\mathcal{I}''_{\mu,\nu,\lambda}\in(0,\infty)\quad\mbox{and}\quad 
 {\mathcal{S}}_{\mu,\nu,\lambda}= \mathcal{S}'_{\mu,\nu,\lambda} =\mathcal{S}''_{\mu,\nu,\lambda}\neq \emptyset.$$
\end{thm}
We prove the above theorem by building a series of lemmas in this section. In fact, we closely follow the approach of  the proof of \cite[Theorem 2.1]{FT81}.  Since some ingredient 
 of the proof   is needed later again (e.g. in Section \ref{sec_cpt}  when proving Theorem \ref{thm_cpt}), we reproduce the proof   here. More specifically,
 the notion of Steiner symmetrization (see Proposition \ref{prop_steiner}) and energy convergence lemma (see Lemma \ref{lem_energy_conv}) will be used again in  Section \ref{sec_cpt}.

\begin{rem}\label{rem_ft81}
  \cite[Theorem 2.1]{FT81} says, in our terminology, that there exists a compactly supported function 
 $$\xi\in   \left(\mathcal{S}_{\mu,\nu,\lambda} \cap\mathcal{S}'_{\mu,\nu,\lambda}\right)$$ satisfying
the symmetry
$\xi(r,z)=\xi(r,-z)$ together with the property
\begin{equation*}
\begin{aligned}
&\xi=\lambda{{1}}_{ \{\Psi>0\}}, 
\quad
 \Psi:= \mathcal{G}[\xi]-\frac{1}{2}Wr^2-\gamma
\end{aligned}
\end{equation*} for some $W>0$ and $\gamma\geq 0$.
Such a function $\xi$
 can be obtained  from a weak-limit of a sequence $\{\xi_\beta\}_{\beta>0}$  as $\beta\to 0$  where $\xi_\beta$ is a  maximizer for the penalized energy functional \eqref{defn_pen_en} (see \cite[Lemma 5.1]{FT81}).
 \end{rem}
 
  In the sequel, we frequently reduce the variational  problems 
\eqref{var_prob}, \eqref{var_prob_extra}  
   to the case $\nu=\lambda=1$ by the scaling    
\begin{align*}
 {\xi_{\nu,\lambda}}(x)=\frac{1}{\lambda  }\xi\left(\Big( \frac{\nu}{\lambda}\Big)^{1/3} x\right).
\end{align*} It is easy to check that if $ \xi\in {\mathcal{P}}_{\mu,\nu,\lambda}$, then $ {\xi_{\nu,\lambda}}\in {\mathcal{P}}_{M,1,1 }$ for $M:=\mu\nu^{-5/3}\lambda^{2/3}$
 and $E[ {\xi_{\nu,\lambda}}]=\lambda^{1/3}\nu^{-7/3}E[\xi]$
due to  $$\mathcal{G}[ {\xi_{\nu,\lambda}}](x)=\frac{1}{\lambda  }\Big( \frac{\nu}{\lambda}\Big)^{-4/3} \mathcal{G}[\xi]\left(\Big( \frac{\nu}{\lambda}\Big)^{1/3} x\right).$$   
Thus we get 
\begin{equation}\label{scaling}
\xi\in {\mathcal{S}}_{\mu,\nu,\lambda}
\, \mbox{ if and only if }\, \xi_{\nu,\lambda}\in {\mathcal{S}}_{M,1,1}
.
\end{equation}
From now on, we abbreviate the notations as
${\mathcal{P}}_{\mu}={\mathcal{P}}_{\mu,1,1}$, ${\mathcal{I}}_{\mu}={\mathcal{I}}_{\mu,1,1}$, 
and 
${\mathcal{S}}_{\mu}={\mathcal{S}}_{\mu,1,1}$. Similarly, we
 abbreviate the notations as
 $\mathcal{P}'_{\mu}=\mathcal{P}'_{\mu,1,1}$, $\mathcal{I}'_{\mu}=\mathcal{I}'_{\mu,1,1}$, 
$\mathcal{S}'_{\mu}=\mathcal{S}'_{\mu,1,1}$,
$\mathcal{P}''_{\mu}=\mathcal{P}''_{\mu,1,1}$, $\mathcal{I}''_{\mu}=\mathcal{I}''_{\mu,1,1}$, 
and
$\mathcal{S}''_{\mu}=\mathcal{S}''_{\mu,1,1}$. \\
 

As a warm-up, we first check that the maximum values 
${\mathcal{I}}_{\mu} ,\mathcal{I}'_{\mu}, \mathcal{I}''_{\mu}$
are non-trivial  for each $\mu>0$.

\begin{lem} 
\label{lem_nontrivial_extra}
Let $\mu\in(0,\infty)$. Then
  $$0<{\mathcal{I}}_{\mu} \leq \mathcal{I}'_{\mu}\leq \mathcal{I}''_{\mu}<\infty. $$
\end{lem}

\vspace{5pt}

\begin{proof}

First we get
 $$0<\sup_{\xi \in {\mathcal{P}}_{\mu}}E[\xi]= {\mathcal{I}}_\mu$$  since     any  $\xi  $ in ${\mathcal{P}}_{\mu}\neq \emptyset$ is non-negative,    and the kernel $G$ is positive a.e. The set relations \eqref{set_rel} 
give
$$0< {\mathcal{I}}_{\mu}\leq  \mathcal{I}'_{\mu}\leq \mathcal{I}''_{\mu}.$$
Lastly, from the estimate \eqref{est_E}, we have
\begin{align*}
 \mathcal{I}''_\mu=\sup_{\xi\in \mathcal{P}''_{\mu}} E [\xi]\lesssim (1+\mu) \mu^{1/2}<\infty.  
\end{align*}

\end{proof}

The following lemma is useful when we need convergence of the  energy for a weak-convergent sequence $\{\xi_n\}$ when the energy of each member $\xi_n$ is   uniformly concentrated in a fixed bounded set.
\begin{lem}\label{lem_en_diff}For  non-negative axi-symmetric functions $\xi_1, \xi_2\in \left(L^{1}_w\cap L^2\cap L^{1}\right)(\mathbb{R}^{3})$ and for   axi-symmetric  set $U\subset\mathbb{R}^3$,  we have
\begin{equation}\begin{split}\label{est_en_diff}
\left| E[\xi_1]-E[\xi_2]\right|
&\leq  \frac{1}{4\pi}\left| \int_{{  U } }\int_{{  U}  }  {G}(x,y)\Big(\xi_1(x)\xi_1(y)-\xi_2(x)\xi_2(y)\Big)\,dxdy \right|
\\&\quad \quad + 
 \int_{\mathbb{R}^{3}\setminus {U }}\xi_1\mathcal{G}[\xi_1]\,dx 
+  \int_{\mathbb{R}^{3}\setminus {U }}\xi_2\mathcal{G}[\xi_2]\,dx.
\end{split}\end{equation}
\end{lem}

\begin{proof}

Let $\xi \in \left(L^{1}_w\cap L^2\cap L^{1}\right)(\mathbb{R}^{3})$ be  non-negative and axi-symmetric.  
By setting 
$\psi=\mathcal{G}[\xi]$, we decompose
\begin{align*}
 4\pi  E[\xi]= {\int}2\pi\psi\xi\,dx= \int_{{U} }+ \int_{\mathbb{R}^{3}\setminus {U}}=:I+II,\quad\mbox{and}
\end{align*}
\begin{align*}
 {I} 
=\int_{{  U} }\xi(x){\int_{\mathbb{R}^3}}G(x,y)\xi(y)\,dy\,dx
=\int_{{ U} }\int_{{ U}  }
+\int_{{ U} }\int_{\mathbb{R}^3\setminus { U}}
=:I_{1}+I_{2}.
\end{align*}
From the symmetry of  $G(x,y)=G(y,x)>0$, 
we estimate the last term $I_2$ by
\begin{align*}
I_{2}&=\int_{{  U} }\xi(x){\int_{\mathbb{R}^3\setminus { U}}}G(x,y)\xi(y)\,dy\,dx
 = \int_{\mathbb{R}^3\setminus { U}}\xi(x){\int_{  { U}}}G(x,y)\xi(y)\,dy\,dx\\
  &\leq  \int_{\mathbb{R}^3\setminus { U}}\xi(x){\int_{\mathbb{R}^3} }G(x,y)\xi(y)\,dy\,dx
= \int_{\mathbb{R}^3\setminus { U}}2\pi \psi(x)\xi(x)\,dx= II.
\end{align*} 
Thus we get $$4\pi E[\xi]\leq    I_{1}
+2II,$$ which gives  
\begin{align*}
0\leq\, & E[\xi]-\frac 1 {4\pi}\int_{{ U} }\int_{{U}  }  G(x,y)\xi(x)\xi(y)\,dydx 
\leq  \int_{\mathbb{R}^{3}\setminus {U}}\psi(x)\xi(x)\,dx.
\end{align*} By  applying the above estimate into  any pair $(\xi_1,\xi_2)$, 
we obtain
\eqref{est_en_diff}.
 
\end{proof} 
We check that the kernel $G(x,y)$ is locally square integrable due to its logarithm
 behavior \eqref{log_beha} near $x=y$. This lemma  will be used not only in this section  and but also in  Section \ref{sec_cpt}.
\begin{lem}\label{lem_g_l_2} The kernel $G(x,y)$ satisfies
$$G\in L^2_{loc}(\mathbb{R}^3\times\mathbb{R}^3).$$ In particular,
 we have 
 \begin{align}\label{est_g_l_2}
\int_{B_{M}(r,z)}  r'|G(r,z,r',z')|^2\,dr'dz'  \lesssim  (Mr^{4}+M^{7/2}r^{3/2}),\quad M>0, \quad(r,z)\in\Pi,
\end{align}  where 
$B_M(r,z)=\{(r',z')\in\Pi\,|\, |(r,z)-(r',z')|<M\}$ as defined in 
\eqref{defn_ball}.
\end{lem}
 \begin{proof}
 
We use the estimate \eqref{est_F} with $\tau=1/4$:
\begin{align*}\label{G_est_2}
G(r,z,r',z')\lesssim \frac{(rr')^{3/4}}{|(r,z)-(r',z')|^{1/2}}.
\end{align*} Thus we estimate, by change of variables 
$r'/r=\tilde r,\quad z'/r=\tilde z$,
\begin{align*} 
&\int_{B_{M}(r,z)}  r'|G(r,z,r',z')|^2\,dr'dz'  \lesssim \int_{B_{M}(r,z)}  r'\frac{(rr')^{3/2}}{|(r,z)-(r',z')|}\,dr'dz'  =\int_{B_{M}(r,0)}  r'\frac{(rr')^{3/2}}{|(r,0)-(r',z')|}\,dr'dz' \\&
=r^5\int_{B_{Mr^{-1}}(1,0)
 }   \frac{(\tilde r)^{5/2}}{|(1,0)-(\tilde r,\tilde z)|}\,d\tilde rd\tilde z \leq r^5(1+Mr^{-1})^{5/2}\int_{B_{Mr^{-1}}(1,0)
 }  \frac{1}{|(1,0)-(\tilde r,\tilde z)|}\,d\tilde rd\tilde z\\
& 
\lesssim r^5(1+M^{5/2}r^{-5/2}) 
\int_0^{{Mr^{-1}}}
 1 d\rho =  (Mr^{4}+M^{7/2}r^{3/2}).
\end{align*}  
  Thanks to the above estimate, we have $G\in L^2_{loc}(\mathbb{R}^3\times\mathbb{R}^3).$ Indeed, for any $M>0$, we estimate
 \begin{align*} 
&\int_{B_M(0)}\int_{B_M(0)} |G(x,y)|^2dydx 
 =4\pi^2 \int_{B_{M}(0,0)} r\int_{B_{M}(0,0)} r'|G(r,z,r',z')|^2\,dr'dz'  drdz\\
 & \lesssim  \int_{B_{M}(0,0)} r\int_{B_{2M}(r,z)} r'|G(r,z,r',z')|^2\,dr'dz'  drdz 
    \lesssim  \int_{B_{M}(0,0)} r (Mr^{4}+M^{7/2}r^{3/2})  drdz\\&\lesssim 
   M^6\int_{B_{M}(0,0)} 1  drdz\lesssim M^8.
 \end{align*}

 \end{proof}
 
 \subsection{Existence of a maximizer}\ \\
 
 In this subsection, our goal is to prove the following existence lemma:
 \begin{lem}\label{lem_exist}
  Let $0<\mu<\infty$. Then
 $\mathcal{S}''_{\mu}\neq \emptyset.$
\end{lem}

In order to prove Lemma \ref{lem_exist},   we introduce Steiner symmetrization (symmetrical rearrangement about the plane  $\{z=0\}$) with its property as in \cite{FT81}. Here we say that a   non-negative  
 function $f$ on $\Pi$ satisfies the \textit{monotonicity} condition (about the plane  $\{z=0\}$) if \begin{equation}\label{cond_sym}
\begin{aligned}
&f(r,z)=f(r,-z),\quad(r,z)\in\Pi\quad\mbox{and} \\
&\mbox{for each fixed}\,\,\, r>0,\quad f(r,z)\ \textrm{is a  non-increasing  function of}\,\,\, z\,\,\,\mbox{for}\,\,\, z>0. 
\end{aligned}
\end{equation}

\begin{prop}[Steiner symmetrization]\label{prop_steiner}
Let $p\in[2,\infty]$. For   axi-symmetric   $\zeta\geq 0$ satisfying $\zeta\in (L^1_w\cap L^{p}\cap L^{1})(\mathbb{R}^{3})$, there exists an axi-symmetric $\zeta^{*}\geq 0$ 
 satisfying {the monotonicity }condition \eqref{cond_sym},
\begin{equation}\begin{split}\label{steiner_rearr}  
&||\zeta^{*}||_{q}=||\zeta||_{q},\quad 1\leq q\leq p,    \\
&||r^2\zeta^{*}||_{1}=||r^2\zeta||_{1},\quad \mbox{ and} \\
&\int_{\{x\in\mathbb{R}^3\,|\, \zeta^*(x) \in I\}}\zeta^*\,dx= \int_{\{x\in\mathbb{R}^3\,|\,|\zeta(x) \in I\}}\zeta\,dx \quad\mbox{for any  interval}\quad I\subset \mathbb{R}.
\end{split}\end{equation}
In particular, it satisfies
\begin{equation}\label{energy_steiner}
 E[\zeta^{*}]\geq E[\zeta].
\end{equation}
\end{prop}
  
\begin{proof}
 The   symmetrical   rearrangement $\zeta^*$ of $\zeta$ about the plane  $z=0$  satisfies the  properties in \eqref{steiner_rearr} (e.g. see \cite[Section 3.3]{LiebLoss}). The inequality \eqref{energy_steiner}   is a consequence of   
   Riesz rearrangement inequality  \cite{Riesz} (or see \cite[p84]{LiebLoss}).
 We also refer to \cite[Appendix I]{FB74}
  which is an adaptation of P\'{o}lya-Szeg\"{o}
inequality \cite{PoSz_book}.
\end{proof}


 

The following lemma says that the kinetic energy is concentrated in a bounded domain when $\xi$ satisfies  {the monotonicity } condition \eqref{cond_sym}.  We present  its proof in Appendix \ref{app_en_conv}
   while a similar estimate can be found  
  in
   \cite[Lemma 3.5]{FT81}.
\begin{lem}\label{lem_stream_AR}
Let $\xi\in (L^1_w \cap L^{2}\cap L^{1})  (\mathbb{R}^{3})$ be an  axi-symmetric nonnegative function satisfying {the monotonicity } condition \eqref{cond_sym}. Then we have
\begin{equation}\label{est_stream_AR}
\begin{aligned}
\int_{\mathbb{R}^{3}\backslash Q} \xi\mathcal{G}[\xi] \dd x 
\lesssim  \left(\frac{1}{\sqrt A}  +
 \frac{1}{R^2}   \right)  \left(\|\xi\|_{L^1\cap L^2}  +\|r^2\xi\|_{1} \right)^2, 
\end{aligned}
\end{equation} where
$$Q=Q_{A,R}=\{x\in\mathbb{R}^3\, |\,  \quad |z|<AR, \quad r <R\}$$
 provided 
$R\geq 1$ and $A\geq1$. 
\end{lem}

The following lemma ensures  convergence of the energy for any bounded  sequence in $(L^1_w \cap L^{2}\cap L^{1})$ satisfying {the monotonicity} condition \eqref{cond_sym}.
This lemma was implicitly appeared in the proof of \cite[Theorem 2.1]{FT81} while it was
not explicitly written in the form of a lemma.  
We will use the lemma both in this section and in Section \ref{sec_cpt}.

\begin{lem}\label{lem_energy_conv}
Let $\{\xi_n\}_{n=1}^\infty$ be a sequence of axi-symmetric non-negative functions on $\mathbb{R}^3$ such that 
\begin{align*}
&\xi_n\quad \mbox{satisfies the monotonicity conditon }\eqref{cond_sym}\quad\mbox{for each}\quad n,\\
&\sup_{n  }\left\{||\xi_n||_{L^1\cap L^2}+||r^2\xi_n||_{1}  \right\}<\infty,\quad\mbox{and}\\
&\xi_n\rightharpoonup \xi\quad \textrm{in}\ L^{2}(\mathbb{R}^{3})\quad \textrm{as}\quad n\to\infty\quad \mbox{for some non-negative axi-symmetric}\quad \xi\in L^{2}(\mathbb{R}^{3}).
\end{align*} 
 Then we have convergence of  the energy:
\begin{align*}
E[\xi_{n}]\to E[\xi]\quad \textrm{as}\ n\to\infty.
\end{align*}
\end{lem}
   \begin{proof}
  First, we observe, by the weak convergence $\xi_n\rightharpoonup \xi$ in $L^{2}(\mathbb{R}^{3})$, 
$$  \|\xi\|_{L^1\cap L^2}+\|r^2\xi\|_{1}\leq C\quad \mbox{for some }\, C>0.$$
We set a bounded domain
$$Q=Q_{A,R}=\{x\in\mathbb{R}^3\, |\,  \quad |z|<AR, \quad r <R\}.$$
  for $R\geq 1$ and $A\geq1$.
  Then, by \eqref{est_en_diff} of Lemma \ref{lem_en_diff}, we have
  \begin{equation*}\begin{split}
\left| E[\xi_n]-E[\xi]\right|
&\leq  \frac{1}{4\pi}\left| \int_{{  Q } }\int_{{  Q}  }  {G}(x,y)\Big(\xi_n(x)\xi_n(y)-\xi(x)\xi(y)\Big)\,dxdy \right|
\\&\quad \quad + 
 \int_{\mathbb{R}^{3}\setminus {Q }}\xi_n\mathcal{G}[\xi_n]\,dx 
+  \int_{\mathbb{R}^{3}\setminus {Q }}\xi\mathcal{G}[\xi]\,dx.
\end{split}\end{equation*}
Since  $\xi_n$  satisfies {the monotonicity }condition \eqref{cond_sym} for each $n\geq 1$, so does $\xi$. Thus
   we can estimate, by \eqref{est_stream_AR} of Lemma \ref{lem_stream_AR},
   \begin{equation*}
\begin{aligned}
&\int_{\mathbb{R}^{3}\backslash Q}\xi\mathcal{G}[\xi]\dd x 
\lesssim  \left(\frac{1}{\sqrt A}  +
 \frac{1}{R^2}   \right)\quad\mbox{and}\quad \sup_n\int_{\mathbb{R}^{3}\backslash Q}\xi_n\mathcal{G}[\xi_n]\dd x 
\lesssim \left(\frac{1}{\sqrt A}  +
 \frac{1}{R^2}   \right).
\end{aligned}
\end{equation*}
Since $G(x,y)\in L^{2}(Q\times Q)$ by Lemma \ref{lem_g_l_2}  and $\xi_n(x)\xi_n(y)\rightharpoonup\xi(x)\xi(y)$ in $L^{2}(Q\times Q)$, sending $n\to\infty$ and $A,R\to\infty$ imply  convergence of the energy.    \\

  \end{proof}

Now we are ready to prove Lemma \ref{lem_exist}.
\begin{proof}[Proof of Lemma \ref{lem_exist}]
Let $\{\xi_n\}\subset \mathcal{P}''_{\mu}$ be a   sequence  satisfying $E[\xi_n]\nearrow \mathcal{I}''_\mu$. By the Steiner symmetrization (Proposition \ref{prop_steiner} with $p=\infty$), we 
obtain the corresponding sequence $\{\xi_n^*\}$ in $\mathcal{P}''_\mu$ satisfying {the monotonicity }condition \eqref{cond_sym}.
 Since $\{\xi^*_n\}$ is uniformly bounded in $L^{2}(\mathbb{R}^3)$ by interpolation between $L^1$ and $L^\infty$, 
by choosing a subsequence (still denoted by $\{\xi^*_n\}$ for simplicity), there exists a non-negative axi-symmetric function $\xi\in L^{2}(\mathbb{R}^3)$ satisfying
$\xi_n^*\rightharpoonup \xi$ in $L^{2}$. 
Hence
  we can apply Lemma \ref{lem_energy_conv} for $\{\xi^*_n\}$ to get 
  \begin{align*}
\lim_{n\to \infty}E[\xi^*_n]=E[\xi].
\end{align*} 
  Since the weak-limit $\xi$ lies on $\mathcal{P}''_{\mu}$ and 
\begin{align*}
\mathcal{I}''_\mu\geq \lim_{n\to \infty}E[\xi^*_n]\geq \lim_{n\to \infty}E[\xi_n]
 =\mathcal{I}''_\mu,
\end{align*} we conclude 
  $\xi\in \mathcal{S}''_{\mu}$.  
\end{proof}

\subsection{Properties of the  set   of  maximizers}\ \\
 
Next we show $\mathcal{S}'_{\mu}=\mathcal{S}''_{\mu}$.
\begin{lem}\label{lem_iden_tilde} Let $\mu\in(0,\infty)$. Then
\begin{align*}
&\mathcal{S}'_{\mu}=\mathcal{S}''_{\mu}\neq \emptyset\quad\mbox{and}\quad
 \mathcal{I}'_{\mu}=\mathcal{I}''_{\mu}.
\end{align*}
\end{lem}
 \begin{proof}

Let us take any  $\xi\in \mathcal{S}''_{\mu}$ by recalling
$\mathcal{S}''_{\mu}\neq \emptyset$ from Lemma \ref{lem_exist}. We claim  
\begin{equation}\label{claim_iden_tilde}
  \xi\in \mathcal{P}'_{\mu}.
\end{equation}
 For a contradiction,  
we suppose 
\begin{equation*}\label{assum1} 
\mu>\frac 1 2 \int r^2 \xi dx=:\mu_0,\end{equation*} i.e. we assume $\xi \in \mathcal{P}''_\mu\setminus \mathcal{P}'_\mu$.   
From $\mathcal{I}''_\mu>0$ by Lemma \ref{lem_nontrivial_extra}, we have 
$\xi\nequiv 0$ so 
$\mu_0>0$. 
 If we define
 the (relative) translation ${\xi_\tau}$ of $\xi$ away from $z-$axis  for $\tau>0$
  by
\begin{equation}\label{tau_trans}
 \xi_\tau(r,z)=\begin{cases}&\frac{r-\tau}{r}\xi(r-\tau,z)\quad\mbox{for}\quad r\geq\tau,\\
&0\quad\mbox{for}\quad 0< r<\tau,
\end{cases}\end{equation} then we have
 \begin{equation}\begin{split}\label{trans_comp} 
 &0\leq \xi_\tau\leq \sup_{r\geq \tau}\left(\frac{r-\tau}{r}\right)\cdot\|\xi\|_\infty\leq 1,\\
 &
\| \xi_\tau\|_1=2\pi\int r \xi_\tau drdz=
2\pi\int_{r\geq \tau} (r-\tau) \xi(r-\tau,z) drdz
=2\pi\int_{\Pi} r \xi drdz=\|\xi\|_1\leq 1,\\
&\frac 1 2 \int r^2\xi_\tau\,dx
=\pi   \int_{r\geq\tau} (r-\tau) r^2  \xi(r-\tau,z) drdz=\pi   \int_{\Pi}   (r+\tau)^2 r \xi drdz
=\mu_0+\frac 1 2    \int  (2\tau r+\tau^2)  \xi dx.
\end{split}\end{equation} Thus
 we can take some constant $\tau>0$ such that 
  $ \frac 1 2 \int r^2\xi_\tau\,dx=\mu,$ 
i.e.  we have $ \xi_\tau\in\mathcal{P}'_\mu$. On the other hand, we observe $$E[ \xi_\tau]> E[\xi]$$  by exploiting  the form \eqref{form_F} of the kernel $G$.
 Indeed, we observe, for $(r,z),(r',z')\in\Pi$,
 \begin{equation}\begin{split}\label{exploit} 
 G(r+\tau,z,r'+\tau,z') &=
 \frac{\sqrt{(r+\tau)(r'+\tau)}}{2\pi}\,F\left(\frac{(r-r')^2+(z-z')^2}{(r+\tau)(r'+\tau)}\right)\\
 &
 > \frac{\sqrt{rr'}}{2\pi}\,F\left(\frac{(r-r')^2+(z-z')^2}{rr'}\right)=G(r,z,r',z')
 \end{split}\end{equation} because
$F(\cdot)$ is strictly decreasing  (see Subsection \ref{subsec_axi-sym}).  Thus, we have
 \begin{equation}\begin{split}\label{exploit2}&\frac 1 \pi E[ \xi_\tau]
 =\iint rr'G(r,z,r',z')\xi_\tau(r',z')\xi_\tau(r,z)dr'dz'drdz\\
 &
 =\iint_{r>\tau,\, r'>\tau}  (r-\tau)(r'-\tau)G(r,z,r',z')\xi (r'-\tau,z') \xi (r-\tau,z)dr'dz'drdz\\
  &
 =\iint rr'G(r+\tau,z,r'+\tau,z')\xi (r',z')\xi (r,z)dr'dz'drdz\\
  & >\iint rr'G(r,z,r',z')\xi(r',z')\xi(r,z)dr'dz'drdz=\frac 1 \pi E[ \xi],
\end{split}\end{equation} where the last inequality comes from
 \eqref{exploit} and non-triviality of $\xi\geq 0$. 
Hence we get 
$$ \mathcal{I}'_\mu\geq E[ {\xi}_\tau]>E[\xi]=\mathcal{I}''_\mu,$$ which contradicts to
   $\mathcal{I}'_{\mu}\leq  \mathcal{I}''_{\mu}$ obtained from Lemma \ref{lem_nontrivial_extra}. Hence  $\mu_0=\mu$ so we get
the claim \eqref{claim_iden_tilde}.   
 Since
 $\xi\in\mathcal{S}''_{\mu}\cap \mathcal{P}'_{\mu}$ implies
  $$\mathcal{I}'_{\mu}\leq  \mathcal{I}''_{\mu}=E[\xi]\leq  \mathcal{I}'_{\mu},$$ we get
 $  \mathcal{I}'_{\mu}=\mathcal{I}''_{\mu}$ and  $\xi \in \mathcal{S}'_{\mu}$.
 In sum, we have shown  $   \mathcal{S}''_{\mu}\subset \mathcal{S}'_{\mu}$.
Due to 
$ \mathcal{P}'_{\mu} \subset \mathcal{P}''_{\mu}$ and 
$  \mathcal{I}'_{\mu}=\mathcal{I}''_{\mu}$, we get 
  $   \mathcal{S}'_{\mu}= \mathcal{S}''_{\mu}.$\ \\
 
 \end{proof}

 \color{black}
Now we show $\mathcal{S}_{\mu}=\mathcal{S}'_{\mu}$.
\begin{lem}\label{key_lemma} Let $\mu\in(0,\infty)$. Then
\begin{align*}
&\mathcal{S}_{\mu}=\mathcal{S}'_{\mu}\neq \emptyset\quad\mbox{and}\quad
 {\mathcal{I}}_{\mu}=\mathcal{I}'_{\mu}.
\end{align*}

\end{lem}
\begin{proof}
Let us  take any $\xi\in \mathcal{S}'_{\mu}$ by recalling $ \mathcal{S}'_{\mu}\neq \emptyset$ by Lemma \ref{lem_iden_tilde}.
In order to show $\xi\in{\mathcal{S}}_{\mu} $,
we first show $\xi\in{\mathcal{P}}_{\mu}$. In other words, we claim
\begin{equation}\label{claim_key_lem}
|\{x\in\mathbb{R}^3\,|\, 0<\xi<1\}|=0.
\end{equation} 
For a contradiction, let us suppose 
$|\{0<\xi<1\}|>0.$
Then there exists $\delta_0>0$ such that
\begin{equation}\label{contra_nontrivial}
|\{\delta_0\leq\xi\leq 1-\delta_0\}|>0. 
\end{equation}
First, we take axi-symmetric compactly supported functions $h_1,h_2\in L^{\infty}(\mathbb{R}^{3})$ such that for $i=1,2$, $$\textrm{spt}\ h_i\subset \{\delta_0\leq \xi\leq 1-\delta_0\},$$
\begin{align*}
{\int}h_1(x)\dd x=1,\quad \frac 1 2{\int}r^2h_1(x)\dd x=0,\\
{\int}h_2(x)\dd x=0,\quad \frac  1 2 {\int}r^2h_2(x)\dd x=1.
\end{align*}
We may consider them as  a basis for our two constraints(mass, impulse) problem. 
Let $h\in L^{\infty}(\mathbb{R}^{3})$ be   any axi-symmetric compactly supported function  such that 
$$\textrm{spt}\ h \subset \{ \xi\leq 1-\delta_0\}\quad\mbox{and}\quad h\geq 0\quad \mbox{on}\quad \{0\leq \xi< \delta_0\}.$$ We set
\begin{align}\label{defn_eta}
\eta:=h-\left({\int}h \dd x  \right)h_1-\left(\frac 1 2{\int}r^2 h \dd x \right)h_2
\end{align}
so that
$\eta\in (L^1_w\cap L^\infty\cap L^1)(\mathbb{R}^3)$ is axi-symmetric,
 $\int \eta\dd x=0$, and $\frac 1 2\int r^2\eta\dd x=0$.  
We consider 
$(\xi+\epsilon\eta)$ for $\epsilon>0$. Thanks to the above construction of $\eta$, we know 
$$(\xi+\epsilon\eta)\in L^\infty(\mathbb{R}^3),\quad \frac 1 2 \int r^2 (\xi+\epsilon\eta) dx=\frac 1 2 \int r^2 \xi dx= \mu,\quad \int (\xi+\epsilon\eta) dx=\int \xi dx\leq 1.$$ We claim $$0\leq (\xi+\epsilon\eta)\leq  1\quad\mbox{for small}\quad \epsilon>0.$$
We observe  
$(\xi+\epsilon\eta) =\xi$ on the set  $\{\xi>1-\delta_0\}$ while
 we have,  for sufficiently small $\epsilon>0$, on the set $\{\delta_0\leq\xi\leq 1-\delta_0\}$, 
$$1\geq 1-\delta_0+\epsilon\|\eta\|_\infty\geq (\xi+\epsilon\eta)\geq \delta_0-\epsilon||\eta||_{\infty}\geq 0.$$ 
On the remainder set $\{0\leq \xi< \delta_0\}$, due to $\eta=h\geq 0$, 
 we get $(\xi+\epsilon\eta)\geq 0$ and, for small $\epsilon>0$, 
 we know  $(\xi+\epsilon\eta)\leq 1$. Hence we have shown  $$(\xi+\epsilon\eta)\in \mathcal{P}'_{\mu}\quad \mbox{for small}\quad \epsilon>0.$$ 
By the assumption $\xi\in \mathcal{S}'_{\mu}$,
we have,  for  small $\epsilon>0$,
$$0\geq  \frac{E[\xi+\epsilon\eta]-E[\xi]}{\epsilon}.$$ By taking the limit $\epsilon\searrow 0,$  we have, for $\psi=\mathcal{G}[{\xi}],$
\begin{align*}
0\geq  \frac 1 2 \int \mathcal{G}[\xi]\eta\,dx +\frac 1 2 \int \xi\mathcal{G}[\eta]\,dx=
\frac 1 2 \int\psi\eta\,dx +\frac 1 2 \int \mathcal{G}[\xi]\eta\,dx=  {\int}\psi \eta\dd x,
\end{align*} 
where we used the symmetry of the kernel $G$.
On the other hand, by the definition of $\eta$,  we have
\begin{align*}
 {\int}\psi \eta\dd x= {\int}\psi h\dd x - \left({\int}\psi h_1\dd x\right)\left({\int} h\dd x\right)-\left({\int}\psi h_2\dd x\right)\left(\frac 1 2{\int}r^2 h\dd x\right).
\end{align*}
By denoting $\beta:=\left({\int}\psi h_1\dd x\right)$, $\alpha:=\left({\int}\psi h_2\dd x\right)$ and $\tilde{\psi}:=\psi-\frac 1 2\alpha r^2-\beta$, we obtain
\begin{align*}
0&\geq {\int}\psi h\dd x-\beta\left({\int} h\dd x\right)-\alpha\left(\frac 1 2{\int}r^2 h\dd x\right)
=
{\int}
\tilde{\psi} h\dd x 
=\int_{0\leq \xi< \delta_0}\tilde{\psi} h\dd x+\int_{1-\delta_0\geq \xi\geq  \delta_0}\tilde{\psi} h\dd x,
\end{align*} where the decomposition is due to the assumption 
$\textrm{spt}\ h \subset \{ \xi\leq 1-\delta_0\}$.
Since $h$ is an arbitrary function satisfying $h\geq 0$ on the set $\{0\leq \xi<\delta_0\}$, we have
\begin{equation}\label{contra}
\begin{aligned}
\tilde{\psi}&=0\quad a.e.\quad  \textrm{on}\quad 
\{\delta_0\leq\xi\leq 1-\delta_0\},\\
\tilde{\psi} &\leq 0\quad a.e.\quad  \textrm{on}\quad \{0\leq \xi< \delta_0\}. 
\end{aligned}
\end{equation}
On the other hand, since 
$\psi\in H^2_{loc}(\Pi)$ by Lemma \ref{lem_en_iden_origin}, we have
$\tilde{\psi}\in H^2_{loc}$, which implies, on the set $\{\tilde{\psi}=0\}$,  $$\nabla \tilde{\psi}=0\quad a.e. 
\quad \mbox{and}\quad -\frac{1}{r^2}\mathcal L\tilde{\psi}=0\quad a.e.$$
Since $-{r^{-2}}\mathcal L\tilde{\psi}=-{r^{-2}}\mathcal L\psi=\xi$, we get $$\xi =0 \quad a.e.\quad \mbox{on}\quad \{\tilde{\psi}=0\}.$$ 
Thus, by \eqref{contra}, we  get
$$\xi =0 \quad a.e.\quad\mbox{on}  \quad 
\{\delta_0\leq\xi\leq 1-\delta_0\}.$$ It 
 contradicts the assumption \eqref{contra_nontrivial}. Hence we get the claim 
\eqref{claim_key_lem}, 
  which implies 
$\xi\in{\mathcal{P}}_{\mu}$.
 Due to 
$\xi\in\mathcal{S}'_{\mu}$, we get $\xi\in{\mathcal{S}}_{\mu}$.
In sum, we have shown 
$\mathcal{S}'_{\mu}\subset \mathcal{S}_{\mu}$. 
As in the last part of the proof of Lemma \ref{lem_iden_tilde}, 
 we conclude
 $ {\mathcal{I}}_{\mu}= \mathcal{I}'_{\mu}$  and 
$ {\mathcal{S}}_{\mu}= \mathcal{S}'_{\mu}.$ 
 
\end{proof}

Now we are ready to prove Theorem \ref{thm_exist_max}.

\begin{proof}[Proof of Theorem \ref{thm_exist_max}]
By  Lemmas \ref{lem_nontrivial_extra}, \ref{lem_iden_tilde} and \ref{key_lemma}, and the scaling argument \eqref{scaling},  we get the theorem. 
   
   \end{proof}

Before closing this section, for a later use, we show that ${\mathcal{I}}_{\mu}$ is strictly increasing in the variable $\mu>0$.
\begin{lem}
\label{lem_strict}
 Let $0<\mu_0<\mu<\infty$. Then
\begin{align*}
&{\mathcal{I}}_{\mu_0}<{\mathcal{I}}_{\mu}.
\end{align*}
\end{lem}

\begin{proof}
 
We take any function $\xi\in {\mathcal{S}}_{\mu_0}$ by recalling $ {\mathcal{S}}_{\mu_0}\neq\emptyset$ 
from Theorem \ref{thm_exist_max}.
 We can find some $\tau>0$ such that  the relative translation ${\xi_\tau}$ of $\xi$ 
 defined by \eqref{tau_trans}
lies on ${\mathcal{P}}'_\mu$ as in the computation
\eqref{trans_comp}. We note $E[\xi_\tau]>E[\xi]$ by 
 \eqref{exploit2}. Hence we get
$${\mathcal{I}}_{\mu_0}
=E[\xi]<E[\xi_\tau]\leq {\mathcal{I}}'_\mu={\mathcal{I}}_\mu.$$\end{proof} 
\color{black}
\section{Steady vortex rings from  maximizers}\label{sec_exist_vortex}

In this section, our goal is to show Theorem
\ref{thm_max_is_ring} below, which is needed when proving 
Theorem \ref{thm_cpt} in Section \ref{sec_cpt} and
Theorem \ref{thm_uniq} in Section 
\ref{sec_uniq_proof_hill}.

\subsection{Every maximizer produces a vortex ring.}\ \\
 \begin{thm}\label{thm_max_is_ring}
 For $0<\mu,\nu,\lambda<\infty$, each element $\xi$ of the set ${\mathcal{S}}_{\mu,\nu,\lambda}$ satisfies

\begin{equation}\label{eq_W}
\begin{aligned}
&\xi=\lambda{{1}}_{ \{\Psi>0\}}\quad\mbox{a.e.} 
\quad\mbox{for}\quad \Psi= \mathcal{G}[\xi]-\frac{1}{2}Wr^2-\gamma
\end{aligned}
\end{equation} 
for some constants $$W>0,\quad \gamma\geq 0,$$ which are uniquely determined by $\xi$. Moreover,
$\xi$ is compactly supported in $\mathbb{R}^3$.
\end{thm}
 
 \begin{rem}
 As a consequence of \eqref{eq_W}, we obtain a steady vortex ring $\xi$ since we have \eqref{jac} for the choice of
$f(s)=\lambda f_H(s)=\lambda 1_{\{s>0\}}$. In other words, 
 the function $\xi(x-tWe_z)$ is an exact solution of  \eqref{3d_Euler_eq}.
 \end{rem}


We split the proof  into 3 steps. 
First,  Proposition \ref{prop_exist_vortex} shows the existence (and uniqueness ) of such a pair of constants $W\geq 0$ and $\gamma\geq 0$ for each maximizer $\xi$. 
Second, Proposition \ref{prop_en_W} proves that the constant $W$ is positive.
Lastly, Proposition \ref{prop_cpt_supp} gives compactness of the  (essential) support of $\xi$.

\subsection{Exceptional points of a measurable set}\label{sec5_step1}\ \\

As the first step, we prove that \eqref{eq_W} holds for some unique non-negative $W,\gamma$:
 \begin{prop}\label{prop_exist_vortex}
 For $0<\mu,\nu,\lambda<\infty$, each element $\xi$ of the set ${\mathcal{S}}_{\mu,\nu,\lambda}$ satisfies

\begin{equation}\label{eq_W_lem}
\begin{aligned}
&\xi=\lambda{{1}}_{ \{\Psi>0\}}\quad\mbox{a.e.} 
\quad\mbox{for}\quad \Psi= \mathcal{G}[\xi]-\frac{1}{2}Wr^2-\gamma
\end{aligned}
\end{equation} 
for some constants $$W, \gamma\geq 0,$$ which are uniquely determined by $\xi$.  \ \\
\end{prop}

\begin{proof}
Let $\mu\in(0,\infty)$. By the scaling argument \eqref{scaling}, it is enough to consider the case $\nu=\lambda=1$.  \\
 
We prove the uniqueness of $W,\gamma$ first by assuming the existence for a moment. 
Let $\xi\in {\mathcal{S}}_\mu$ and  set $\psi=\mathcal{G}[\xi]$. Suppose that there exist constants $W,\gamma\in\mathbb{R}$ satisfying
\begin{equation*}
\begin{aligned}
&\xi={{1}}_{ \{(\psi-(1/2) Wr^2-\gamma)>0\}}\quad\mbox{a.e.} 
\end{aligned}
\end{equation*} 
We set ${A}=\{x\in\mathbb{R}^3\,|\,(\psi(r,z)-(1/2)Wr^2-\gamma)>0\}$. Since $\Psi:=(\psi-(1/2) Wr^2-\gamma)$ is continuous on $\overline\Pi$ by Lemma \ref{lem_en_iden_origin}, 
the axi-symmetric  set $ {A}$ is open in $\mathbb{R}^3$. 
Clearly, we know
 $$0<|{A}|<\infty$$ because
$E[\xi]=\mathcal{I}_\mu>0$ by Lemma \ref{lem_nontrivial_extra} and 
  $|A|=\|\xi\|_1\leq\nu$.
We take any two points $y', y''\in\mathbb{R}^3$ from the boundary $\partial {A}$ satisfying $r'>r''>0$. Here we match $y', y''\in\mathbb{R}^3$ into $(r',z'), (r'',z'')\in\Pi$, respectively.
Since  $y', y''\in\partial {A}$ implies $\Psi(y')=\Psi(y'')=0$,
  we get
 $$\psi(y')-\frac 1 2Wr'^2-\gamma=0\quad\mbox{and}\quad \psi(y'')-\frac 1 2W r''^2-\gamma=0.$$ By solving these equations about $W$ and $\gamma$, we have
 \begin{equation}\label{def_W_gamma_uniq}W=2\frac{\psi(y')-\psi(y'')}{r'^2-r''^2}\quad \mbox{and}\quad
\gamma=\frac{r'^2}{r'^2-r''^2}\psi(y'')-\frac{r''^2}{r'^2-r''^2}\psi(y'),\end{equation} 
which produces the uniqueness of such $W,\gamma\in\mathbb{R}$. \\

 It remains to show the existence of such constants $W, \gamma\geq 0$ satisfying \eqref{eq_W_lem}.
Let $\xi\in {\mathcal{S}}_\mu$.  Due to 
$\mathcal{S}_\mu\subset \mathcal{P}_\mu$, we have,
by the definition \eqref{defn_adm} of the class $\mathcal P_\mu$,
  $$\xi={{1}}_{{A}}\quad\mbox{for some axi-symmetric measurable subset}\quad A\subset \mathbb{R}^3.$$  
 As before, we know $|{A}|\in(0,\infty)$. 
Our goal is to find some $W,\gamma\geq0$ satisfying 
\begin{equation}\label{goal_exist_w_gamma}
A=\{x\in\mathbb{R}^3\,|\, \mathcal{G}[\xi](x)-(1/2)Wr^2-\gamma>0\} \quad \mbox{a.e.}
\end{equation}
As observed when proving uniqueness, we might expect that any two points $y',y''$ on the boundary $\partial A$ with different distances toward the axis (i.e. $r'\neq r''$)  play an  important role when  verifying 
\eqref{goal_exist_w_gamma}.
 However,
  we know only that $A$ is  measurable (until showing \eqref{goal_exist_w_gamma})
  so the set  is defined up to measure zero. Thus,
   the notion of topological boundary
    is not useful as before. 
    To overcome the difficulty, we use  so-called 
\textit{metrical} boundary, which is the set of     
    \textit{exceptional} points.    
These terminologies were used in  \cite[p78]{Croft}  and \cite[p765]{Szenes} (also see \cite{Kol}).
\begin{defn}\label{defn_excep} 
Let  $\Omega$ be a given fixed open subset of 
$\mathbb{R}^N$, $N\geq1$. For any (Lebesgue) measurable subset  $U\subset\Omega$, 
 we define the density $\mathcal{D}_e(U)$ of the set $U$ by the collection of points $x\in \Omega$   
  such that
$$ {\liminf}_{r\rightarrow 0}\frac{|B_r(x)\cap U|}{|B_r(x)|}=1,$$ where  $B_r(x)=\{y\in\Omega \, | \, |x-y|<r\}$. Since
$0\leq {|B_r(x)\cap U|}/{|B_r(x)|}\leq 1,\,r>0$ holds,
$$x\in \mathcal{D}_e(U)\quad \mbox{if and only if}\quad {\lim}_{r\rightarrow 0}\frac{|B_r(x)\cap U|}{|B_r(x)|}\quad\mbox{exists and the limit value is equal to } 1.$$  
Similarly, we define the dispersion $\mathcal{D}_i(U)$ by 
the set of points  $x\in\Omega$ such that
$$\limsup_{r\rightarrow 0}\frac{|B_r(x)\cap U|}{|B_r(x)|}= 0.$$ 
  We call   the set $$\Omega\setminus (\mathcal{D}_e(U)\cup\mathcal{D}_i(U))=:\mathcal{E}(U)$$ 
the \textit{metrical} boundary of $U$. Any  point of the   {metrical} boundary    $\mathcal{E}(U)$ is called 
   an \textit{exceptional point} of $U$. We note that
  $\mathcal{E}(U)=\mathcal{E}(V)$ when $V$  is equal to $U$ a.e. 
\end{defn} 

The   Lebesgue density theorem says
\begin{lem}\label{lem_lebesgue} For any measurable subset $U$ of  $ \,\mathbb{R}^N,\,N\geq 1$, almost every point of $U$ is a point of the density $\mathcal{D}_e(U)$. 
\end{lem}
It is an immediate corollary of the well-known Lebesgue differentiation theorem.
For example, a proof can be found in  \cite[Theorem 7.13]{book_zygmund}.
The above lemma applied to the complement  $U^c$ says that almost every point of $U^c$ is a point of the dispersion  $\mathcal{D}_i(U)$. As a result, the metrical boundary
$\mathcal{E}(U)$
  has always zero measure. However, the set  is non-empty for typical cases:
 
 %
  \begin{lem}\label{lem_exist_one}
 Let $\Omega\subset\mathbb{R}^N, N\geq 1$ be a non-empty connected open set. For any measurable set $U\subset\Omega$ satisfying $0<|U|\leq\infty$ and $0<|\Omega\setminus U|\leq \infty$,  there exists an exceptional point of $U$.
\end{lem} 

A proof of the above lemma can be found  in  
 \cite[Lemma 4]{Croft} \footnote{As noted in \cite[p78]{Croft}, this result may have been previously known.}  for  the case $\Omega=\mathbb{R}^2$ and \cite{Szenes} for the 1d case even with  a sharper quantitative estimate. One can easily extend their existence proofs up to general connected open sets in $\mathbb{R}^N$.
We present a similar proof  in Appendix \ref{app_excep} for completeness.\\


 On the other hand,
  an exceptional point  has an interesting feature in the following sense:\\
If  $U, U^c$ have positive measures  and if $x$ is an exceptional point of $U$, then there exists 
a positive sequence $\{r_n\}$ satisfying   $r_n\to 0$
while
both $(B_{r_n}(x)\cap U)$ and $(B_{r_n}(x)\cap U^c)$ have positive measures for each $n$. An an application, we can construct a sequence $\{f_n\}$ of non-negative bounded functions which are supported in $U$ converging to the Dirac mass at the exceptional point $x$. For instance, we can take
$$f_n=\frac{1}{|B_{r_n}(x)\cap U|}1_{B_{r_n}(x)\cap U},\quad n\geq 1.$$
Similarly, there is a sequence of functions supported in  $U^c$  with the same property.  \\


For our purpose, we need two exceptional points with different distances toward the axis of symmetry.
 \begin{cor}\label{cor_exist_two}
 For any 
 measurable set $U\subset\mathbb{R}^3$ satisfying $0<|U|<\infty$,  there exist at least two exceptional points 
 $y',y''\in\mathbb{R}^3$ of the set $U$  satisfying $r'>r''>0$.
\end{cor} 
 \begin{proof}
 We take any $\tilde r>0$ such that both
$U_A:=U\cap \{0<r<\tilde r\}$ and $U_B:=U\cap \{r>\tilde r\}$ have finite positive measure. 
By Lemma \ref{lem_exist_one} for 
  $\Omega:= \{0<r<\tilde r\}$, we have an exceptional point $y'$ of $U_A$. Definitely, the point $y'$ lies on $ \{0<r<\tilde r\}$.  Similarly, we get an exceptional point $y''\in \{r>\tilde r\}$ of $U_B$ by setting $\Omega= \{r>\tilde r\}$.
 \end{proof}

Coming back to the proof of existence of $W,\gamma\geq 0$, since $|{A}|\in(0,\infty)$, we can apply Corollary
\ref{cor_exist_two} into  the set $ {A}$
so that 
   there are (at least) two exceptional points $y'=(y'_1,y'_2,y'_3), y''=(y''_1,y''_2,y''_3)\in\mathbb{R}^3$ of the set $A$ whose cylindrical coordinates   $ (r',z'),  (r'',z'')\in \Pi$ satisfy  $r'>r''>0$.
    Set
  \begin{equation}\label{defn_r_0}
 r_0 =\min\{r'-r'',r''\}>0.
\end{equation}
 
$\bullet$ Step 1 - construction of $W,\gamma$ from stream function via exceptional points:\\

We define \begin{equation}\label{def_abc} a=\frac{r'^2}{r'^2-r''^2}, \quad b=\frac{r''^2}{r'^2-r''^2},\quad\mbox{and}\quad c=\frac{2}{r'^2-r''^2}.\end{equation} Then, $a>b>0$ and $c>0$.
By using the stream function
$ \psi=\mathcal{G}[\xi]$,
we define  $\gamma$ and $W$ by
 \begin{equation}\label{def_W_gamma}\gamma  =a\psi({y''})-b\psi({y'}) \quad \mbox{ and }\quad W=c(\psi({y'})-\psi({y''})).\end{equation}
Using such constants $W,\gamma$, we will show 
\begin{align*}
&0\geq  
{\int}\left(\psi-W {\frac 1 2 r^2}-\gamma \right)h\dd x 
\end{align*} for any axi-symmetric compactly supported  $h\in L^\infty(\mathbb{R}^3)$ satisfying
\begin{equation}\label{sign_h}
h\geq 0\quad \mbox{on}\quad A^c,\quad h\leq 0 \quad \mbox{on}\quad A.
\end{equation}
Once we obtain it, it implies
$$\left(\psi-W {\frac 1 2 r^2}-\gamma \right)\leq 0 \quad\mbox{a.e. } \mbox{on}\quad A^c,\quad \left(\psi-W {\frac 1 2 r^2}-\gamma \right)\geq 0 \quad\mbox{a.e. }  \mbox{on}\quad A.$$ Then  the goal 
\eqref{goal_exist_w_gamma} 
of Proposition \ref{prop_exist_vortex} for the existence part will follow once we show
$$\left(\psi-W {\frac 1 2 r^2}-\gamma \right)\neq  0 \quad\mbox{a.e. } \mbox{on}\quad A.$$

\begin{rem}\label{rem_intui} 
Here is the motivation for the choice of $W,\gamma$ in \eqref{def_W_gamma} (and $a,b,c$ in \eqref{def_abc}):\\
If we define axi-symmetric 
functions $h_i$  in $\mathbb{R}^3$ by their cylindrical form:
  $$h_1 =\frac a {2\pi r}\delta_{{(r'',z'')}} -\frac b {2\pi r}\delta_{{(r',z')}} \quad \mbox{and}\quad h_2  =\frac c {2\pi r}(\delta_{{(r',z')}}-\delta_{{(r'',z'')}}),$$ where $\delta_{(r,z)}$ is the Dirac-delta function in $\Pi$ at $(r,z)$,  
   then we formally obtain
   \begin{equation}\label{gamma_defn_mot}
   \int\psi h_1dx= \int\psi (a\delta_{{(r'',z'')}} -b\delta_{{(r',z')}})   drdz=\gamma.
   \end{equation}
Similarly, we get 
\begin{equation}\label{W_defn_mot}
\int\psi h_2dx
 =W. 
\end{equation}
Due to  the relations
 \begin{equation}\label{prop_abc}
   a-b=1, \quad \frac 1 2(ar''^2-br'^2)=0, \quad\frac c 2 (r'^2-r''^2)=1,
 \end{equation}
 we obtain 
\begin{equation}\label{int_cond}
   \int h_{1}dx=1,\quad \int {\frac 1 2 r^2} h_{1}dx=0,\quad \int h_{2}dx=0,\quad \mbox{and}\quad \int {\frac 1 2 r^2} h_{2}dx=1.
\end{equation} 
In order to search $W,\gamma$ as Lagrange multipliers, we recall the proof of Lemma \ref{key_lemma}. As in \eqref{defn_eta}, we    set  $\eta$  by
 \begin{equation}\label{defn_eta_5}
 \eta=h-\left(\int hdx\right)h_1-\left(\int {\frac 1 2 r^2} hdx\right)h_2 
 \end{equation}
for arbitrary bounded $h$ with \eqref{sign_h}. Then we want to have
$(\xi+\epsilon\eta)\in {\mathcal{P}}_\mu$ (see the definition \eqref{defn_adm}) for small $\epsilon>0$.  From 
$\int \eta dx=0$ and $\int {\frac 1 2 r^2} \eta dx=0,$
we have
\begin{equation}\label{int_cond_2}
\int (\xi+\epsilon\eta) dx=\int \xi dx\leq 1\quad\mbox{and}\quad \frac 1 2\int {  r^2} (\xi+\epsilon\eta) dx= \frac 1 2\int { r^2}\xi dx=\mu.
\end{equation} However, it is     too optimistic to  ask that $(\xi+\epsilon\eta)$ is of patch-type. 
Instead, thanks to 
Theorem \ref{thm_exist_max}, we are allowed to have 
$(\xi+\epsilon\eta)$
 in the larger class
${\mathcal{P}}'_\mu$ (or $\mathcal{P}_\mu''$) (see the definition \eqref{defn_prime_class}). Nevertheless, we still have a problem since
 $(\xi+\epsilon\eta)$ fails  in general to satisfy the pointwise bound  
\begin{equation}\label{bdd_cond}
0\leq (\xi+\epsilon\eta)\leq 1\quad \mbox{a.e.}
\end{equation} 
In fact, $\eta$ is not a measurable function but  merely a measure. 
Hence, we are asked to approximate the formal perturbation $\eta$ 
 of \eqref{defn_eta_5}
  by a sequence of measurable functions 
$\{\eta_n\}$  
satisfying   \eqref{int_cond_2} and \eqref{bdd_cond} at the same time. It raises 
technical difficulties. For instance, it asks certain sign condition for $\eta_n$ due to
the form $\xi=1_A$. Indeed, from $\epsilon>0$, we need
\begin{equation}\label{eta_sign_rem}
 \eta_n\leq 0 \quad \mbox{on}\quad A,\quad \eta_n\geq 0 \quad \mbox{on}\quad A^c 
\end{equation}
for arbitrary $h$. Hence,
  we need sequences  $\{h_{i,n}\}$ converging to $h_i$ for each $i=1,2$ such that
 the functions
$h_{i,n}$ 
 satisfy  not only
\eqref{int_cond}  but also
certain sign condition depending the choice of $h$ (see \eqref{condn_h_n}).   As a result, 
 $\eta_n$ satisfies \eqref{eta_sign_rem} for each $n$. 
 In short, for different $h$,   we have to construct different $h_{i,n}$ (so different $\eta_n$)  (see \eqref{defn_h_n}, \eqref{defn_eta_n}). 
 The construction   will be given  below in detail   since  the process searching such a perturbation sequence seems not standard. \\ %

   \end{rem}

$\bullet$ Step 2 - approximation (toward   Dirac mass) 
 supported on each side:\\

Since  $y', y''$ are  exceptional points of $A$, we have 
 a decreasing sequence $\{r_n\}_{n=1}^\infty$ of positive numbers satisfying $r_n\to 0$ as $n\to\infty$ and
$$|B_{r_{n}}({y'})\cap A|,\, |B_{r_{n}}({y'})\cap A^c|,\,  |B_{r_{n}}({y''})\cap A|,\, |B_{r_{n}}({y''})\cap A^c|>0\quad \mbox{for}\quad n\geq 1.$$
 Here $B_{r_n}(y'), B_{r_n}(y'')$ are usual balls  
 in $\mathbb{R}^3$ as defined in \eqref{defn_ball}, and  $|\cdot|$ is the Lebesgue measure for $\mathbb{R}^3$. Since $A\subset \mathbb{R}^3$ is axi-symmetric, we have
$$ |T_{r_{n}}({r',z'})\cap A|,\, |T_{r_{n}}({r',z'})\cap A^c|,\,  |T_{r_{n}}({r'',z''})\cap A|,\, |T_{r_{n}}({r'',z''})\cap A^c|>0\quad \mbox{for}\quad n\geq 1.$$
 Here $T_{r_n}(r',z'), T_{r_n}(r'',z'')$
 are  tori  in $\mathbb{R}^3$ defined in \eqref{defn_ball}.  We may assume 
 \begin{equation*}\label{condn_r_1}
0<r_{1}<\frac{r_0}{2}. 
 \end{equation*}
Then, due to   \eqref{defn_r_0} and the monotonicity of $\{r_n\}$, we know
 $$  T_{r_{n}}({r',z'})\cap T_{r_{n}}({r'',z''})=\emptyset\quad\mbox{for   }\quad n\geq 1.$$
 We denote the above axi-symmetric sets  of positive measures b{y}
  $$Y^{{+}}_n =T_{r_{n}}({r',z'})\cap A,\quad Y^{{-}}_n =T_{r_{n}}({r',z'})\cap A^c,\quad Z^{{+}}_n =T_{r_{n}}({r'',z''})\cap A,\quad Z^{{-}}_n =T_{r_{n}}({r'',z''})\cap A^c.$$  
For each $n\geq 1$, we define the axi-symmetric, compactly supported  functions
$f^\pm_{n},g^\pm_{n}\in L^\infty(\mathbb{R}^3)$
  by 
$$f^\pm_{n}(x) =\frac{1}{|Y^{{\pm}}_n|}{{1}}_{Y^{{\pm}}_n}(x),\quad 
 g^\pm_{n}(x) =\frac{1}{ |Z^{{\pm}}_n|}{{1}}_{Z^{{\pm}}_n}(x).$$
The above functions are designed to satisfy
\begin{equation*}\begin{split}
f^\pm_{n},g^\pm_{n}\geq 0, \quad &\int f^\pm_{n}\,dx=1, \quad \int g^\pm_{n}\,dx=1,\\  
\mbox{spt}(f^\pm_{n}) \subset T_{r_n}(r',z'),\quad 
&\mbox{spt}(g^\pm_{n}) \subset T_{r_n}(r'',z''),\\
\mbox{spt}(f^+_{n}), \,\mbox{spt}(g^+_{n})\subset A,\quad 
&\mbox{spt}(f^-_{n}), \,\mbox{spt}(g^-_{n})\subset A^c, \quad n\geq 1,
\end{split}\end{equation*} where the last line 
 means that each sequence is supported either on $A$ or $A^c$.
 Thanks to the convergence $r_n\to 0$,
 the sequences $\{f^\pm_{n}\},\{g^\pm_{n}\}$  approximate the Dirac-masses on the circles
 $$\{(x_1,x_2,x_3)\in\mathbb{R}^3\,|\, x_1^2+x_2^2=r'^2, x_3=z'\}, \quad\{(x_1,x_2,x_3)\in\mathbb{R}^3\,|\,  x_1^2+x_2^2=r''^2, x_3=z''\},$$ respectively in the following sense:\\For any continuous axi-symmetric function $\phi$ in $\mathbb{R}^3$, we have 
 \begin{equation}\label{pro_f_g}
 \int f^\pm_{n}\phi dx \to \phi(y'),\quad  \int g^\pm_{n}\phi dx \to \phi(y'') \quad \mbox{as}\quad n\to \infty.
 \end{equation}\\

$\bullet$ Step 3 - construction of a basis for two constraints(mass, impulse) problem with sign condition:\\

The impulse of the functions are estimated by
$$ \int {\frac 1 2 r^2} f^\pm_{n}\,dx \in\left (\frac 1 2 (r'-r_n)^2,\frac 1 2 (r'+r_n)^2\right),\quad   \int {\frac 1 2 r^2} g^\pm_{n}\,dx\in\left (\frac 1 2 (r''-r_n)^2,\frac 1 2 (r''+r_n)^2\right).$$
Thus, for each $n$, we can set  $\tau'^\pm_{n} \in(r'-r_n,r'+r_n)$ and
 $\tau''^\pm_{n}\in(r''-r_n,r''+r_n)$ b{y} solving   $$\frac 1 2 \left(\tau'^\pm_{n}\right)^2=\int {\frac 1 2 r^2} f^\pm_{n}\,dx, \quad \frac 1 2\left(\tau''^\pm_{n}\right)^2=\int {\frac 1 2 r^2} g^\pm_{n}\,dx.$$
We note 
  $\tau'^\mp_{n}>\tau''^\pm_{n}$. Then we define 
   $$a^\pm_{n}=\frac{(\tau'^\mp_{n})^2}{(\tau'^\mp_{n})^2-(\tau''^\pm_{n})^2}, \quad b^\pm_{n}=\frac{(\tau''^\pm_{n})^2}{(\tau'^\mp_{n})^2-(\tau''^\pm_{n})^2},\quad  c^\pm_{n}=\frac{2}{(\tau'^\pm_{n})^2-(\tau''^\mp_{n})^2}.$$
 By recalling 
the definition of  $a,b,c$ in \eqref{def_abc},
   we  have 
\begin{equation}\label{pro_a_n}
 a^\pm\rightarrow a, \quad b^\pm_n\rightarrow b,\quad  c^\pm_n\rightarrow c \quad\mbox{as}\quad n\to\infty,
\end{equation}   due to $$ |\tau'^\pm_{n}-r'|,
  |\tau''^\pm_{n}-r''|\leq r_n\rightarrow 0.$$   
 We simply observe 
(cf. \eqref{prop_abc}) 
\begin{equation}\label{prop_abc_n}
     a^\pm_n>b^\pm_n>0,\quad c^\pm_n>0, \quad a^\pm_n-b^\pm_n=1,\quad \frac 1 2\left(a^\pm_n (\tau''^\pm_{n})^2-b^\pm_n (\tau'^\mp_{n})^2\right)=0,\quad   \frac {c^\pm_n} 2 \left((\tau'^\pm_{n})^2-(\tau''^\mp_{n})^2\right)=1.
\end{equation}

Now, we define axi-symmetric, compactly supported   functions 
$h^{\pm}_{n,1}, h^{\pm}_{n,2}\in L^\infty(\mathbb{R}^3)$  by
\begin{equation}\label{defn_h_n} h^{\pm}_{n,1}=a^\pm_{n} g^\pm_{n}-b^\pm_{n} f^\mp_{n},\quad h^{\pm}_{n,2}=c^\pm_{n}(f^\pm_{n}- g^\mp_{n}).\end{equation}
By 
\eqref{prop_abc_n},
they form a basis for two constraints(mass, impulse) problem: 
\begin{equation}\label{condn_h_n_2}
\int h^{\pm}_{n,1}dx =1,\quad \int {\frac 1 2 r^2} h^{\pm}_{n,1} dx=0,\quad\quad \int h^{\pm}_{n,2} dx=0,\quad \int {\frac 1 2 r^2} h^{\pm }_{n,2} dx=1.
\end{equation}
Moreover, they satisfy the sign condition:
\begin{equation}\label{condn_h_n}\begin{split}
& h^{+}_{n,1}, h^{+}_{n,2}\geq 0,\quad  h^{-}_{n,1}, h^{-}_{n,2}\leq 0\quad\mbox{on}\quad A,\\
& h^{+}_{n,1}, h^{+}_{n,2}\leq 0,\quad  h^{-}_{n,1}, h^{-}_{n,2}\geq 0\quad\mbox{on}\quad A^c.
\end{split}\end{equation}


By setting 
\begin{equation}\label{defn_w_n}
W^\pm_n=\int\psi h^\pm_{n,2}dx\quad\mbox{and}\quad \gamma^\pm_n  =\int\psi h^\pm_{n,1}dx
\end{equation}(cf. \eqref{W_defn_mot} and \eqref{gamma_defn_mot}),
 we can show
\begin{equation}\label{conv_w_n}
W^\pm_n\rightarrow W,\quad \gamma^\pm_n\rightarrow \gamma\quad\mbox{as}\quad n\rightarrow \infty,
\end{equation}
 where $W,\gamma$ are defined in
\eqref{def_W_gamma}. 
 Indeed, since the stream function $\psi$ is continuous by Lemma \ref{lem_en_iden_origin}, 
 we have, by \eqref{pro_f_g}, \eqref{pro_a_n},
\begin{equation*}\begin{split}
&W^\pm_n=\int\psi h^\pm_{n,2}dx= c^\pm_{n}\left(\int \psi f^\pm_{n}dx -\int \psi g^\mp_{n}dx\right)\to c (\psi(y')-\psi(y''))=W\quad\mbox{as}\quad n\to \infty
\end{split}\end{equation*} and
 \begin{equation*}\begin{split}
&\gamma^\pm_n=\int\psi h^\pm_{n,1}dx=
 a^\pm_{n}\int \psi g^\pm_{n}dx-b^\pm_{n}\int \psi  f^\mp_{n}dx
\to a\psi({y''})-b\psi({y'})=\gamma \quad\mbox{as}\quad n\to \infty.
\end{split}\end{equation*} \\

$\bullet$ Step 4 - construction of a sequence of 
perturbations around patch-type data :\\

Now we are ready to define an approximation in $\mathcal{P}'_\mu$ toward the formal perturbation \eqref{defn_eta_5} around our patch-type function $\xi=1_A$.
Let us take and fix any axi-symmetric function $h\in L^\infty(\mathbb{R}^3)$ satisfying
\begin{equation}\label{cond_sign_h}
h:\mbox{compactly supported},\quad h\geq 0\quad\mbox{on}\quad A^c,\quad \mbox{and}\quad h\leq 0\quad\mbox{on}\quad A.
\end{equation}
 Then there are 4 possible cases depending on the sign of the integrals $\int h\,dx, \int {\frac 1 2 r^2} h \,dx$:\\
case (I) : $\int h\,dx\geq 0, \int {\frac 1 2 r^2} h \,dx\geq 0$, \quad 
case (II) : $\int h\,dx\geq 0, \int {\frac 1 2 r^2} h \,dx< 0$,\\
case (III) : $\int h\,dx< 0, \int {\frac 1 2 r^2} h\,dx \geq 0$,\quad
case (IV) : $\int h\,dx< 0, \int {\frac 1 2 r^2} h\,dx < 0$.\\

For each $n\geq 1$, we define 
\begin{equation}\label{defn_eta_n}
\eta_n:=\begin{cases}&h-(\int h\,dx) h^+_{n,1}-(\int {\frac 1 2 r^2} h\,dx)h^+_{n,2}\quad\mbox{when case (I),}\\
&h-(\int h\,dx) h^+_{n,1}-(\int {\frac 1 2 r^2} h\,dx)h^-_{n,2}\quad\mbox{when case (II),}\\
&h-(\int h\,dx) h^-_{n,1}-(\int {\frac 1 2 r^2} h\,dx)h^+_{n,2}\quad\mbox{when case (III),}\\
&h-(\int h\,dx) h^-_{n,1}-(\int {\frac 1 2 r^2} h\,dx)h^-_{n,2}\quad\mbox{when case (IV).}
\end{cases}
\end{equation}
They are designed to satisfy 
\begin{equation}\label{obs_sign}
 \eta_n\leq 0\quad \mbox{on}\quad A,\quad\eta_n\geq 0\quad\mbox{on}\quad A^c,\quad \int \eta _n \,dx=0,\quad\mbox{and}\quad\int {\frac 1 2 r^2}\eta_n\,dx=0,
\end{equation} for any possible cases.\\

For each $n\geq 1$, we claim 
$$(\xi+\epsilon\eta_n)\in\mathcal{P}'_\mu\quad\mbox{for sufficiently small}\quad  \epsilon>0.$$ 
Indeed,  by fixing the function $h$ and $n\geq 1$, 
we have 
$$0\leq \xi+\epsilon\eta_n\leq 1
\quad \mbox{for sufficiently small}\quad \epsilon>0\quad\mbox{(depending  on  }  \|h^\pm_{n,i}\|_\infty,\,  i=1,2\mbox{)}
$$  thanks to the sign property \eqref{obs_sign}.
In addition, the property \eqref{obs_sign} on the integrals    implies
$$\int(\xi+\epsilon\eta_n)\,dx=\int \xi\,dx\leq 1\quad\mbox{and}\quad
\int {\frac 1 2 r^2}(\xi+\epsilon\eta_n)\,dx=\int {\frac 1 2 r^2} \xi\,dx=\mu. $$
As a result, we obtain the above claim for small $\epsilon>0$.\\

$\bullet$ Step 5 - verification of the boundary of the patch $1_A$ via Lagrange multipliers $W,\gamma$:\\

   Since $\xi\in\mathcal{S}_\mu=\mathcal{S}'_\mu$ by Theorem \ref{thm_exist_max},
we obtain, for small $\epsilon>0$, $$0\geq\frac{E[\xi+\epsilon\eta_n]-E[\xi]}{\epsilon}.$$ 
  Hence, b{y} taking the  limit  $ \epsilon\searrow 0$, we have 
  \begin{align}\label{ineq_int_eta_n}
\quad \quad 0\geq  \frac 1 2 \int\mathcal{G}[\xi]\eta_n\,dx +\frac 1 2 \int \xi\mathcal{G}[\eta_n]\,dx=
\frac 1 2 \int\psi\eta_n\,dx +\frac 1 2 \int \mathcal{G}[\xi]\eta_n\,dx=  {\int}\psi \eta_n\dd x,
\end{align} 
where we used the symmetry of the kernel $G$.
By the definitions \eqref{defn_eta_n} and \eqref{defn_w_n}, for the case $(I)$,  we  have
\begin{align*}
0\geq   {\int}\psi \eta_n\dd x= {\int}\psi h\dd x-\gamma^+_n\left({\int} h\dd x\right)-W^+_n\left({\int}{\frac 1 2 r^2} h\dd x\right)
=
{\int}\left(\psi-W^+_n {\frac 1 2 r^2}-\gamma^+_n \right)h\dd x.
\end{align*}
Similarly, we have
\begin{align*}
&0\geq  
{\int}\left(\psi-W^-_n {\frac 1 2 r^2}-\gamma^+_n \right)h\dd x,
\mbox{ for the  case (II),}\\
&0\geq  
{\int}\left(\psi-W^+_n {\frac 1 2 r^2}-\gamma^-_n \right)h\dd x,
\mbox{ for the case  (III),}\\
&0\geq  
{\int}\left(\psi-W^-_n {\frac 1 2 r^2}-\gamma^-_n \right)h\dd x,
\mbox{ for the case (IV).}
\end{align*}\\
By \eqref{conv_w_n}, 
we can take limit  $n\to \infty$ into the above inequalities  to get
\begin{align}\label{ineq_adj_stream}
&0\geq  
{\int}\left(\psi-W {\frac 1 2 r^2}-\gamma \right)h\dd x
=\int_{A^c}+\int_{A}
\end{align} for any possible cases. We set the adjusted stream function $\Psi =\psi-{ (1/2)  W r^2}-\gamma.$
Since $h$ is an arbitrary function satisfying the sign condition
 \eqref{cond_sign_h},  the inequality 
\eqref{ineq_adj_stream}
 implies
\begin{equation*} 
\begin{aligned}
\Psi&\leq 0\quad a.e.\quad  \textrm{on}\quad A^c \quad\mbox{and}\quad \Psi \geq 0\quad a.e.\quad  \textrm{on}\quad A. 
\end{aligned}
\end{equation*} Thus we get
\begin{equation}\label{semi_con} 
\{\Psi>0\}\subset A\subset \{\Psi\geq 0\}\quad\mbox{up to measure zero}.
\end{equation}\\

On the other hand, since 
$\psi\in H^2_{loc}(\Pi)$ by Lemma \ref{lem_en_iden_origin}, we have
$\Psi\in H^2_{loc}$, which implies
    $$\nabla\Psi=0\quad a.e. \quad \mbox{on}\quad \{\Psi=0\}.$$ Thus we obtain
    $$-\frac{1}{r^2}\mathcal{L}\Psi=0\quad a.e.\quad \mbox{on}\quad \{\Psi=0\}.$$
Due to $- {r^{-2}}\mathcal{L}\Psi=- {r^{-2}}\mathcal{L}\psi=\xi=1_A$,
we have
$$\{\Psi=0\}\subset A^c\quad\mbox{up to measure zero}.$$ 
Hence we conclude, by 
\eqref{semi_con},
$$A=\{\Psi>0\}\quad\mbox{up to measure zero},$$ which gives  the goal \eqref{goal_exist_w_gamma}.\\

Lastly, let us show   $W\geq 0$ and $\gamma\geq 0$.
 B{y} the convergence $\psi\to 0$ as $|(r,z)|\to\infty$
due to Lemma
\ref{lem_en_iden_origin}
and b{y} $|\{\Psi>0\}|<\infty$, 
we can take a sequence $\{(r_n,z_n)\}$ in $\Pi$   such that $$(r_n, z_n)\in\{\Psi\leq 0\},\quad n\geq 1,$$ $$\psi(r_n,z_n)\to 0,\quad  z_n\to\infty,\quad r_n\to 0\quad \mbox{as}\quad n\to\infty.$$
 Thus we have
\begin{align}\label{eq_limsup}
\limsup_{n\to\infty}\left(\psi(r_n,z_n)-\frac 1 2Wr_n^2-\gamma\right)=\limsup_{n\to\infty}\Psi(r_n,z_n) \leq 0.
\end{align}
Hence we obtain $\gamma\geq 0$.  Similarly, 
 by taking  another sequence $\{(r_n,z_n)\}$ such that $(r_n,z_n)\in\{\Psi\leq 0\}$,  $\psi(r_n,z_n)\to0$,   $z_n\to0$, and $r_n\to\infty$ as $n\to \infty$, we get \eqref{eq_limsup} again, which cannot be true when  $W<0$. Thus we obtain $W\geq 0$. It finishes the proof of  Proposition \ref{prop_exist_vortex}.

\end{proof}

\color{black}
  \subsection{Traveling speed is non-trivial.}\label{sec5_step2}\ \\
  
In this subsection, we prove that  the traveling speed $W$ of  vortex rings from our variational problem \eqref{var_prob} is  positive.
To prove,  we
use the energy identity \eqref{iden_en2} below due to   \cite[Lemma 3.1]{FT81}.
\begin{lem}
\label{lem_3.1_FT}
For axi-symmetric   $\xi\in \left(  L^1_w\cap L^\infty\cap L^1\right)(\mathbb{R}^3)$,  if $\xi$ is 
compactly supported, then we have 
\begin{equation}\label{iden_en2}
E[\xi]
=\int_{\mathbb{R}^3}(x\cdot \nabla\psi)\xi\,dx=\int_{\mathbb{R}^3}(r\partial_r \psi +z\partial_z \psi)\xi\,dx,
\end{equation} where $\psi=\mathcal{G}[\xi]$.
\end{lem}
The above lemma can be formally obtained  by using the integration by parts formula
\begin{equation}\label{ibp_dx}
\int\frac{(\partial_r f)(\partial_r g)+(\partial_z f)(\partial_z g)}{r^2}dx=-\int f\frac{\mathcal Lg}{r^2}dx,\quad f,g:\mbox{axi-symmetric} 
\end{equation}
 twice. 
A proof in detail can be found  in  \cite[p9-10]{FT81}.

\begin{rem}\label{rem_abs}
The  assumption on compactness (so boundedness) of $\mbox{spt}(\xi)$ in Lemma \ref{lem_3.1_FT} guarantees that the integral in \eqref{iden_en2} converges absolutely. Indeed,  if
$\mbox{spt}(\xi)\subset B_K(0)$, then
we have
\begin{equation}\label{abs_conv}
\int|(r\partial_r \psi +z\partial_z \psi)\xi|\,dx\lesssim K^2
\int\left|\frac{\partial_r \psi}{r} +\frac{\partial_z \psi}{r}\right|\cdot|\xi|\,dx
\lesssim K^2 \|\mathcal{K}[\xi]\|_2\|\xi\|_2 <\infty
\end{equation}
by \eqref{iden_en} of Lemma \ref{lem_en_iden_origin}.
If one wants to drop the assumption,
the integral 
  should be understood in a limit sense.
\end{rem}


\begin{prop}\label{prop_en_W}
For any   $\xi\in \mathcal{S}_\mu$, the constant
$W$ in \eqref{eq_W_lem} of Proposition \ref{prop_exist_vortex} is positive. 
\end{prop}
\begin{proof}
Let  $\xi\in \mathcal{S}_\mu$. By 
Proposition \ref{prop_exist_vortex},
we have
\begin{equation*}
\begin{aligned}
&\xi= {{1}}_{ \{\psi-\frac{1}{2}Wr^2-\gamma>0\}}\quad\mbox{a.e.} 
\end{aligned}
\end{equation*} 
 for some $W\geq 0$ and $\gamma\geq 0$.    
  For a contradiction, let us assume $W=0$. By setting 
\begin{equation}\label{set_a}
  A=\{x\in\mathbb{R}^3\,|\, \psi(x)-\gamma>0\},
\end{equation}  
 we know
  $|A|=\|\xi\|_1\leq 1$. Thus we obtain $\gamma>0$. Indeed, if $\gamma=0$, then $$A=\{\psi>0\}=\mathbb{R}^3\setminus\{r=0\}$$ 
  since the kernel $G$ is positive a.e. 
  and $\xi\geq0$ is non-trivial.  It implies $|A|=\infty$, which is a contradiction. Thus $\gamma=0$ is impossible.\\
  
   By 
\eqref{psi_conv} in Lemma \ref{lem_en_iden_origin},
we know the convergence $\psi(x)\to 0$ as $|x|\to\infty$.   
It implies that the set $A$ in \eqref{set_a} is bounded due to $\gamma>0$. Thus $\xi$ is (essentially) compactly supported. 
Now we can apply the identity 
\eqref{iden_en2} of
Lemma \ref{lem_3.1_FT}
 to have
 \begin{align*}
E[\xi]&= \int(r\partial_r\psi +z\partial_z\psi)\xi\,dx= \int(r\partial_r\Psi +z\partial_z\Psi)\xi\,dx
 =
 \int\left(r\partial_rF(\Psi) +z \partial_zF(\Psi)\right)\,dx,
\end{align*}   where
 we set $\Psi=(\psi-\gamma)$ and 
$F(s)=s^+=\begin{cases} &s,\quad s>0\\ &0,\quad s\leq 0,\end{cases} $ which 
 is an antiderivative of the vorticity function  $f_H(s)=1_{\{s>0\}}$. Then,
we can compute, by integration by parts which will be verified below,
 \begin{equation}\begin{split}\label{ibp_pre}
E[\xi]&=
   \int x\cdot \nabla[F(\Psi)] \,dx=  -\int  (\nabla\cdot x) F(\Psi) \,dx
=
- 3  \int F(\Psi) \,dx\leq 0,
 \end{split}\end{equation}  
  which gives a contradiction
due to $E[\xi]=\mathcal{I}_\mu >0$ from Theorem \ref{thm_exist_max}. Therefore we get $W>0$ once we justify 
the integration by parts done in the above. \\

To justify, we take
  a radial function
 $\varphi\in C^{\infty}_{c}(\mathbb{R}^3)$ satisfying  $\varphi(x)=1$ for $|x|\leq 1$ and $ \varphi(x)=0$ for $|x|\geq 2$, and set the cut-off function on $\mathbb{R}^3$ by $\varphi_M(x)=\varphi(x/M)$ for any $M>0$. Then we have
 \begin{equation}\begin{split}\label{ibp} 
  \int  (x\varphi_M) \cdot \nabla[F(\Psi)]\,dx
  =
  -\int   \varphi_M (\nabla\cdot x) F(\Psi) \,dx
  -\int (x\cdot \nabla\varphi_M)F(\Psi)  \,dx.
 \end{split}\end{equation}  
The first integral  on the left-hand side of 
\eqref{ibp} converges absolutely as $M\to\infty$ since
$$
 \int  |(x\varphi_M) \cdot \nabla[F(\Psi)]|\,dx=\int|\varphi_M(r\partial_r \psi +z\partial_z \psi)\xi|\,dx\leq \int| (r\partial_r \psi +z\partial_z \psi)\xi|\,dx<\infty
$$ by the computation \eqref{abs_conv} in Remark \ref{rem_abs}. 
We observe $$0\leq F(\Psi)= \Psi 1_{A}\leq \psi 1_{A}  $$
and
$$\|\psi\|_\infty\lesssim \left(   \|r^2\xi\|_{1}+
   \|\xi\|_{L^1\cap L^2} \right)
   \lesssim \left(   \|r^2\xi\|_{1}+
   \|\xi\|_{L^1\cap L^\infty}\right)\lesssim (\mu+1)$$ 
 by  \eqref{est_psi_bdd} for $\delta=1$.
Thus, the first integral  on the right-hand side of 
\eqref{ibp}
converges absolutely since 
$$
\int   |\varphi_M (\nabla\cdot x) F(\Psi)| \,dx\leq 3
\int  | F(\Psi)| \,dx\lesssim (1+\mu)|A|<\infty.
$$
For the last integral of \eqref{ibp}, we compute
 \begin{align*} 
 \left|\int  (x\cdot \nabla\varphi_M)  F(\Psi)\,dx\right|\lesssim
 (1+\mu)\int_A |x| |\nabla\varphi_M|  \,dx
 \lesssim
 (1+\mu)|{A\cap\{M\leq |x|\leq 2M\}} |.
 \end{align*} 
 Since $A$ is bounded in $\mathbb{R}^3$, the intersection 
  $A\cap\{M\leq |x|\leq 2M\}$
 has zero measure for sufficiently large $M>0$.
 We have justified the integration by parts done in \eqref{ibp_pre}.\\
  
\end{proof}

  \subsection{Vortex core is bounded.}\label{sec5_step3}\ \\

In order to show compactness of the vortex core,  we prove first that   $r^{-2}\mathcal{G}[\xi]$ for each maximizer $\xi$ vanishes at infinity. 
\begin{lem}\label{lem_psi_hol_conti}
Let   $\xi\in \mathcal{S}_\mu$.
 Then, the stream function  $\psi=\mathcal{G}[\xi]$  
 satisfies, for any $\alpha\in(0,1)$, $$\frac \psi r\in BUC^{1+\alpha}(\overline\Pi)\quad\mbox{and}\quad \frac \psi {r^2}\in BUC^{\alpha}(\overline\Pi).$$ In particular, it satisfies 

\begin{align}\label{decay_frac}
\frac{\psi(r,z)}{r^2}\to 0\quad \textrm{as}\quad |(r,z)|\to\infty.   
\end{align}
\end{lem}
 
\begin{proof}
Let   $\xi\in \mathcal{S}_\mu$. By Propositions \ref{prop_exist_vortex} and \ref{prop_en_W}, it satisfies

\begin{equation*}
\begin{aligned}
&\xi= {{1}}_{ \{\psi-\frac{1}{2}Wr^2-\gamma>0\}}\quad\mbox{a.e.} 
\end{aligned}
\end{equation*}
for some $W>0$ and $\gamma\geq 0$. We observe
that if $x\in \mbox{spt}( \xi)$, then $$\psi(r,z)\geq  \frac 1 2 W r^2.$$
Since we have 
$\|\psi\|_\infty\lesssim   (\mu+1)$
 by \eqref{est_psi_bdd} for $\delta=1$,  there exists $R>0$ such that $$\mbox{spt}( \xi)\subset \{x\in\mathbb{R}^3\,|\, r\leq R\}. $$
Thus, by setting 
$\omega(x)=r\xi(r,z) e_\theta(\theta)$, we get 
$|\omega(x)|\leq R\xi(r,z)$, which implies
$$\omega\in (L^1\cap L^\infty)(\mathbb{R}^3).$$
By setting 
 $\phi(x)=
 ({\psi(r,z)}/{r})e_\theta(\theta)$, we have
 $ \phi =({1}/{4\pi|x|})* \omega$ (recall Subsection \ref{subsec_axi-sym}).
Since $\omega\in (L^1\cap L^\infty)$, this representation 
implies
\begin{equation}\label{est_phi_r}
\nabla \phi\in L^{p}(\mathbb{R}^3),\quad p\in (3/2,\infty)\quad\mbox{and}\quad
\phi\in W^{2,q}(\mathbb{R}^3),\quad q\in (3,\infty)
\end{equation} (e.g. see \cite[p354]{Stein93}).
It implies
$\phi$ and $\nabla\phi$ are continuous on $\mathbb{R}^3$  (e.g. see \cite[p284]{Evans_book}). 
In particular, $\phi\in BUC^{1+\alpha}(\overline{\mathbb{R}^3})$  for any $\alpha\in(0,1)$  (e.g. see \cite[p280]{Evans_book}).   Since
\begin{equation*}
 \frac{\psi(r,z)}{r}=  \phi_2(r e_{x_1}+z e_{x_3}),\quad (r,z)\in \Pi,
\end{equation*} 
 as in \eqref{psi_from_phi}  where $\phi(x)=\phi_{1}(x) e_{x_1}+\phi_{2}(x)e_{x_2} +\phi_3(x) e_{x_3}$, we get
 $\psi/r\in BUC^{1+\alpha}(\overline\Pi)$ and
  $\partial_r(\psi/r)\in BUC^{\alpha}(\overline\Pi)$.\ \\

We observe that the form $\phi(x)=
 ({\psi(r,z)}/{r})e_\theta(\theta)$ implies
  $\phi(x)|_{r=0}=0$ thanks to continuity of $\phi$ on $\mathbb{R}^3$ and $\psi/r$ on $\overline\Pi$. Thus,  we obtain $(\psi(r,z) / r)|_{r=0}=|\phi(x)||_{r=0}
  =0$.  By applying  the identity
\begin{align*}
f(r)-f(0)=f(rs)|_{s=0}^1=r\int_0^1 f'(rs)ds
\end{align*} into $f(r)=r^{-1}\psi(r,z)$ (for fixed $z$), we obtain, for each $(r,z)\in\Pi,$
\begin{align*}
\frac{\psi(r,z)}{r^2} 
=\int_{0}^{1}\partial_{r}\left(\frac{\psi(r,z)}{r}\right)\Big|_{(r,z)=(rs,z)}\dd s. 
\end{align*}
Due to 
  $\partial_r(\psi/r)\in BUC^{\alpha}(\overline\Pi)$,
  we get  $\psi/r^2\in BUC^\alpha(\overline\Pi)$.
As a consequence, we establish $|\phi|/r \in BUC^{ \alpha}(\overline{\mathbb{R}^3})$.\\ 

On the other hand, 
by  the weighted Hardy's inequality for $\mathbb{R}^2$ (see e.g. \cite{Kufner} or \cite[Corollary 14 (ii)]{DaHi}) and by \eqref{est_phi_r},  we have, for any $p\in(3/2,2)$,
\begin{align*}
\left\|\frac{|\phi|}{r} \right\|^p_{p}&\lesssim\sum_{i=1}^3\int_{\mathbb{R}}\left( \int_{\mathbb{R}^2}\frac {|\phi_i|^p}{{\sqrt{x_1^2+x_2^2}}^p}dx_1dx_2\right)dx_3 \\&
\lesssim_p\sum_{i=1}^3\int_{\mathbb{R}}\left( \int_{\mathbb{R}^2} {|\nabla_{x_1,x_2}\phi_i|^p}dx_1dx_2\right)dx_3
\lesssim \left\|\nabla \phi \right\|^p_{p}<\infty.
\end{align*}  
Thus we have the convergence
$$\frac{|\phi(x)|}{r} \to 0\quad\mbox{as}\quad  |x|\to \infty$$ thanks to
$|\phi|/r\in BUC^{ \alpha}(\overline{\mathbb{R}^3})$.  In other words, we obtain \eqref{decay_frac} .
 
\end{proof}
 
Thanks to Lemma \ref{lem_psi_hol_conti}, we can prove below  that every maximizer is compactly supported. We also obtain
a lower bound of   the traveling speed 
$W$
  depending only on its impulse $\mu$ together with an interesting identity. They are essentially contained in \cite[Lemma 3.2]{FT81} (for $\mu=1$).  We will need them 
  when proving Proposition \ref{prop_small_mu} and Theorem \ref{thm_uniq} in Section \ref{sec_uniq_proof_hill}.  
\begin{prop}\label{prop_cpt_supp}
For any   $\xi\in \mathcal{S}_\mu$, the  support of $\xi$ in $\mathbb{R}^3$ is compact. Moreover, 
the constant
$W$ in \eqref{eq_W_lem} of Proposition \ref{prop_exist_vortex} satisfies
\begin{equation}\label{est_W_lower}
 0<\frac{\mathcal{I}_\mu}{2\mu}\leq W\quad \mbox{and}\quad 
 7\mathcal{I}_\mu=5W\mu+3\gamma\int_{\mathbb{R}^3} \xi dx.
\end{equation}

\end{prop}
\begin{proof}
Let   $\xi\in \mathcal{S}_\mu$. By Proposition \ref{prop_exist_vortex} and \ref{prop_en_W},  we get
$
\xi= {{1}}_{ A} 
$ a.e.  where $A=\{\psi-\frac{1}{2}Wr^2-\gamma>0\}$ 
for some $W>0$ and $\gamma\geq 0$. 
  If 
$x\in A$, 
 then we have 
$$ \frac{\psi(r,z)}{r^2}\geq  \frac 1 2 W>0.$$
Thus the convergence \eqref{decay_frac} 
from Lemma \ref{lem_psi_hol_conti}
 implies
that $A$ is bounded so that 
the (essential) support of $\xi$ is bounded in $\mathbb{R}^3$.\\ 

 Now we  use the identity 
 \eqref{iden_en2} from  Lemma \ref{lem_3.1_FT}. Here we follow a similar line of the proof of 
Proposition \ref{prop_en_W}. 
 By setting 
 $\Psi=\psi-(1/2)Wr^2-\gamma$,  we observe
 $\partial_r\Psi =\partial_r\psi -Wr$ and $\partial_z\Psi =\partial_z\psi.$ Thus we  have, by setting 
 $F(s)=s^+$, \begin{equation*}
 E[\xi]=\int(r\partial_r\Psi +z\partial_z\Psi)\xi\,dx+W\int r^2\xi\,dx
 =\int x\cdot \nabla [F(\Psi)] dx+2W\mu.
\end{equation*}  
 By integration by parts, we   compute 
$$ \int x\cdot \nabla [F(\Psi)] dx= -3 \int F(\Psi)dx,$$
which  can be  justified as done in the proof of the integration by parts \eqref{ibp_pre} in Proposition \ref{prop_en_W}. 
In sum, we get
\begin{equation}\label{W_lower_comp}
\mathcal{I}_\mu =E[\xi]=
-3{\int F(\Psi)\,dx}+2W\mu\leq 2W\mu.
\end{equation} Due to $\mathcal{I}_\mu>0$ from Theorem \ref{thm_exist_max}, the estimate in \eqref{est_W_lower} is obtained.\\

To prove the identity in 
\eqref{est_W_lower}, we claim 
 \begin{align}\label{W_claim}
\mathcal{I}_\mu=\frac{1}{2}\left(
\int F(\Psi) dx + W\mu +\gamma \int \xi dx
\right).
\end{align}  Indeed, we compute 
\begin{align*}
 \int F(\Psi) dx=&\int \Psi^+dx= \int \Psi \xi dx
= \int \left(\psi-\frac 1 2 W r^2-\gamma \right) \xi dx\\
&=  2 E[\xi] -\frac 1 2 W\int r^2\xi dx-\gamma \int \xi dx= 2 \mathcal{I}_\mu-  W\mu -\gamma \int \xi dx,
\end{align*}  which gives the above claim. By combining the
identities \eqref{W_lower_comp} and \eqref{W_claim},
we get the identity 
$$
7\mathcal{I}_\mu=5W\mu+3\gamma\int \xi dx.
$$

   \end{proof}

 Now we are ready to finish proving Theorem \ref{thm_max_is_ring}.
 \begin{proof}[Proof of Theorem \ref{thm_max_is_ring}]
 
By using
Proposition \ref{prop_exist_vortex},   \ref{prop_en_W},    \ref{prop_cpt_supp}, the proof is done.

\end{proof}

\subsection{Positive flux constant gives the full mass.}\ \\

Before finishing this section, we present  the following lemma saying that
  positivity of $\gamma$ in \eqref{eq_W} for  $\xi\in \mathcal{S}_\mu$ guarantees that $\xi$ has the full mass. We put this lemma here since its proof  follows   a similar manner of the proof of Proposition \ref{prop_exist_vortex} even though the result  will be used only in Section \ref{sec_uniq_proof_hill}.
  \begin{lem}\label{lem_pos_gam}
  For each $\xi\in \mathcal{S}_{\mu}$ having $\gamma>0$ in \eqref{eq_W}, we have $$\int_{\mathbb{R}^3} \xi dx =1.$$
\end{lem}
\begin{rem}
The above lemma was first shown by \cite[Remark 1 in Section 5]{FT81} for a certain maximizer $\xi\in\mathcal{S}_\mu$ constructed in the paper. We simply adapt the proof here so that it works for \textit{every} maximizer (cf. \cite[Remark 2.6 (ii)]{AC2019} for 2d case).
\end{rem}
\begin{proof}
 
 Let  $\xi\in \mathcal{S}_{\mu}$. Then, 
  $\xi\in  \mathcal{S}'_{\mu}$ by Theorem \ref{thm_exist_max}, and there exist unique $W>0, \gamma\geq 0$ such that  $\xi={1}_A$ with $A=\{\psi-(1/2)Wr^2-\gamma>0\}$   by Theorem \ref{thm_max_is_ring} which we just have proved. We recall that the unique constants 
$W, \gamma$ are simply obtained  by \eqref{def_W_gamma_uniq} 
or \eqref{def_W_gamma} (in the proof of Proposition \ref{prop_exist_vortex}).  \\
 
 Let us suppose   $$\int \xi dx <1.$$ Then our goal is to  show $\gamma=0$. 
For each $n\geq 1$, we define (cf. see \eqref{defn_eta_n}) $$\eta_n =\begin{cases}&h-\left(\frac 1 2\int {  r^2} hdx\right)h^+_{n,2}\quad\mbox{when}\quad \frac 1 2\int {   r^2}h dx \geq 0 \quad \mbox{i.e. case (I),(III)},\\
&h-\left(\frac 1 2\int {  r^2} hdx\right)h^-_{n,2}\quad\mbox{when}\quad \frac 1 2\int {  r^2}h dx < 0\quad \mbox{i.e. case (II),(IV)},
\end{cases}$$ where $ h, h^\pm_{n,2}\in L^\infty(\mathbb{R}^3)$ are axi-symmetric, compactly supported functions appeared in 
\eqref{defn_h_n} and \eqref{cond_sign_h} during the proof of Proposition  \ref{prop_exist_vortex}. Thanks to the properties 
\eqref{condn_h_n_2}, \eqref{condn_h_n} of $h^\pm_{n,2}$, the function $\eta_n$ for each $n$ satisfies 
(cf. \eqref{obs_sign})
$$ \eta_n\leq 0\quad\mbox{on}\quad A,\quad \eta_n\geq 0\quad  \mbox{on}\quad A^c,\quad \frac 1 2 \int {  r^2}\eta_n\,dx=0.$$
We fix $n\geq 1$ and consider $(\xi+\epsilon\eta_n)$ for $\epsilon>0$.
Due to the assumption $\int \xi dx <1$, we get $$(\xi+\epsilon\eta_n)\in \mathcal{P}'_\mu\quad \mbox{for any sufficiently small}\quad \epsilon>0.$$
Thus we have $E[\xi+\epsilon\eta_n]-E[\xi]\leq 0$ for such $\epsilon>0$.
  Hence,
   by taking limit $\epsilon\searrow 0$, we have
$$0\geq \int \psi\eta_n dx$$ as in \eqref{ineq_int_eta_n}.  
From the definition of $W^+_n$ in \eqref{defn_w_n}, we have 
\begin{align*}
&0\geq 
{\int}\left(\psi-W^+_n {\frac 1 2 r^2} \right)h\dd x\quad \mbox{when}\quad \frac 1 2\int {  r^2} h dx \geq 0,\\
&0\geq  
{\int}\left(\psi-W^-_n {\frac 1 2 r^2} \right)h\dd x\quad \mbox{when}\quad \frac 1 2\int {  r^2} h dx< 0.
\end{align*}
By  the convergence $W^\pm_n\rightarrow W$  
in \eqref{conv_w_n}, we can take the limit to the above inequalities   to obtain
\begin{align*}
&0\geq  
{\int}\left(\psi-W {\frac 1 2 r^2}  \right)h\dd x=\int_{A^c}+\int_{A}
\end{align*} for any case.
Since $h$ is an arbitrary function satisfying 
the sign condition \eqref{cond_sign_h}, 
 we   follow the same approach as in  
 (the last part of) the proof of Proposition 
 \ref{prop_exist_vortex} to arrive at
$$ A=\{\psi-\frac 1 2 W{ r^2}>0 \}\quad a.e.$$ By uniqueness of the pair of constants $(W, \gamma)$ from Proposition \ref{prop_exist_vortex}, we get $\gamma=0$.  \ \\
 
\end{proof}
\section{Compactness}\label{sec_cpt}

In this section, we prove compactness (Theorem \ref{thm_cpt}), which is needed  when proving the stability    (Theorem \ref{thm_hill_gen})   in Subsection \ref{subsec_proof}.

\vspace{15pt}
 
\subsection{Concentrated compactness lemma:  Lions (1984)}\label{subsec_cpt}\ \\

We start the proof by stating  a slight variation of the concentrated compactness lemma  \cite[Lemma I.1]{Lions84a}. For instance, such a variation can be found in  \cite[Lemma 1]{BNL13}.
Here,  $  B_R(r',\,z')=\{(r,z)\in\Pi\,|\, |(r,z)-(r',\,z')|<R\}$ is the disk in the half-space $\Pi$ centered at $(r',z')$ with radius $R$ as  defined in \eqref{defn_ball}.
\begin{lem}\label{lem_concent}
Let $0<\mu<\infty$. Let $\{\rho_n\}_{n=1}^\infty\subset L^{1}(\Pi)$ satisfy 
\begin{align*}
\rho_n\geq 0\quad \mbox{for}\quad  n\geq 1\quad \mbox{and}\quad 
{\int_\Pi}\rho_n drdz\to \mu\quad \textrm{as}\quad n\to\infty.
\end{align*}
Then, there exists a subsequence $\{\rho_{n_k}\}_{k=1}^\infty$ satisfying  one of the three following possibilities:\\

\noindent 
(i) (Compactness)
There exists a sequence $\{(r_k,z_k)\}_{k=1}^\infty\subset  {\overline\Pi}$ such that 
for arbitrary $\varepsilon>0$, there exist $R>0$ and an integer $k_0\geq 1$  such that 
\begin{align*}
\int_{
 B_R(r_k,z_k)
}
\rho_{n_k} drdz\geq \mu-\varepsilon \qquad \textrm{for}\quad k\geq k_0.
\end{align*}\\ 
\noindent 
(ii) (Vanishing) For each $R>0$,
\begin{align*}
\lim_{k\to\infty}\sup_{(r',\,z')\in \overline\Pi }\int_{
 B_R(r',\,z')
}\rho_{n_k} drdz=0.    
\end{align*}\\
\noindent 
(iii) (Dichotomy) There exists a constant $\alpha\in (0,\mu)$ such that for arbitrary $\varepsilon>0$, there exist 
an integer $k_0\geq 1$ and
sequences  $\{\rho_{k}^{(1)}\}_{k=1}^\infty$, $\{\rho_{k}^{(2)}\}_{k=1}^\infty\subset L^{1}(\Pi)$ such that for each $k\geq k_0$, \\
$$ \rho_{k}^{(1)} ={1}_{\Omega_k^{(1)}}\rho_{n_k},
\quad \rho_{k}^{(2)} ={1}_{\Omega_k^{(2)}}\rho_{n_k}\quad\mbox{
for some disjoint measurable subsets }\quad \Omega_k^{(1)},\Omega_k^{(2)}\subset \Pi,$$
\begin{equation*}
\begin{aligned}
||\rho_{n_k}-(\rho_{k}^{(1)}+\rho_{k}^{(2)})||_{L^1(\Pi)}+
\left|{\int_\Pi}\rho_{k}^{(1)}drdz-\alpha  \right|
+\left|{\int_\Pi}\rho_{k}^{(2)}drdz-(\mu-\alpha)  \right|
\leq \varepsilon,
\end{aligned}
\end{equation*}
 and
  $$\textrm{dist}\ (\Omega_k^{(1)},\Omega_k^{(2)})\to \infty\quad \textrm{as}\quad k\to\infty.$$ 

\end{lem}
\begin{proof}[Proof of Lemma \ref{lem_concent}]
  We  first observe that the concentrated compactness lemma \cite[Lemma I.1]{Lions84a} holds even  for the half space $\Pi=\{(r,z)\in\mathbb{R}^2\,|\, r,z,\in\mathbb R ,\, r>0\}$. Then,
 we 
 apply the result 
into $$\tilde{\rho}_{n}:=\frac{\mu}{ {{\int_\Pi}\rho_n drdz}}\rho_n\,$$
due to
$\int_{\Pi} \tilde\rho_n\,drdz=\mu$. 
\end{proof}
 \subsection{Proof of compactness theorem (Theorem \ref{thm_cpt})} \ \\

 \begin{proof}[Proof of Theorem \ref{thm_cpt}]
Let $\mu\in(0,\infty)$. As in the previous sections, it is enough to show Theorem \ref{thm_cpt} for the case $\nu=\lambda=1$ since the general case will follow the scaling argument \eqref{scaling}. \\

  
  Let  $\{\xi_n\}$ be a sequence of non-negative axi-symmetric functions and let $\{a_n\}$ be  a sequence of positive numbers such that
 \begin{equation}\begin{split}\label{condn_cpt}
 &a_n\to 0\quad \mbox{as}\quad n\to\infty,\\
 &\limsup_{n\to\infty}\|\xi_n\|_{{1}}\leq1,\quad
 \lim_{n\to\infty}\int_{\{x\in\mathbb{R}^3\,|\,|\xi_n(x)-1|\geq a_n\}}\xi_n\,dx=0,\quad \color{black}
 \lim_{n\to\infty}\frac 1 2\|r^2\xi_n\|_{1} =\mu,  \\
&\sup_n\|\xi_n\|_2<\infty,\quad  \mbox{and}\quad \quad \lim_{n\to\infty}E[\xi_n]= {\mathcal{I}}_{\mu}.
\end{split}\end{equation}  
 We  set
  \begin{equation}\label{unif_l2}
  K_0=\sup_n\|\xi_n\|_{2}<\infty.
  \end{equation}    
By taking a subsequence if necessary  (still denoted by $\{\xi_n\}$  for simplicity),  we may assume 
\begin{equation}\label{condn_cpt2} \|\xi_n\|_1\leq 2, \quad \frac 1 2 \mu < \mu_n <2\mu \quad\mbox{for each}\quad n\geq 1,\end{equation}
   where we set $\mu_n=\frac 1 2 \|r^2\xi_n\|_{1}$.
 By setting $$\rho_n(r,z)=\pi  r^3\xi_n(r,z)\geq 0,\quad (r,z)\in\Pi,$$  we have
  $$\int_\Pi \rho_ndrdz=\pi\int_\Pi r^3 \xi_n drdz= \frac 1 2\int_{\mathbb{R}^3}r^2\xi_n \,dx= \mu_n\to \mu\quad\mbox{as} \quad n\to\infty.$$
 Now we can apply Lemma \ref{lem_concent} into the sequence $\{\rho_n\}$. 
 Then, for a certain subsequence  (still using the same parameter $n$),  one of the three cases,  (ii) Vanishing, (iii) Dichotomy, (i) Compactness, should occur. First, we shall exclude the cases (ii) Vanishing, (iii) Dichotomy  in order to get the case (i) Compactness.\\  

$\bullet$ Elimination of Case (ii)\ \textit{Vanishing:}\ \\

Let us suppose that the vanishing case (ii) happens. i.e. we assume
\begin{align}\label{vanishing}
\lim_{n\to\infty}\sup_{(r',\,z')\in \overline \Pi}\int_{
B_R(r',\,z')
}r^3\xi_n \,drdz=0\quad \textrm{for each}\ R>0.
\end{align} 
To contradict, it is enough to show 
\begin{equation}\label{en_conv_1}
 \lim_{n\to\infty}E[\xi_n]=0
\end{equation}
 since
this implies ${\mathcal{I}}_{\mu}= 0$ by \eqref{condn_cpt}, which gives  a contradiction to ${\mathcal{I}}_{\mu}>0$ by Theorem \ref{thm_exist_max} 
To show \eqref{en_conv_1},
we recall  the estimate \eqref{est_F} with $\tau=3/2$:
\begin{align}\label{G_est_}
G(r,z,r',z')\leq C_1 \frac{(rr')^{2}}{|(r,z)-(r',z')|^3},
\end{align} where $C_1>0$ is a universal constant. Then, for any $R\geq 1$,
 we decompose
\begin{align*}
E[\xi_n]&= \iint  \pi G(r,z,r',z')\xi_n(r',z')\xi_n(r,z)rr'\,dr'dz'drdz \\&
=\iint_{|(r,z)-(r',z')|\geq R} +
\iint\limits_{\substack{|(r,z)-(r',z')|< R, \\ G(r,z,r',z')<  C_1 Rr'^2r^2}}+\iint\limits_{\substack{|(r,z)-(r',z')|< R, \\ G(r,z,r',z')\geq C_1 Rr'^2r^2}}=:I_{R,n}+II_{R,n}+III_{R,n}.
\end{align*}
For the integral $I_{R,n}$, 
we have, by \eqref{condn_cpt2} and \eqref{G_est_}, 
\begin{align*}
I_{R,n}
\lesssim \frac{\mu_n^{2}}{ R^{3}}\lesssim  \frac{\mu^2}{ R^{3}}.
\end{align*} 
\color{black}
For the integral $II_{R,n}$, we estimate
\begin{align*}
II_{R,n}& \lesssim R
\int_{\Pi}\int_{B_R(r,z)} r'^3 r^3 \xi_n(r',z')\xi_n(r,z)\,dr'dz'drdz\\
&\lesssim R   \left( \int_{\Pi} r^3\xi_n(r,z)\,drdz\right)  \left(\sup_{(r,z)\in \Pi}\int_{
B_R(r,z)
}r'^3\xi_n(r',z')\,dr'dz'\right)\\
&\lesssim R  \mu   \left(\sup_{(r,z)\in \Pi}\int_{
B_R(r,z)
}r'^3\xi_n(r',z')\,dr'dz'\right). 
\end{align*}
For the integral $III_{R,n}$, we observe that 
the condition $G\geq  C_1 Rr'^2r^2$ together with
\eqref{G_est_}
 implies $|(r,z)-(r',z')|\leq R^{-1/3}\leq 1\leq R$.  
Thus we get
\begin{align*}&III_{R,n}\lesssim \iint\limits_{\substack{|(r,z)-(r',z')|\leq  R^{-1/3} }}   G(r,z,r',z')r'\xi_n(r',z')r\xi_n(r,z)\,dr'dz'drdz  \\
&= \int_\Pi r\xi_n(r,z) \left(\int_{B_{R^{-1/3}}(r,z)}  G(r,z,r',z')r'\xi_n(r',z')\,dr'dz' \right)drdz \\
&\leq   \int_\Pi r\xi_n(r,z)\left(\int_{B_{R^{-1/3}}(r,z)}  r'|G(r,z,r',z')|^2\,dr'dz'\right)^{1/2}\left(\int_{B_{R^{-1/3}}(r,z)}  r'|\xi_n(r',z')|^2\,dr'dz'\right)^{1/2}  drdz \\
&\lesssim  \|\xi_n\|_{2}  \int_\Pi r\xi_n(r,z)   \left(\int_{B_{R^{-1/3}}(r,z)}  r'|G(r,z,r',z')|^2\,dr'dz'   \right)^{1/2} drdz.  
\end{align*}
On the other hand, we recall the estimate 
\eqref{est_g_l_2} of Lemma \ref{lem_g_l_2}:  
 \begin{align*}
\int_{B_{M}(r,z)}  r'|G(r,z,r',z')|^2\,dr'dz'  \lesssim  (Mr^{4}+M^{7/2}r^{3/2}),\quad M>0, \quad(r,z)\in\Pi.
\end{align*}  
By plugging $M=R^{-1/3}$ into the above, we have
\begin{align*} 
\left(\int_{B_{R^{-1/3}}(r,z)}  r'|G(r,z,r',z')|^2\,dr'dz'   \right)^{1/2}\lesssim (1+r^2)  R^{-1/6},\quad R\geq 1, \quad (r,z)\in\Pi.
\end{align*} Hence, we get
 \begin{align*}
III_{R,n}&\lesssim  \|\xi_n\|_{2}  \int_\Pi r\xi_n(r,z)   \left(R^{-1/6}(1+r^{2}) \right) drdz.  \\
&\lesssim R^{-1/6}\|\xi_n\|_{2}  \left(
\|\xi_n\|_1+\|r^2\xi_n\|_{1}
\right) \lesssim R^{-1/6}   K_0  (1+\mu).
 \end{align*}  
Collecting the above estimates, we get, for $R\geq 1$ and for $n\geq 1$,
\begin{align*}
E[\xi_n] \lesssim \frac{\mu^{2}}{ R^{3}}
+  R \mu   \left(\sup_{(r,z)\in \Pi}\int_{
B_R(r,z)
}r'^3\xi_n(r',z')\,dr'dz'\right)
+  R^{-1/6}   K_0  (1+\mu).
\end{align*}
We take $\limsup_{n\to\infty}$ to get, by \eqref{vanishing},
\begin{align*}
\limsup_{n\to\infty} E[\xi_n]\lesssim\frac{\mu^{2}}{ R^{3}}
+  R^{-1/6}   K_0  (1+\mu),\quad R\geq 1.
\end{align*}  Then sending
 $R\to\infty$ implies \eqref{en_conv_1}. Thus the case  (ii) \textit{Vanishing} cannot occur.\\

 \color{black}

$\bullet$ Elimination of Case (iii)\ \textit{Dichotomy:}\ \\

Let us suppose that the dichotomy case (iii) happens with some
  $\alpha\in (0,\mu)$. \\
  
$\ast$ Step 1 - Applying Steiner symmetrization into each half:\\
We  fix $\varepsilon>0$. Then there exist  an integer $k_0\geq 1$ and sequences  $\{\xi_{1,n}\},\{\xi_{2,n}\}\subset L^{1}_w(\mathbb{R}^3)$
such that 
 \begin{equation}\label{disjoint}
  \xi_{1,n} ={1}_{\Omega_n^{(1)}}\xi_n,
\quad \xi_{2,n} ={1}_{\Omega_n^{(2)}}\xi_n\quad\mbox{
for some disjoint axi-symmetric   subsets }\quad \Omega_n^{(1)},\Omega_n^{(2)}\subset \mathbb{R}^3, 
\end{equation} 
\begin{align}
&||r^2\xi_{3,n}||_{1}+|\alpha_n-\alpha|+|\beta_n-(\mu-\alpha)|\leq \varepsilon\quad \textrm{for}\quad n\geq k_0,\quad\mbox{and} \label{ep_al}\\
&d_n \to\infty\quad \textrm{as}\quad n\to\infty, \label{dn_0}
\end{align} where we set $$
\xi_{3,n}=\xi_{n}-\xi_{1,n}-\xi_{2,n},\quad
d_n=\textrm{dist}\ (\Omega_n^{(1)},\Omega_n^{(2)}),\quad\alpha_n=\frac 1 2 {\int}r^2\xi_{1,n}\dd x,\quad \mbox{and}\quad \beta_n=\frac 1 2{\int}r^2\xi_{2,n}\dd x.$$
By choosing a subsequence (still using the same index $n$),
 we may assume that 
\begin{equation}\label{alpha_bar}
 \alpha_n\to \overline{\alpha}\quad \mbox{and}\quad \beta_n\to \overline{\beta}\quad\mbox{as}\quad  n\to\infty
\end{equation}
 for some $\overline{\alpha}\in[\alpha-\varepsilon,\alpha+\varepsilon]$ and $\overline{\beta}\in[(\mu-\alpha)-\varepsilon,(\mu-\alpha)+\varepsilon]$. 
Recalling
$\xi_{n}=\xi_{1,n}+\xi_{2,n}+\xi_{3,n}$,
we split the energy of $\xi_n$ into
\begin{align*}
E[\xi_n]&=\pi\iint rr'G(r,z,r',z')\xi_n(r,z)\xi_n(r',z')dr'dz'drdz \\
&=E[\xi_{1,n}]+E[\xi_{2,n}]+ 2\pi\iint rr'G(r,z,r',z') \xi_{1,n}(r,z)\xi_{2,n}(r',z')dr'dz'drdz\\
&\quad \quad
+\pi\iint rr'G(r,z,r',z') (2\xi_{n}(r,z)-\xi_{3,n}(r,z))  \xi_{3,n}(r',z')dr'dz'drdz\\
&=:E[\xi_{1,n}]+E[\xi_{2,n}]+ 
I_n+II_n. 
\end{align*}  
The estimate \eqref{est_F}  for $\tau=3/2$
 implies
\begin{align*}
I_n 
\lesssim \frac{(\mu_n)^{2}}{  (d_n)^{3}}\lesssim \frac{\mu^{2}}{  (d_n)^{3}} .
\end{align*} 
For the integral $II_n$, we apply \eqref{est_iint} to get
\begin{align*}
|II_n|
&\lesssim\left(\|r^2(2\xi_n-\xi_{3,n})\|_{1} 
+\|2\xi_n-\xi_{3,n}\|_{L^1\cap L^2}\right)  \|r^2\xi_{3,n}\|^{1/2}_{1}  \|\xi_{3,n}\|^{1/2}_{1}
\\
&\lesssim\left(\|r^2\xi_n\|_{1}
+\|\xi_n\|_{L^1\cap L^2}\right)    \|\xi_{n}\|^{1/2}_{1}  \varepsilon^{1/2}\lesssim(1+\mu+K_0)\varepsilon^{1/2},\quad n\geq k_0,
\end{align*} 
where we used \eqref{ep_al},
\eqref{unif_l2} and \eqref{condn_cpt2}.
Hence, we get
\begin{align*}
E[\xi_n] 
\leq E [\xi_{1,n}]+E [\xi_{2,n}]+C\frac{\mu^{2}}{ (d_n)^{3}}+C(1+\mu+K_0)\varepsilon^{1/2},
\end{align*} where $C>0$ is a universal constant. 
Now we take   Steiner symmetrization  $\xi_{i,n}^{*}$ (Proposition \ref{prop_steiner} with $p=2$) of $\xi_{i,n}$ for $i=1,2$ to see that  for $n\geq k_0$,
\begin{equation}\begin{split}\label{star_facts} 
&E[\xi_n] 
\leq E [\xi_{1,n}^*]+E [\xi_{2,n}^*]+C\frac{\mu^{2}}{ (d_n)^{3}}+C(1+\mu+K_0)\varepsilon^{1/2},\\
&||\xi_{1,n}^{*}||_1+||\xi_{2,n}^{*}||_1=
||\xi_{1,n}||_1+||\xi_{2,n}||_1=||\xi_{1,n}+\xi_{2,n}||_1\leq  \|\xi_n\|_{{1}}, \\ 
&\alpha_n=\frac 1 2 \|r^2\xi_{1,n}^*\|_{1},\quad \beta_n=\frac 1 2 \|r^2\xi_{2,n}^*\|_{1}.
\end{split}\end{equation} \\

$\ast$ Step 2 - Taking limit on $n\to\infty$:\\
We observe that  $\xi_{i,n}^{*} $ for  each $i$ and for each $n$ is
axi-symmetric and non-negative.
 We also have 
 \begin{equation}\label{dic_key}
    \int_{\{x\in\mathbb{R}^3\,|\,|\xi_{i,n}^*(x)-1|\geq a_n\}}\xi_{i,n}^*\,dx= \int_{\{x\in\mathbb{R}^3\,|\,|\xi_{i,n}(x)-1|\geq a_n\}}\xi_{i,n}\,dx\leq\int_{\{x\in\mathbb{R}^3\,|\,|\xi_{n}(x)-1|\geq a_n\}}\xi_{n}\,dx, 
 \end{equation}
 where the  equality follows from the property
\eqref{steiner_rearr} of 
 Steiner symmetrization (Proposition \ref{prop_steiner})
while
 the   inequality comes from  the form
$$  \xi_{i,n} ={1}_{\Omega_n^{(i)}}\xi_n$$ 
  in  \eqref{disjoint}.
  Thus, by \eqref{condn_cpt}, we have, for each $i=1,2$, 
 \begin{equation}\label{conv_star_outside}
 \lim_{n\to\infty}\int_{\{x\in\mathbb{R}^3\,|\,|\xi_{i,n}^*(x)-1|\geq a_n\}}\xi_{i,n}^*\,dx= 0.
 \end{equation}   
  We also have, by \eqref{steiner_rearr},
$$ ||\xi_{1,n}^{*}||_2^2+||\xi_{2,n}^{*}||_2^2=
||\xi_{1,n}||_2^2+||\xi_{2,n}||_2^2=||\xi_{1,n}+\xi_{2,n}||_2^2\leq  \|\xi_n\|^2_{2}.$$ Thus,
by \eqref{unif_l2}, we have  the uniform $L^2-$bound  $$\quad 
\sup_n\left(||\xi_{1,n}^{*}||_2+ ||\xi_{2,n}^{*}||_2\right)\lesssim K_0.$$  
By choosing a subsequence (still denoted by $\{\xi_{i,n}^{*}\}$), we obtain  
\begin{equation}\label{dicto_weak}
\xi_{i,n}^{*} \rightharpoonup \overline{\xi}_i  \quad\mbox{in} \quad L^2(\mathbb{R}^3)\quad \mbox{as} \quad n\to\infty 
\end{equation}
 for some non-negative axi-symmetric $\overline{\xi}_i\in L^{2}(\mathbb{R}^3)$ for $i=1,2$.  We note 
 \begin{equation}\label{l2_sum}
 ||\overline{\xi}_i||_2\leq \liminf_{n\to\infty}||\xi_{i,n}^{*}||_2\lesssim K_0\quad\mbox{for}\quad i=1,2.
\end{equation}   
We observe, for any bounded set $U\subset\mathbb{R}^3$, we have \begin{equation}\label{dicto_weak_l1}
 \xi_{i,n}^{*} \rightharpoonup \overline{\xi}_i  \quad\mbox{in} \quad L^1(U)\quad \mbox{as}\quad n\to\infty 
\end{equation} due to
$L^\infty(U)\subset L^2(U)$. Thus we get, for   bounded  $U\subset\mathbb{R}^3$, 
 $$||\overline{\xi}_i||_{L^1(U)}\leq \liminf_{n\to\infty}||\xi_{i,n}^{*}||_{L^1(U)}\leq \liminf_{n\to\infty}||\xi_{i,n}^{*}||_{1},  $$ which implies 
  $$||\overline{\xi}_i||_{1}\leq   \liminf_{n\to\infty}||\xi_{i,n}^{*}||_{1},  $$
 for $i=1,2$. Thus
we have, by \eqref{condn_cpt} and \eqref{star_facts},
\begin{equation}\label{l1_sum}
||\overline{\xi}_1||_{1}+||\overline{\xi}_2||_{1}  \leq \liminf_{n\to\infty}||\xi_{1,n}^{*}||_{1}+
\liminf_{n\to\infty}||\xi_{2,n}^{*}||_{1} \leq
\limsup_{n\to\infty}\left(||\xi_{1,n}^{*}||_{1}+
 ||\xi_{2,n}^{*}||_{1}\right)
 \leq\limsup_{n\to\infty} ||\xi_{n}||_{1}\leq 1. 
\end{equation} 
 By a similar argument, we have
\begin{equation*}
 r^2\xi_{i,n}^{*} \rightharpoonup r^2\overline{\xi}_i  \quad\mbox{in} \quad L^1(U)\quad \mbox{as } n\to\infty 
\end{equation*} for   bounded   $U\subset\mathbb{R}^3$
   so that
 \begin{equation*}
  \|r^2\overline{\xi}_i||_{1}\leq   \liminf_{n\to\infty}||r^2\xi_{i,n}^{*}||_{1},\quad i=1,2.
\end{equation*} Thus, by recalling
  \eqref{alpha_bar} and \eqref{star_facts}, we have
\begin{equation}\label{alpha_bar_conv}
\overline{\alpha}\geq \frac 1 2\|r^2\overline{\xi}_1\|_{1},\quad \overline{\beta}\geq \frac 1 2 \|r^2\overline{\xi}_2\|_{1}.
\end{equation}  
 
   We claim\begin{equation}\label{l_infty_bdd}
   \|\overline{\xi}_i\|_\infty\leq 1\quad  \mbox{for}\quad i=1,2.
\end{equation}  To prove, let us suppose that the case $\|\overline{\xi}_i\|_\infty>1$ happens either  $i=1$  or $2$. We may assume that the case for $i=1$ occurs since the other case can be solved in  the same way. Then there exist  a measurable subset $U\subset\mathbb{R}^3$ and a small constant $\eta>0$ such that
$$|U|>0\quad\mbox{and}\quad
\overline{\xi}_1\geq 1+\eta\quad \mbox{  in}\quad U.$$ We may assume that $U$ is bounded in $\mathbb{R}^3$. 
From the weak convergence \eqref{dicto_weak_l1},
 we have 
\begin{equation}\label{lower_eta}
(1+\eta)  {|U|}\leq ||\overline{\xi}_1||_{L^1(U)}\leq \liminf_{n\to\infty}||\xi_{1,n}^{*}||_{L^1(U)}.   
\end{equation}
On the other hand, by setting
$A^*_{1,n}=\{x\in\mathbb{R}^3\,|\,|\xi_{1,n}^*(x)-1|\geq a_n\}$,  we estimate
$$||\xi_{1,n}^{*}||_{L^1(U)}=\int_{U\cap A^*_{1,n}}  \xi^*_{1,n}  \,dx+\int_{U\setminus A^*_{1,n}}  \xi^*_{1,n} \,dx
\leq \int_{A^*_{1,n}}  \xi^*_{1,n}  \,dx+
(1+a_n)   |U|.
$$ By taking $\limsup_{n\to\infty}$, we get,  by \eqref{conv_star_outside} and  \eqref{condn_cpt},
$$\limsup_{n\to\infty}||\xi_{1,n}^{*}||_{L^1(U)}
\leq \limsup_{n\to\infty}\int_{A^*_{1,n}}  \xi^*_{1,n}  \,dx+
\limsup_{n\to\infty}(1+a_n)   |U|=|U|.
$$ It contradicts to \eqref{lower_eta}. Thus we have \eqref{l_infty_bdd}.\\

Next, we claim
\begin{align}\label{conv_ki_e}
\lim_{n\to\infty}E[\xi_{i,n}^{*}]=E[\overline{\xi}_i]\quad \mbox{for}\quad i=1,2.
\end{align} Indeed,
since $\xi_{i,n}^{*}$ is  from  Steiner symmetrization (Proposition \ref{prop_steiner}) by definition, 
 it  satisfies
    {the monotonicity }condition  \eqref{cond_sym} 
    and 
  $$\|\xi_{i,n}^{*}\|_{L^1\cap L^2} +\|r^2\xi_{i,n}^{*}\|_{1}=\|\xi_{i,n}\|_{L^1\cap L^2}+\|r^2\xi_{i,n}\|_{1}\leq \|\xi_{n}\|_{L^1\cap L^2}
   +\|r^2\xi_{n}\|_{1}
\lesssim 1+K_0+\mu   
  $$ by \eqref{unif_l2}, \eqref{condn_cpt2}, \eqref{disjoint}.
  Since we have the weak-convergence \eqref{dicto_weak}, we can apply
Lemma \ref{lem_energy_conv} into the sequence $\{\xi_{i,n}^{*}\}$ for $i=1,2$ so that we obtain
the convergence  \eqref{conv_ki_e} of the kinetic energy.\\

 \color{black}
 
By  sending $n\to\infty$ in \eqref{star_facts} and by using  \eqref{condn_cpt}, \eqref{conv_ki_e}, \eqref{dn_0},
 \eqref{l1_sum}, \eqref{l2_sum}, \eqref{alpha_bar_conv}, and \eqref{l_infty_bdd}, we obtain
\begin{equation}\begin{split}\label{eq_ep_fct}
&0\leq \overline{\xi}_i\leq 1\quad  \quad i=1,2,\\
&{\mathcal{I}}_{\mu}\leq E[\overline{\xi}_1]+E[\overline{\xi}_2] +C(1+\mu+K_0)\varepsilon^{1/2},\\
&||\overline{\xi}_1||_1+||\overline{\xi}_2||_1\leq 1, \quad 
||\overline{\xi}_1||_2+||\overline{\xi}_2||_2\leq  C K_0, \color{black} \\
&\overline{\alpha}\geq \frac 1 2\|r^2\overline{\xi}_1\|_{1},\quad \overline{\beta}\geq \frac 1 2\|r^2\overline{\xi}_2\|_{1}.
\end{split}\end{equation} \\

$\ast$ Step 3 - Dropping the parameter $\varepsilon$:\\
By summarizing  what we have done in Step 1 and Step 2, for \textit{each} $\varepsilon>0$, there exist
 functions $\overline\xi_1=\overline\xi^\varepsilon_1$, $\overline\xi_2=\overline\xi^\varepsilon_2$ and the constants $\overline\alpha=\overline\alpha^\varepsilon$, $\overline\beta=\overline\beta^\varepsilon$ satisfying \eqref{eq_ep_fct}
   while the constants $C>0$ in \eqref{eq_ep_fct} are independent of the choice of $\varepsilon>0$. 
Then we can apply a similar argument in Step 2 for the sequences $$\{\overline{\xi}^{\varepsilon_m}_i\}_{m=1}^\infty,\quad \varepsilon_m=\frac 1 m$$ for $i=1,2$ in order to obtain the following claim: \\

There exist axi-symmetric non-negative functions $\hat{\xi}_i\in L^2(\mathbb{R}^3)$ for $i=1,2$ such that 
 \begin{equation}\begin{split}\label{eq_hat_fct}
&0\leq \hat{\xi}_i\leq 1,\quad i=1,2,\\
&{\mathcal{I}}_{\mu}\leq E[\hat{\xi}_{1}]+E[\hat{\xi}_{2}],\\
&||\hat{\xi}_{1}||_1+||\hat{\xi}_{2}||_1\leq 1, \\
&\alpha\geq \frac 1 2 \|r^2\hat{\xi}_1\|_{1},\quad \mu-\alpha\geq  \frac 1 2 \|r^2\hat{\xi}_2\|_{1}.
 \end{split}\end{equation} 
 Indeed,  by the uniform $L^2$-bound in \eqref{eq_ep_fct}, we can first extract  subsequential weak-limits  $\hat{\xi}_i\in L^2(\mathbb{R}^3)$ in $L^2$ for $i=1,2, $ i.e.  $$\overline{\xi}_{i}^{\varepsilon_m}\rightharpoonup \hat{\xi}_i\quad\mbox{in}\quad L^{2}\quad\mbox{as}\quad m\to\infty\quad\mbox{(by reindexing)}.$$
Thanks to \eqref{eq_ep_fct}, such limits   satisfies 
 the same  pointwise estimate: $$0\leq \hat{\xi}_i\leq 1.$$ 
We also have 
$\overline{\xi}_{i}^{\varepsilon_m}\rightharpoonup \hat{\xi}_i$ in $L^{1}(U)$ for any bounded set $U\subset\mathbb{R}^3$ as we proved \eqref{dicto_weak_l1}. Thus, as in \eqref{l1_sum}, we get the same $L^1$-bound: $$||\hat{\xi}_{1}||_1+||\hat{\xi}_{2}||_1\leq 1.$$ Similarly, 
since 
   $\overline{\alpha}^{\varepsilon_m}\in[\alpha-\varepsilon_m,\alpha+\varepsilon_m]$ and $\,\overline{\beta}^{\varepsilon_m}\in[(\mu-\alpha)-\varepsilon_m,(\mu-\alpha)+\varepsilon_m]$, 
we get
$$\alpha\geq \frac 1 2 \|r^2\hat{\xi}_1\|_{1},\quad \mu-\alpha\geq  \frac 1 2 \|r^2\hat{\xi}_2\|_{1}.$$  
Since $ \overline{\xi}_i^{\varepsilon_m}$    satisfies  {the monotonicity }condition \eqref{cond_sym} for each $m$ and for $i=1,2$, 
we can apply Lemma \ref{lem_energy_conv} to the sequences  to obtain $E[\overline{\xi}^{\varepsilon_m}_i]\to E[\hat{\xi}_{i}]$ for $i=1,2$. Applying the convergence  into \eqref{eq_ep_fct}, we get
$${\mathcal{I}}_{\mu}\leq E[\hat{\xi}_{1}]+E[\hat{\xi}_{2}].$$ Hence, we get the claim \eqref{eq_hat_fct}.\\

$\ast$ Step 4: Contradiction to the dichotomy case.\\
We first observe that if both $\hat{\xi}_1$ and $\hat{\xi}_2$ are identically zero, then we have $${\mathcal{I}}_{\mu}\leq E[\hat{\xi}_{1}]+E[\hat{\xi}_{2}]=0$$ which is a contradiction to ${\mathcal{I}}_{\mu}>0$ in Theorem 
\ref{thm_exist_max}.
 Therefore, either 
 $\hat{\xi}_1$ or  $\hat{\xi}_2$ should be nontrivial. 
Without loss of generality, we  may   assume  $\hat{\xi}_1\nequiv 0$. Then we have 
 \begin{equation}\label{nu_1}
  0<||\hat{\xi}_1||_1\leq 1-||\hat{\xi}_2||_1=:\nu_1.
 \end{equation}
 By using Theorems \ref{thm_exist_max} and  \ref{thm_max_is_ring},
 we   take a compactly supported function $$\zeta_1\in {\mathcal{S}}_{\alpha,\nu_1,1}.$$ 
We recall  $ {\mathcal{S}}_{\alpha,\nu_1,1}= {\mathcal{S}}''_{\alpha,\nu_1,1}$
  (by Theorem \ref{thm_exist_max}) and $ \hat{\xi}_1\in \mathcal{P}''_{\alpha,\nu_1,1}$ (by \eqref{eq_hat_fct} and \eqref{nu_1}). Thus we have 
  $E[\zeta_1]\geq E[\hat{\xi}_1].$ 
 Now we have
\begin{align*}
&0\leq  {\zeta}_1\leq 1,\quad 0\leq \hat{\xi}_2\leq 1,\\
& {\mathcal{I}}_{\mu}\leq E[\zeta_{1}]+E[\hat{\xi}_{2}],\\
&||\zeta_{1}||_1+||\hat{\xi}_{2}||_1\leq 1, \\
&\alpha=  \frac 1 2 \|r^2\zeta_1\|_{1},\quad \mu-\alpha\geq  \frac 1 2 \|r^2\hat{\xi}_2\|_{1}.
\end{align*}
 
Next, we observe that if $\hat{\xi}_2\equiv 0$, we have $${\mathcal{I}}_{\mu}\leq E[\zeta_1]={\mathcal{I}}_{\alpha,\nu_1,1}\leq {\mathcal{I}}_{\alpha,1,1} =\mathcal{I}_{\alpha},$$ where the last inequality comes from $\nu_1\leq 1$.  Thus it is a contradiction to $\mathcal{I}_{\alpha}<\mathcal{I}_{\mu}$ by Lemma \ref{lem_strict}. Thus we may assume   $\hat{\xi}_2\nequiv 0$, which implies
$$0<||\hat{\xi}_2||_1\leq 1-||\zeta_1||_1=:\nu_2.$$
By using Theorems \ref{thm_exist_max} and  \ref{thm_max_is_ring} again,
we   take a compactly supported function $$\zeta_2\in {\mathcal{S}}_{\mu-\alpha,\nu_2,1}.$$ As before, we have
 $E[\zeta_2]\geq E[\hat{\xi}_2]$ due to
 $ {\mathcal{S}}_{\mu-\alpha,\nu_2,1}= {\mathcal{S}}''_{\mu-\alpha,\nu_2,1}$ and
 $\,\hat{\xi}_2\in \mathcal{P}''_{\mu-\alpha,\nu_2,1}$.
 In sum, we have
\begin{equation}\begin{split}\label{zeta_imp}
&0\leq  {\zeta}_i\leq 1\quad \mbox{for}\quad i=1,2,\\
&{\mathcal{I}}_{\mu}\leq E[\zeta_{1}]+E[\zeta_{2}],\\
&||\zeta_{1}||_1+||\zeta_{2}||_1\leq 1, \\
&\alpha=\frac 1 2 \|r^2\zeta_1\|_{1},\quad \mu-\alpha= \frac 1 2 \|r^2\zeta_2\|_{1}.
\end{split}\end{equation} 
By  a translation in $z$-variable (if necessary), we may assume that $$\textrm{spt}\ \zeta_1\cap \textrm{spt}\ \zeta_2=\emptyset.$$ Then we have
$$0\leq (\zeta_1+\zeta_2)\leq 1,\quad \int (\zeta_1+\zeta_2)\,dx\leq 1,\quad\mbox{and}\quad \frac{1}{2}\int r^2 (\zeta_1+\zeta_2)\,dx=\mu,$$ which implies
$(\zeta_1+\zeta_2)\in \mathcal{P}'_\mu$ so that
$$E[(\zeta_1+\zeta_2)]\leq 
 \mathcal{I}_{\mu}'= \mathcal{I}_{\mu}$$ by Theorem \ref{thm_exist_max}. 
Together with
\eqref{zeta_imp}, we get
\begin{align*}
{\mathcal{I}}_{\mu}&\leq E[\zeta_1]+E[\zeta_2]
 =E[\zeta_1+\zeta_2]-2\pi\iint rr'G(r,z,r',z')\zeta_1(r,z)\zeta_2(r',z')dr'dz'drdz \\
&\leq {\mathcal{I}}_{\mu}-2\pi\iint rr'G(r,z,r',z')\zeta_1(r,z)\zeta_2(r',z')dr'dz'drdz.
\end{align*} 
Hence, either $\zeta_1\equiv 0\,$  or  $\,\zeta_2\equiv 0$ holds due to
$G(r,z,r',z')>0$ a.e. 
This contradicts to 
 the last line of \eqref{zeta_imp} due to $\alpha\in(0,\mu)$. Thus the case (iii) \textit{Dichotomy} cannot occur.\\

$\bullet$ Case (i) \textit{Compactness}:\ \\

Up to now, we have shown that the case (i) \textit{Compactness} should occur. That means: \\

There exists a sequence $\{(\tilde r_n,\tilde z_n)\}\subset \overline{\Pi}$ such that for arbitrary $\varepsilon>0$, there exist $R=R(\varepsilon)>0$ and $k_0=k_0(\varepsilon)\geq1$ such that 

\begin{align}\label{cpt_original}
\frac 1 2\int_{
T_n^\varepsilon
}r^2\xi_n\,dx\geq \mu-\varepsilon,\qquad \textrm{for all}\ n\geq k_0(\varepsilon),
\end{align} where axi-symmetric $T^\varepsilon_n \subset \mathbb{R}^3$ is defined by  $$T^\varepsilon_n=T_{R(\varepsilon)}(\tilde r_n, \tilde z_n)=
\{
x\in\mathbb{R}^3\,|\, |(r,z)-(\tilde r_n, \tilde z_n)|<R(\varepsilon)
\} $$ when
$x_1^2+x_2^2=r^2,\,x_3=z$ as defined in  \eqref{defn_ball}.\\


$\ast$ Step 1 - Boundedness of $\{\tilde r_n\}$:\\
We note that there are only two cases whether (a) $\limsup_{n\to\infty} \tilde r_n=\infty$ or (b) $\sup_{n }\tilde r_n<\infty$. We shall first show that the case (a) cannot occur.\\

Let us suppose the case (a) $\limsup_{n\to\infty} \tilde r_n=\infty$ happens. We may assume that $\lim_{n\to\infty}\tilde r_n=\infty$ by choosing a subsequence (and by reindexing). We claim  $$\lim_{n\to\infty}E[\xi_n]=0.$$ This claim implies $\mathcal{I}_{\mu} =0$ by \eqref{condn_cpt},  which is a contradiction to $\mathcal{I}_{\mu}>0$ in 
Theorem \ref{thm_exist_max}.
To prove the claim, we set $
\psi_n(x)=\mathcal{G}[\xi_n].$ Then, for $\varepsilon>0$, we decompose
\begin{align*}
 E[\xi_n]=\frac{1}{2}{\int} \psi_n\xi_n\dd x=\frac{1}{2}\int_{
T_n^\varepsilon
 }+\frac{1}{2}\int_{\mathbb{R}^{3}\setminus T_n^\varepsilon
 }=:I_{n,\varepsilon}+II_{n,\varepsilon}, 
\end{align*} 
For the first term $I_{n,\varepsilon}$, by using  \eqref{est_psi}, we estimate,
for $n\geq 1$,
\begin{align*}
I_{n,\varepsilon}
\leq \frac 1 2\left\|\frac{\psi_n}{r}\right\|_{\infty}\int_{
T^\varepsilon_n
}r\xi_n\dd x
\lesssim (1+K_0+\mu)\left(\sup_{|r-\tilde r_n|<R(\varepsilon)}\frac 1 r\right)\int_{
T^\varepsilon_n
}r^2\xi_n\dd x 
\lesssim \frac{(1+K_0+\mu)\mu}{\left(\tilde r_n-R(\varepsilon)\right) } 
\end{align*}  by using
\eqref{unif_l2} and \eqref{condn_cpt2}, which gives $ I_{n,\varepsilon}\to0$ as $n\to\infty$.
For the second term $II_{n,\varepsilon}$, by using   H\"older's inequality,
for $n\geq k_0(\varepsilon)$,
\begin{equation}\begin{split}\label{II_n_epsilon_comp}
II_{n,\varepsilon}
&\leq \frac 1 2\left\|\frac{\psi_n}{r}\right\|_{\infty} \int_{\mathbb{R}^{3}\setminus
T^\varepsilon_n
}r\xi_n\dd x 
 \leq \frac 1 2\left\|\frac{\psi_n}{r}\right\|_{\infty}\left(\int_{\mathbb{R}^{3}\setminus
T^\varepsilon_n
}r^2\xi_n\dd x\right)^{1/2}\left(\int \xi_n\dd x\right)^{1/2} \\&
\lesssim (1+K_0+\mu)(\mu_n-(\mu-\varepsilon))^{1/2}, 
 \end{split}\end{equation} which gives
$ \limsup_{n\to\infty}II_{n,\varepsilon}\lesssim (1+K_0+\mu)\varepsilon^{1/2}$.
  Collecting the above estimates, we get, for any $\varepsilon>0$,
    \begin{align*}
  \limsup_{n\to\infty}E[\xi_n] \lesssim (1+K_0+\mu)\varepsilon^{1/2},
\end{align*} which implies $$E[\xi_n]\to 0\quad\mbox{as}\quad n\to\infty.$$ Thus the case (a) cannot occur, and the case (b)  $$\sup_{n }\tilde r_n<\infty$$ should occur. \\

  
  
$\ast$ Step 2 - Reformulated  goal via translations:\\
 We may assume that $\tilde r_n=0$ for $n\geq1$ by replacing $R(\varepsilon)$ with $$
 \left(R(\varepsilon)+\sup_{n  }\tilde r_n\right)>0$$ due to the set inclusion
  $ T^\varepsilon_n\subset
  \{
x\in\mathbb{R}^3\,|\, |(r,z)-(0, \tilde z_n)|< \left(R(\varepsilon)+\sup_{n  }\tilde r_n\right)
\}$.  
   Now we have $(\tilde r_n, \tilde z_n)=(0, \tilde z_n)\in \partial\Pi$ for $n\geq 1$.
By redefining $\xi_n$ by  translation of $\xi_n$ in $x_3$-variable, 
  our goal \eqref{conclu_cpt} is transformed  into (by setting $c_n=\tilde z_n$) the following new goal:
  \begin{equation}\label{new_goal}
  \mbox{ to extract a subsequence  of}\, \{\xi_{n}\} \mbox{ converging to }\,
  \xi   \mbox{ in }\,L^1_w\, \mbox{ for some }\,\xi\in\mathcal{S}_\mu.
  \end{equation}
 We note that this translation in $x_3-$variable does not
  break \eqref{condn_cpt}, \eqref{unif_l2}, \eqref{condn_cpt2}. Moreover,    
\eqref{cpt_original} can be re-written by the following form:\\ 

  
 For arbitrary $\varepsilon>0$, there exist $R=R(\varepsilon)>0$ and $k_0=k_0(\varepsilon)\geq1$ such that 

\begin{align}\label{st_cpt}
\mu_n\geq \frac 1 2\int_{B^\varepsilon }r^2\xi_n\dd x\geq (\mu-\varepsilon),\qquad \textrm{for all}\ n\geq k_0(\varepsilon),
\end{align} by setting  $B^\varepsilon=B_{R(\varepsilon)}(0)=
T_{R(\varepsilon)}(0,0
)= \{
x\in\mathbb{R}^3\,|\, |x |<R(\varepsilon)
\}$ (see the definition \eqref{defn_ball}). \\

 
$\ast$ Step 3 - Extracting a weak-limit:\\
 Since the sequence $\{\xi_n\}$ is uniformly bounded in $L^{2}$ from \eqref{unif_l2}, by choosing a subsequence (still denoted by $\{\xi_n\}$), we get $$\xi_n\rightharpoonup \xi\quad\mbox{in} \quad L^{2}(\mathbb{R}^3)\quad \mbox{as}\quad n\to\infty$$ for some non-negative axi-symmetric function $\xi\in L^2(\mathbb{R}^3)$. The weak convergence in $L^2$ implies, for any bounded subset $U\subset\mathbb{R}^3$,
$$\xi_n\rightharpoonup \xi\quad\mbox{and}\quad 
r^2\xi_n\rightharpoonup r^2\xi
\quad\mbox{in} \quad L^{1}(U)\quad \mbox{as}\quad n\to\infty.$$ Hence, we get, by \eqref{condn_cpt}, \eqref{unif_l2},  \eqref{st_cpt},
\begin{equation}\label{othercond}
\|\xi\|_2\leq K_0,\quad  \int\xi\,dx\leq 1,\quad 
(\mu-\varepsilon)\leq\frac 1 2{\int_{B^\varepsilon}}r^2\xi \dd x\leq \mu
\quad\mbox{for}\quad \varepsilon>0, \quad \mbox{and}\quad
\frac 1 2{\int}r^2\xi \dd x= \mu.
\end{equation}\\

$\ast$ Step 4 - Verifying the pointwise bound:\\
   We claim   
\begin{equation}\label{l_infty_bdd_}
  \|\xi\|_\infty\leq 1.
\end{equation}   
  Indeed, we can follow the same approach in the proof of \eqref{l_infty_bdd} in the dichotomy case. For a contradiction, we assume that there is a bounded set $U\subset \mathbb{R}^3$ and a constant $\eta>0$ such that
   $$|U|>0\quad\mbox{and}\quad
{\xi}\geq 1+\eta\quad \mbox{in}\quad U.$$ Then we get
\begin{equation}\label{lower_eta_}
(1+\eta)  {|U|}\leq ||{\xi}||_{L^1(U)}\leq \liminf_{n\to\infty}||\xi_{n}||_{L^1(U)}.   
\end{equation}
   On the other hand,  by setting
\begin{equation}\label{defn_A_n}
A_n=\{x\in\mathbb{R}^3\,|\,|\xi_{n}(x)-1|\geq a_n\},
\end{equation}
 we get
$$||\xi_{n}||_{L^1(U)}=\int_{U\cap A_n}  \xi_{n}  \,dx+\int_{U\cap A^c_n}  \xi_{n} \,dx
\leq \int_{A_n}  \xi_{n}  \,dx+
(1+a_n)   |U|.
$$ By taking $\limsup_n$ and by using   \eqref{condn_cpt}, we get
$$\limsup_{n\to\infty}||\xi_{n}||_{L^1(U)}
\leq  |U|.
$$  It contradicts to \eqref{lower_eta_}. We have proved \eqref{l_infty_bdd_}. Thanks to  \eqref{othercond}, 
we know 
\begin{equation}\label{conc_p'}
\xi\in\mathcal{P}'_\mu
\end{equation}	
(recall  the definition \eqref{defn_prime_class} of $\mathcal{P}'_\mu$ in Subsection \ref{subsec_other_cl}).\\

$\ast$ Step 5 - Establishing convergence in energy:\\
Next, we claim 
\begin{align}\label{conv_en}
\lim_{n\to \infty}E[\xi_n]=E[\xi].     
\end{align}  
Indeed, we estimate, as in \eqref{II_n_epsilon_comp},
\begin{align*}
 \int_{\mathbb{R}^{3}\setminus {B^\varepsilon}}\xi_n \mathcal{G}[\xi_n]\,dx 
&\leq 
 \left\|\frac{\mathcal{G}[\xi_n]}{r}\right\|_{\infty} \int_{\mathbb{R}^{3}\setminus {B^\varepsilon}}r\xi_n\,dx  \leq 
 \left\|\frac{\mathcal{G}[\xi_n]}{r}\right\|_{\infty}\left(\int_{\mathbb{R}^{3}\setminus {B^\varepsilon}}r^2\xi_n\,dx\right)^{1/2}\left(\int \xi_n\,dx\right)^{1/2}\\&
\leq C (1+K_0+\mu)  \left(\mu_n-(\mu-\varepsilon)\right)^{1/2},\quad n\geq k_0(\varepsilon),
\end{align*} by  \eqref{est_psi}, \eqref{unif_l2},   \eqref{condn_cpt2}, and \eqref{st_cpt}. In the same way,  we obtain
\begin{align*}
 \int_{\mathbb{R}^{3}\setminus {B^\varepsilon}}\xi \mathcal{G}[\xi]\,dx &
\leq C(1+K_0+\mu) \varepsilon^{1/2}
\end{align*} thanks to  the estimate \eqref{othercond}.
Then, by using \eqref{est_en_diff} of Lemma \ref{lem_en_diff},   we can estimate  difference in energy for $n\geq k_0(\varepsilon)$ by
\begin{align*}
 \left| E[\xi_n]-E[\xi]\right|
&\leq \frac{1}{4\pi}\left| \int_{{  B^\varepsilon} }\int_{{  B^\varepsilon}  }  {G}(x,y)\Big(\xi_n(x)\xi_n(y)-\xi(x)\xi(y)\Big)\,dxdy \right|
\\&\quad 
+C(1+K_0+\mu) |\mu_n-\mu|^{1/2}+C(1+K_0+\mu) \varepsilon^{1/2}.
\end{align*} 
Since $ G \in L^{2}({ B^\varepsilon}\times { B^\varepsilon})$  by Lemma \ref{lem_g_l_2} and $\xi_n(x)\xi_n(y)\rightharpoonup \xi(x)\xi(y)$ in $L^{2}({ B^\varepsilon}\times { B^\varepsilon})$,  we get  the claim \eqref{conv_en} by sending $n\to\infty$ first and then $\varepsilon \to0$.\\


$\ast$ Step 6 - It is of a patch-type:\\
Since the  convergence \eqref{conv_en} implies $E[\xi]={\mathcal{I}}_\mu=\mathcal{I}'_\mu$ by \eqref{condn_cpt},
  we obtain $\xi\in \mathcal{S}'_{\mu}$ due to   \eqref{conc_p'}.
By using ${\mathcal{S}}_\mu=\mathcal{S}'_\mu$ from Theorem \ref{thm_exist_max}, 
we get $$\xi\in {\mathcal{S}}_\mu,$$ which implies 
$\xi\in \mathcal{P}_\mu$. That means 
$$\xi={{1}}_A\quad\mbox{for some axi-symmetric measurable set } A\subset\mathbb{R}^3.$$\\ 

$\ast$ Step 7 - Weak convergence $\sqrt{{r^2}\xi_n}\rightharpoonup  \sqrt{{r^2}\xi} \,$ in $\, L^2(\mathbb{R}^3)$:\\
We observe, as $n\to \infty$, \begin{equation}\label{conv_norm_2}
\frac 1 {\sqrt{2}}\|\sqrt{{r^2}\xi_n}\|_{2}=\sqrt{\mu_n} \quad\to \quad\sqrt{\mu}= \frac 1 {\sqrt{2}}\|\sqrt{{r^2}\xi}\|_{2},
\end{equation}
 which implies
$\sup_n\|\sqrt{{r^2}\xi_n}\|_{2}<\infty$. Thus there exists a non-negative axi-symmetric function $ g\in L^2(\mathbb{R}^3)$ such that 
$$\sqrt{{r^2}\xi_n}\rightharpoonup  g$$ in $L^2(\mathbb{R}^3)$ (by reindexing).
We claim
\begin{equation}\label{claim_g}
 g=\sqrt{{r^2}\xi}\quad a.e.\quad \mbox{in}\quad \mathbb{R}^3. 
\end{equation}
Indeed, for any bounded set $U\subset\mathbb{R}^3$,
we have $$\sqrt{{r^2}\xi_n}\rightharpoonup  g\quad\mbox{in}\quad L^1(U).$$
On the other hand, for any $\phi\in L^\infty(U)$,    we can prove
\begin{equation}\label{weakconv} \int_U\sqrt{{r^2}\xi_n}\phi dx 
\rightarrow   \int_U\sqrt{{r^2}\xi}\phi dx\quad\mbox{as}\quad n\to \infty
\end{equation} in the following way:\\

 First we estimate, by using $\sqrt{\xi}=\xi$,  
\begin{equation*}
\begin{split}
&\left|\int_U\sqrt{{r^2}\xi_n}\phi dx 
-  \int_U\sqrt{{r^2}\xi}\phi dx\right| 
 =
\left|\int_Ur\phi\left(\sqrt{\xi_n}-\xi\right) dx 
 \right| \\
 &\leq    \int_Ur|\phi|  \left|\sqrt{\xi_n}-\xi_n\right| dx  
  +\left|\int_Ur\phi  \left({\xi_n}-\xi\right) dx 
 \right|  =:I_n+II_n. 
\end{split}
\end{equation*}   
We observe $II_n\to 0 $ {as} $ n\to \infty$ due to
  $(r\phi)\in L^2(U)$ and   $\xi_n\rightharpoonup \xi$ in $L^{2}(U)$.   
For $I_n$, 
by recalling the definition \eqref{defn_A_n} of $A_n$,
we estimate 
  \begin{equation*}
\begin{split}
&I_n
=  \int_{U\cap A_n} 
+ \int_{U\cap A^c_n}  
\leq  \int_{U\cap A_n} r|\phi| \left(\sqrt{\xi_n}+\xi_n\right)dx
+ \int_{U\cap A^c_n}r|\phi| \left(|\sqrt{\xi_n}-1|+|\xi_n-1|\right)dx \\
&\leq  
\left(\sup_{r\in U} r\right)  \|\phi\|_{L^\infty(U)} \left(\sqrt{|U|} \left(\int_{A_n}\xi_ndx\right)^{1/2}+ \int_{A_n}\xi_ndx\right)
+ 2\int_{U\cap A^c_n}r|\phi|  |\xi_n-1| dx \\
&\leq \left(\sup_{r\in U} r\right)  \|\phi\|_{L^\infty(U)} \left(\sqrt{|U|} \left(\int_{A_n}\xi_ndx\right)^{1/2}+ \int_{A_n}\xi_ndx\right)
+ 2 \left( \sup_{r\in U} r\right)  \|\phi\|_{L^\infty(U)} a_n {|U|}.
\end{split}
\end{equation*}  Thus, by \eqref{condn_cpt}, we have
$I_n\to 0$ as $n\to\infty$.  Now we have \eqref{weakconv}, which implies
$$
g= \sqrt{r^2\xi}\quad a.e. \quad \mbox{in}\quad U.
$$ for any bounded set $U\subset\mathbb{R}^3$. Hence we get the claim \eqref{claim_g}.\\

   $\ast$ Step 8 - Strong convergence $ {{r^2}\xi_n}\to  {{r^2}\xi}\,$ in $\,L^1(\mathbb{R}^3)$:\\
Since we have   $\sqrt{{r^2}\xi_n}\rightharpoonup  \sqrt{{r^2}\xi}\,$ in $L^2(\mathbb{R}^3)$ by \eqref{claim_g}
and $\|\sqrt{{r^2}\xi_n}\|_{2}\rightarrow \|\sqrt{{r^2}\xi}\|_{2}\,$   by \eqref{conv_norm_2}, 
we have the strong convergence  \begin{equation}\label{l2_st_conv}
\sqrt{{r^2}\xi_n}\rightarrow \sqrt{{r^2}\xi}\quad\mbox{in}\quad L^2(\mathbb{R}^3)\quad\mbox{as}\quad n\to\infty.
\end{equation} 
We note that for the ball $B^\varepsilon$ in \eqref{st_cpt} and for all $n\geq k_0(\varepsilon)$,
  \begin{equation}\begin{split}\label{est_decomp_l1w} 
\int {r^2}|\xi_n-\xi|\dd x&\leq 
\int_{B^\varepsilon }{r^2}|\xi_n-\xi|\dd x
+\int_{\mathbb{R}^{3}\setminus B^\varepsilon}{r^2} \xi_n \dd x
+ \int_{\mathbb{R}^{3}\setminus B^\varepsilon}{r^2} \xi\dd x\\
& \leq \int_{B^\varepsilon }{r^2}|\xi_n-\xi|\dd x+2[\mu_n-(\mu-\varepsilon)]+ 2\varepsilon.
\end{split} \end{equation}
We claim, for each $\varepsilon>0$,
\begin{equation}\label{claim_l1}
 \lim_{n\to\infty}
\int_{B^\varepsilon }{r^2}|\xi_n-\xi|\dd x=0
\end{equation} for each fixed $\varepsilon>0$. Indeed, we estimate
\begin{align*}
\int_{B^\varepsilon }{r^2}|\xi_n-\xi|\dd x&\leq
\int_{B^\varepsilon }{r^2}|\xi_n-{1}_{A_n^c}|\dd x
+\int_{B^\varepsilon }{r^2}|{1}_{A_n^c}-\xi|\dd x=:I_n^\varepsilon+II_n^\varepsilon.
\end{align*} For the integral $I_n^\varepsilon$, we have,
by recalling the definition \eqref{defn_A_n} of $A_n$,
\begin{equation}\begin{split}\label{est_i_n_l1w}
 I_n^\varepsilon&=
\int_{B^\varepsilon\cap A_n }{r^2}|\xi_n-{1}_{A_n^c}|\dd x
+\int_{B^\varepsilon\cap A_n^c }{r^2}|\xi_n-{1}_{A_n^c}|\dd x\\
&= \int_{B^\varepsilon\cap A_n }{r^2}\xi_n\dd x
+\int_{B^\varepsilon\cap A_n^c }{r^2}|\xi_n-1|\dd x\\
&\leq\left( R(\varepsilon)^2\int_{A_n }\xi_n\dd x
+   R(\varepsilon)^2a_n |B^\varepsilon|\right)\to0\quad\mbox{as}\quad n\to \infty
 \end{split}\end{equation} by \eqref{condn_cpt}. 
   For the integral $II_n^\varepsilon$, we estimate, by recalling $\xi=1_A$,
\begin{align*}
 II_n^\varepsilon&=
\int_{B^\varepsilon }{r^2}|{1}_{A_n^c}-{1}_A|\dd x
=
\int_{B^\varepsilon }{r^2}|{1}_{A_n^c}-{1}_A|^2\dd x
=
\int_{B^\varepsilon }{r^2}|\sqrt{{1}_{A_n^c}}-\sqrt{{1}_A}|^2\dd x\\
&\leq 2\int_{B^\varepsilon }{r^2}|\sqrt{{1}_{A_n^c}}-\sqrt{\xi_n}|^2\dd x +2\int_{B^\varepsilon }{r^2}|\sqrt{\xi_n}-\sqrt{{1}_A}|^2\dd x
=: II_{n,1}^\varepsilon+ II_{n,2}^\varepsilon.
 \end{align*} 
As in \eqref{est_i_n_l1w},  we estimate
\begin{align*}
 II_{n,1}^\varepsilon&
 = 2\int_{B^\varepsilon\cap A_n }{r^2}| \sqrt{\xi_n}|^2\dd x  + 2\int_{B^\varepsilon\cap A_n^c }{r^2}|1-\sqrt{\xi_n}|^2\dd x \\
  & \leq  2\int_{ B^\varepsilon\cap A_n }{r^2}{\xi_n}\dd x  + 2\int_{B^\varepsilon\cap A_n^c }{r^2}|1-{\xi_n}|^2\dd x \\
   & \leq\left(  2R(\varepsilon)^2\int_{ A_n }{\xi_n}\dd x  + 2(a_n  R(\varepsilon))^2  |{B^\varepsilon}|\right)\to0\quad\mbox{as}\quad n\to \infty
 \end{align*} by \eqref{condn_cpt}. 
 Lastly, we observe, by recalling $\xi={1}_A$,  $$II_{n,2}^\varepsilon=2\int_{B^\varepsilon }{r^2}|\sqrt{\xi_n}-\sqrt{\xi}|^2\dd x=2\int_{B^\varepsilon }\Big|\sqrt{{r^2}\xi_n}-\sqrt{{r^2}\xi}\Big|^2\dd x\to 0 \quad \mbox{as}\quad n \to \infty$$ by \eqref{l2_st_conv}. Hence we get the claim \eqref{claim_l1}.\\
 
Combining the claim 
\eqref{claim_l1}
with \eqref{est_decomp_l1w}, we get, for each $\varepsilon>0$,
\begin{align*}
\limsup_{n\to\infty}\int {r^2}|\xi_n-\xi|\dd x\leq \limsup_{n\to\infty}
\int_{B^\varepsilon }{r^2}|\xi_n-\xi|\dd x
 +4\varepsilon\leq 4\varepsilon ,
\end{align*}\\
Sending   $\varepsilon \to0$, we obtain the convergence ${r^2}\xi_n\to {r^2}\xi$ in $L^{1}(\mathbb{R}^3)$ as $n\to\infty$, which is our goal 
\eqref{new_goal}. 
Lastly, the set ${\mathcal{S}}_{\mu,\nu,\lambda} $ is non-empty thanks to
Theorem \ref{thm_exist_max}.
It finishes the proof of Theorem \ref{thm_cpt}. \\

 \end{proof}



 \section{Uniqueness  of Hill's vortex}\label{sec_uniq_proof_hill}
 \subsection{Hill's problem and uniqueness result:  Amick-Fraenkel (1986)}\ \\
 
The final goal is to   
prove Theorem \ref{thm_uniq}.
First, we introduce the  setting  of Amick-Fraenkel \cite[Theorem 1.1]{AF86}.\\
  
 

The paper \cite{AF86} denotes $\mathcal H(\Pi)$ the completion of $C^\infty_c(\Pi)$ (the class of infinitely smooth and compactly supported functions in $\Pi$) in the norm $\|\cdot\|_{\mathcal{H}}$ from the inner product defined by
\begin{equation}\label{defn_h_inner}
\langle\phi,\psi\rangle_{\mathcal{H}}=\int_\Pi\frac{1}{r^2}((\partial_r\phi)(\partial_r\psi) +(\partial_z\phi)(\partial_z\psi))\,r\,drdz.
\end{equation}
Note if $\phi,\psi\in C^\infty_c(\Pi)$, then we can   integrate by parts (e.g. as in \eqref{ibp_dx}) to get another representation
\begin{equation}\label{equiv_norm}
\langle\phi,\psi\rangle_{\mathcal H}=\int_\Pi \left(-\frac{1}{r^2}\mathcal{L}\psi\right)\phi \,r\,drdz.
\end{equation}
For instance, we  have  $E[-{r^{-2}}\mathcal{L}\psi]=\pi\|\psi\|^2_{\mathcal{H}}$ for any $\psi\in C^\infty_c(\Pi)$
(see \eqref{defn_en} and \eqref{iden_en}).
This setting can be embedded in $\mathbb{R}^5$ in the following sense:\\

We denote $ y=(y_1,y_2,y_3,y_4,y_5)=(y',y_5)\in\mathbb{R}^5$. 
For a function $\phi:\Pi\to\mathbb{R}$, we define the   cylindrical symmetric function $\mathcal{T}[\phi]:\mathbb{R}^5\to\mathbb R$ by $$\mathcal{T}[\phi](y)=\frac{\phi(r,z)}{r^2},$$ where $r^2=|y'|^2=y_1^2+y_2^2+y_3^2+y_4^2$ and $z=y_5$. 
When
$\phi\in C^\infty_c(\Pi)$, we get
  $\mathcal{T}\phi\in C^\infty_c(\mathbb{R}^5\setminus\{r=0\})$, and it satisfies 
\begin{equation}\label{tran_h_iden} \Delta_{\mathbb{R}^5}(\mathcal{T}\phi)=\frac{1}{r^2}\mathcal L \phi=\mathcal{T}[\mathcal{L}\phi].\end{equation}
Moreover, for $\phi,\psi\in C^\infty_c(\Pi)$, we can compute
\begin{equation}\begin{split}\label{comp_iso}
 \langle\phi,\psi\rangle_{\mathcal{H}}&=\int_\Pi
\left[
r^3\left(\partial_r(\phi/r^2) \partial_r(\psi/r^2)+\partial_z(\phi/r^2)  \partial_z(\psi/r^2)\right)
+2\partial_r(\phi\psi/r^2) 
\right]drdz\\
&=\int_\Pi
\left[r^3\left(\partial_r(\phi/r^2)\partial_r(\psi/r^2) +\partial_z(\phi/r^2)\partial_z(\psi/r^2)\right)
\right]drdz=\frac{1}{2\pi^2}\int_{\mathbb{R}^5}
\nabla\mathcal{T}\phi\cdot \nabla\mathcal{T}\psi
d y,
\end{split}\end{equation}
  which produces  the identity 
\begin{equation}\label{iden_norm}
\|\phi\|_{\mathcal{H}}=\frac 1 {\sqrt {2 \pi^2}}\|\nabla\mathcal{T}\phi\|_{L^2(\mathbb{R}^5)}.
\end{equation}
 For given constants 
  $\lambda, W>0$, 
 the authors of \cite{AF86} defined the \textit{Hill's problem} for   $(\lambda, W)$  in the following way (also see \cite{FB74}):\\

To find $\psi$ such that
$$-\frac{1}{r^2} \mathcal{L}\psi =\lambda f_H(\Psi),\quad \Psi:=\psi-\frac{1}{2}Wr^2$$
$$\psi|_{r=0}=0, \quad \psi(r,z)\rightarrow 0\quad \mbox{as}\quad r^2+z^2\rightarrow \infty\quad\mbox{in}\quad \overline{\Pi},$$
for the vorticity function  $f_H=1_{(0,\infty)}$ as defined in \eqref{defn_hill_vor_fct}.\\

In \cite{AF86},  any function $\psi\in \mathcal H(\Pi)\setminus\{0\}$ is said to be a \textit{weak} solution of the Hill's problem for  $(\lambda, W)$
  if  it satisfies
\begin{equation}\label{prob_hill}
 \langle\phi,\psi\rangle_{\mathcal{H}}=\lambda\int_{A(\psi)}\phi 
\, r\, dr dz\quad\mbox{for any}\quad \phi\in \mathcal H(\Pi), 
\end{equation}
where 
$$ A(\psi)=\{(r,z)\in\Pi\,|\,\psi(r,z)-\frac{1}{2}Wr^2>0\}.$$
Thanks to 
\eqref{equiv_norm}, we expect that  a weak solution $\psi$ satisfies
\begin{equation}\label{rough}
-\frac 1 {r^2} \mathcal{L}\psi= \lambda 1_{\{\psi(r,z)-(1/2)Wr^2>0\}}
\end{equation}
in a certain weak sense.
In our variational setting,  
 for any maximizer $\xi\in\mathcal{S}_{\mu,\nu,\lambda}$,
  the stream function $\psi=\mathcal{G}[\xi]$
 satisfies
\eqref{rough}
once we assume   $\gamma=0$  in \eqref{eq_W}. 
 Thus it becomes  a weak solution of the Hill's problem. We write the statement in the form of a lemma, whose proof in detail is given in Appendix \ref{app_hill_sol} (even if it looks very natural and  trivial):
\begin{lem}\label{lem_hill_sol} If $\xi\in\mathcal{S}_\mu$ satisfies

 \begin{equation}\label{eq_W_no_g_}
\begin{aligned}
&\xi= {{1}}_{ \{\psi-
(1/2)Wr^2>0\}} \quad\mbox{ a.e.}\quad\mbox{ for some }\quad W>0,
\end{aligned}
\end{equation}
  then
  the stream function $\psi=\mathcal{G}[\xi]$
  is a weak solution of the Hill's problem for $(1,W)$.
\end{lem}

 
 
Now we borrow the uniqueness result 
of 
 Amick-Fraenkel  
\cite{AF86}: 
\begin{thm}\label{thm_AF}[Theorem 1.1 in \cite{AF86}]
If $\psi\in \mathcal H(\Pi)\setminus\{0\}$ is a weak solution of the Hill's problem   for $(\lambda,W)$, then we have $$\psi(r,z)=\psi_{H(\lambda,a)}(r,z-c)$$ for some $c\in\mathbb{R}$ and 
 for the  constant $a=a(\lambda,W)>0$ solving the equation
$W=(2/15)\lambda a^2$
where $\psi_{H(\lambda,a)}$ is the stream function
\eqref{defn_hill_gen_st} of the Hill's vortex 
$$\xi_{H(\lambda,a)}=\lambda 1_{B_a}(x),\quad B_a\subset\mathbb{R}^3:\mbox{the ball centered at the origin with radius\,} a.$$
\end{thm} 
 \begin{rem}\label{rem_idea_uniqAF}
One of the key ideas  of \cite{AF86}
 is to use the observation \eqref{tran_h_iden}.  Indeed,
 for  a weak solution $\psi$ of the Hill's problem for $(\lambda, W)$,  
 we expect \eqref{rough} (i.e. \eqref{eq_W} for $\gamma=0$). It implies,
  by  \eqref{tran_h_iden},
 \begin{equation*} 
 -\Delta_{\mathbb{R}^5}(\mathcal{T}\psi) 
 =-\mathcal{T}(\mathcal L \psi)=\lambda f_H(\psi-(1/2)Wr^2)=\lambda f_H(\mathcal{T}\psi-(1/2)W)\geq 0\quad\mbox{in}\quad\mathbb{R}^5.
\end{equation*}
One may expect spherical symmetry of $\mathcal{T} \psi$ in $\mathbb{R}^5$ (up to a translation in $y_5$-direction) via  the moving plane method due to \cite{Serrin71}, \cite{GNN}. 
The main difficulty lies on the fact that the vorticity function $f_H$ is not regular enough to satisfy the original setting of \cite{GNN} directly. 
By overcoming the obstacle (see Section 3 of \cite{AF86}),
spherical symmetry of $\mathcal{T} \psi$  is obtained.
We may assume that
$\mathcal{T}\psi$ is radially symmetric by shifting in $y_5$-variable
if necessary.  Moreover, it is strictly decreasing in the radial direction.   Then,
  it only remains to solve 
for  
$\eta(|y|)=(\mathcal{T}\psi)(y) $ and some unknown $a>0$ to the following O.D.E. problem:
\begin{equation*}\label{ode_hill}
\begin{split}
&\eta \in C^1[0,\infty): \mbox{strictly decreasing},\\
& -\frac 1 {t^4} (t^4 \eta')'=\lambda, \quad 0<t<a,\\
 & -\frac 1 {t^4} (t^4 \eta')'=0, \quad t>a,\\
 & \eta(a)=\frac{1}{2}W,\quad \eta(\infty)=0.\\
\end{split}
\end{equation*} It has the unique solution $\eta$ so that 
$\psi(x)=r^2\eta(|x|)$ is equal to $\psi_{H(\lambda,a)}(x)$ in \eqref{defn_hill_gen_st}. 
The radius $a$ is determined by \eqref{def_a}
 (cf. for the circular vortex pair
\eqref{lamb_dipole}, 
   refer to \cite{Burton96} or 
  Section 6.2 of \cite{AC2019}).
 \ \\
 \end{rem}
\subsection{Every maximizer with small impulse loses certain mass. }\ \\

In order to prove Theorem \ref{thm_uniq}, we first show the following proposition saying that 
every maximizer with small impulse has zero flux constant $\gamma$. 

 \begin{prop}\label{prop_small_mu}
   There exists a constant $ M_1>0$ such that for any 
 $0<\mu\leq M_1$ and for each $\xi\in \mathcal{S}_{\mu}$, we have 
 $$\int_{\mathbb{R}^3}\xi\,dx<1.$$ In that case,
 the flux constant $\gamma$   in \eqref{eq_W} of Theorem \ref{thm_max_is_ring} is equal to $0$. 
 \end{prop}
\begin{rem}\label{rem_small_mu} The same result  for a certain maximizer $\xi\in\mathcal{S}_\mu$ can be found in \cite[Remark 5.2]{FT81}, which was obtained from some uniform estimates for a sequence of maximizers for the penalized  energy functional \eqref{defn_pen_en}. Here we adapt  the proof so that it works for \textit{every} maximizer (cf.\cite[Remark 2.6 (iii)]{AC2019} for the circular vortex pair \eqref{lamb_dipole}).
\end{rem}

To prove 
Proposition \ref{prop_small_mu}, we need the following  estimate  of the kernel $G$
({cf}. Lemma 4.8 in \cite{FT81}).
 \begin{lem}\label{lem_est_psi_axis}
For $\alpha>0$, we have
$$ {\int_{ r'< \alpha} }G(r,z,r',z')r'dr'dz'\lesssim \alpha^4, \quad (r,z)\in\Pi.$$
\end{lem} 
\begin{proof}

Let $\alpha>0$ and $(r,z)\in\Pi$.  
 Since we have  
 \begin{equation*}\begin{split}
{\int_{ r'< \alpha} }G(r,z,r',z')r'dr'dz'={\int_{ r'< \alpha} }G(r,0,r',z'-z)r'dr'dz'
={\int_{ r'< \alpha} }G(r,0,r',z')r'dr'dz',
\end{split}\end{equation*}  we may assume $z=0$. We denote  $$  t 
=\sqrt{(r-r')^2+(z-z')^2}=r\sqrt{\left(1-\frac{r'}{r}\right)^2+\left(\frac{z-z'}{r}\right)^2}.$$
 When $r>2\alpha$, we can estimate, by setting $\tau=3/2$ in \eqref{est_F} ,
\begin{equation}\label{alpha_change}\begin{split}
{\int_{ r'< \alpha} }G(r,0,r',z')r'dr'dz'&\lesssim
\int_{\substack{r'< \alpha }}\frac{r^{2}(r')^{3}}{t^{3}}dr'dz'=r^{2}
\int_{\substack{r'< \alpha }}\left(\frac{r'}{r}\right)^{3}\left(\sqrt{\left(1-\frac{r'}{r}\right)^2+\left(\frac{z'}{r}\right)^2}\right)^{-3}dr'dz'\\
&= r^{4}
\int_{\substack{r'< \alpha/r }}
  ({r'})^{3}
  \left(\sqrt{\left(1-{r'}\right)^2+{z'}^2}\right)^{-3}dr'dz'\\
  &= r^{4}  \left(\frac{\alpha}{r}\right)^{3}
\int_{0}^{\alpha/r }\int_{\mathbb{R}}
  \left(\sqrt{\left(1-{r'}\right)^2+{z'}^2}\right)^{-3}dz'dr' \\&
\lesssim\alpha^3r  
\int_0^{\alpha/r} 
  \frac{1}{\left(1-r'\right)^{2}} dr' 
  \lesssim\alpha^4,
  \end{split}\end{equation}
  where we used change of variables 
  $\frac {r'}{r}\mapsto r',\quad\frac{z'}{r}\mapsto z'.$\\
  
  Now we consider  the remained case $r\leq 2\alpha$. We 
 split the integral
\begin{equation*}\begin{split}
{\int_{ r'< \alpha} }G(r,0,r',z')r'dr'dz'=\int_{\substack{r'< \alpha,\\ t<r/2}}
+\int_{\substack{r'< \alpha,\\ t\geq r/2}}=:I+II.
\end{split}\end{equation*} 
For $I$, we have, by \eqref{est_F} with $\tau=1/2$,
\begin{equation*}\begin{split}
\int_{\substack{r'< \alpha,\\ t< r/2}}G(r,0,r',z')r'dr'dz'&\lesssim
\int_{\substack{r'< \alpha,\\ t< r/2}}\frac{r(r')^{2}}{t}dr'dz'=r^{2}
\int_{\substack{r'< \alpha,\\ t< r/2}}\left(\frac{r'}{r}\right)^{2}\left(\sqrt{\left(1-\frac{r'}{r}\right)^2+\left(\frac{z'}{r}\right)^2}\right)^{-1}dr'dz'\\
&\leq r^{4}
\int_{\substack{
 \sqrt{(1-{r'})^2+{z'}^2}<1/2}}
  ({r'})^{2}
  \left(\sqrt{\left(1-{r'}\right)^2+{z'}^2}\right)^{-1}dr'dz'\\
  &\lesssim r^{4}
\int_{\substack{
 \sqrt{(1-{r'})^2+{z'}^2}<1/2}}
  \left(\sqrt{\left(1-{r'}\right)^2+{z'}^2}\right)^{-1}dr'dz'\lesssim 
r^4\lesssim \alpha^4.
\end{split}\end{equation*} 
For $II$, we estimate, by \eqref{est_F} with $\tau=3/2$ as in \eqref{alpha_change},
\begin{equation*}\begin{split}
\int_{\substack{r'< \alpha,\\ t\geq r/2}}G(r,0,r',z')r'dr'dz'&\lesssim
\int_{\substack{r'< \alpha,\\ t\geq r/2}}\frac{r^{2}(r')^{3}}{t^{3}}dr'dz'=r^{2}
\int_{\substack{r'< \alpha,\\ t\geq  r/2}}\left(\frac{r'}{r}\right)^{3}\left(\sqrt{\left(1-\frac{r'}{r}\right)^2+\left(\frac{z'}{r}\right)^2}\right)^{-3}dr'dz'\\
&= r^{4}
\int_{\substack{r'< \alpha/r,\\
 \sqrt{(1-{r'})^2+{z'}^2}\geq 1/2}}
  ({r'})^{3}
  \left(\sqrt{\left(1-{r'}\right)^2+{z'}^2}\right)^{-3}dr'dz'\\
  &\lesssim r^{4}\left(\frac{\alpha}{r}\right)^{3}
\int_{\substack{
 \sqrt{(1-{r'})^2+{z'}^2}\geq 1/2}}
  \left(\sqrt{\left(1-{r'}\right)^2+{z'}^2}\right)^{-3}dr'dz'
\lesssim r^{4}\left(\frac{\alpha}{r}\right)^{3}\lesssim \alpha^4.  
 \end{split}\end{equation*}

\end{proof}

\color{black}

Now we prove Proposition \ref{prop_small_mu}.
\begin{proof}[Proof of Proposition \ref{prop_small_mu}]\ \\

Let  $\xi\in \mathcal{S} _{\mu}$ for some $\mu\in(0,\infty)$. Then, by Theorem \ref{thm_max_is_ring}, there exist unique $W=W_\xi>0, \gamma=\gamma_\xi\geq 0$ such that  $\xi={1}_{\{\psi-(1/2)Wr^2-\gamma>0\}}$. Moreover,
$\xi$ is compactly supported. 
  By $$\mu=\frac{1}{2}{\int}r^2\xi dx\geq \frac{1}{2}{\int}_{r\geq2\sqrt\mu}r^2\xi dx\geq 2\mu   \int_{r\geq2\sqrt\mu} \xi\dd x,$$ we get
  $\int_{r\geq2\sqrt\mu} \xi\dd x\leq \frac{1}{2}, $
which implies
\begin{align}\label{mass_1_2}
 \int  \xi  dx\leq  \int_{0<r<2\sqrt\mu} \xi\dd x+\frac 1 2.
\end{align} 
Since we have
$$0\leq  \xi={1}_{\{\psi-\frac{1}{2}Wr^2-\gamma>0\}}\leq \frac{2\psi}{Wr^2},$$
we estimate, for any $\alpha>0
$,
\begin{align*}
\int_{\alpha \leq r<2 \alpha}\xi\dd x
&=2\pi \int_{\alpha \leq r<2 \alpha}\xi r drdz\leq 
2\pi \int_{\alpha \leq  r<2 \alpha}\frac{2\psi}{Wr^2}rdrdz\\ 
&= \frac{4\pi}{\alpha^2 W}\int_{\alpha\leq r<2\alpha} \left({\int_\Pi}G(r,z,r',z')\xi(r',z')r'dr'dz'\right)rdrdz \\
&= \frac{4\pi}{\alpha^2 W}\int_{\Pi} \xi(r,z)\left( {\int_{\alpha\leq r'<2\alpha} }G(r,z,r',z')r'dr'dz'\right)rdrdz \\
&\lesssim \frac{1}{\alpha^2 W}
\sup_{(r,z)\in\Pi}\left( {\int_{\alpha\leq r'<2\alpha} }G(r,z,r',z')r'dr'dz'\right)
\int \xi \,dx,
\end{align*}  where we used the symmetry of $G$ in the last equality.
Using Lemma \ref{lem_est_psi_axis}, we get
\begin{align*}
\int_{\alpha \leq r<2 \alpha}\xi\dd x
&\lesssim \frac{\alpha^{2}}{ W}\int \xi\,dx\leq \frac{\alpha^{2}}{ W}.
\end{align*} 
Thus we get
\begin{equation}\begin{split}\label{mu_over_W}
\int_{0 <r<2\sqrt \mu }\xi\dd x&= \sum_{i=0}^\infty\int_{(2\sqrt\mu) 2^{-i-1} \leq r<(2\sqrt\mu) 2^{-i}}\xi\dd x
\lesssim \frac{1}{W}\sum_{i=0}^\infty \left((2\sqrt\mu) 2^{-i-1}\right)^{2}
\lesssim \frac{ \mu }{W}\sum_{i=0}^\infty \left( \frac 1 4\right)^{i} 
\lesssim \frac{\mu }{W}.
\end{split}\end{equation}  Now we recall the estimate \eqref{est_W_lower} of Proposition \ref{prop_cpt_supp}:
\begin{equation*}\label{3d_Est} \mathcal{I}_\mu\leq 2W\mu.\end{equation*}

On the other hand, we claim
$$\mathcal I_\mu \geq \mathcal{I}_1\mu^{7/5}$$ for  any $\mu\leq 1$. Indeed, it is a simple consequence from scaling. Let us take and fix any $\xi_1\in \mathcal{S}_1$ whose existence is guaranteed by Theorem \ref{thm_exist_max}. By setting 
$$\xi_\mu(x)=\xi_1(\mu^{-1/5}x),$$ we get 
$\xi_\mu\in \mathcal{P}_\mu$ for any $\mu\leq 1$.
Thus, we get the above claim due to $\mathcal I_\mu\geq E(\xi_\mu)= \mu^{7/5}E(\xi_1)=\mu^{7/5}\mathcal I_1$. \\

Hence, we have 
$ 2W\mu\geq \mathcal I_\mu \geq \mathcal{I}_1\mu^{7/5}$ for any $\mu\leq 1,
$ which gives
\begin{equation}\label{lower_W}
  W   \geq \frac{\mathcal{I}_1}{2}\mu^{2/5},\quad\mu\leq 1.
\end{equation}
In sum,  for any $\mu\leq1$ and for any $\xi\in\mathcal{S}_\mu$, we have, by \eqref{mass_1_2}, \eqref{mu_over_W}, \eqref{lower_W},
$$   \int \xi  dx \leq 
\int_{0 <r<2\sqrt \mu }\xi\dd x 
+\frac 1 2
\leq   {C\mu^{3/5}}+\frac 1 2 $$ for some universal constant $C>0$.
Thus there exists a sufficiently small constant $M_1>0$ such that for any 
$0<\mu\leq M_1$ and for any $\xi\in\mathcal{S}_\mu$, we have  $$\int \xi dx<1,$$ which implies $$\gamma=\gamma_\xi=0$$ by Lemma \ref{lem_pos_gam}.
\end{proof}
 Now we are ready to prove Theorem \ref{thm_uniq}.
\subsection{Proof of  uniqueness theorem (Theorem \ref{thm_uniq})}\ \\

 \begin{proof}[Proof of Theorem \ref{thm_uniq}]

Thanks to the  scaling argument \eqref{scaling}, it is enough to show the theorem for general $\mu>0$ with fixed $\lambda=\nu=1$.\\

Let $\mu\in(0,M_1]$ where $M_1>0$ is the constant from Proposition \ref{prop_small_mu}. 
Due to
 $\mathcal{S}_\mu\neq \emptyset$ from 
Theorem \ref{thm_exist_max}, we can take any $\xi\in\mathcal{S}_\mu$. Then, by Theorem \ref{thm_max_is_ring}, we get 
 \begin{equation*}
\begin{aligned}
&\xi= {{1}}_{ \{ \psi-(1/2)W_\xi r^2-\gamma_\xi >0\}}.
\end{aligned}
\end{equation*} for some constants $W_\xi>0$ and $\gamma_\xi\geq0$.  
By Proposition \ref{prop_small_mu}, we get $\gamma_\xi=0$, which implies
$$7 \mathcal{I}_\mu=5W_\xi\cdot\mu$$ by the identity
\eqref{est_W_lower} of Proposition \ref{prop_cpt_supp}. In other words, $W_\xi$ is determined by  knowing only the value of $\mu$. Let us denote $W_\xi$ by $W_\mu$ from now on. Then, by Lemma \ref{lem_hill_sol}, the stream function $\psi=\mathcal{G}[\xi]$ is a weak solution of the Hill's problem for $(\lambda,W)=(1,W_\mu)$. 
  We 
set 
the radius $a_\mu>0$ solving $W_\mu=(2/15)(a_\mu)^2$. 
  Then, by Theorem \ref{thm_AF}, there exists  a constant $c'\in\mathbb{R}$ such that
$$\psi(x)=\psi_{H(1,a_\mu)}(x+c'e_z),$$
where 
$\psi_{H(1,a_\mu)}$ is the stream function
\eqref{defn_hill_gen_st} of the Hill's vortex 
$\xi_{H(1,a_\mu)}=1_{B_{a_\mu}}$. Thus we get
$\xi(x)=\xi_{H(1,a_\mu)}(x+c'e_z).$ In sum, we have shown, for any $0<\mu\leq M_1$, 
$$\emptyset\neq\mathcal{S}_\mu\subset \{
\xi_{H(1,a_\mu)}(\cdot+ce_z)\,|\,  c\in\mathbb{R}
\}. $$ In particular,
the radius $a_\mu$ is explicitly computed by 
$$\mu 
=\frac{1}{2}\int r^2 \xi dx
=\frac{1}{2}\int r^2 \xi_{H(1,a_\mu)}dx=\frac{4\pi}{15}(a_\mu)^5$$ (e.g. see \eqref{comp_imp}).
 To show the reverse inclusion, we recall that any  translation in $z-$variable does not change the quantities involved in the variational problem \eqref{var_prob}. Thus, from 
$\xi_{H(1,a_\mu)}(\cdot+c'e_z)=\xi\in \mathcal{S}_\mu$ for some $c'\in\mathbb{R}$, 
 we obtain $$\mathcal{S}_\mu\supset \{
\xi_{H(1,a_\mu)}(\cdot+ce_z)\,|\, c\in\mathbb{R}
\}. $$

\end{proof}


 
\appendix
\addcontentsline{toc}{section}{Appendices}

 \section{Proof of Lemma \ref{lem_stream_AR}}\label{app_en_conv}
 
We begin with proving the following estimate of the stream function $\psi=\mathcal{G}[\xi]$ when $\xi$ satisfies the {the monotonicity }condition \eqref{cond_sym}. The result 
  is essentially due to \cite[Lemma 3.5]{FT81}. Here we follow the approach of \cite[Proposition 3.3]{AC2019}.
\begin{lem} \label{lem_stream_sym}
Let  $\xi\in (L^1_w \cap L^{2}\cap L^{1})  (\mathbb{R}^{3})$   be an  axi-symmetric nonnegative function satisfying {the monotonicity }condition \eqref{cond_sym}.
Then, 
$\psi=\mathcal{G}[\xi]$ satisfies 
\begin{align}\label{est_stream_sym}
\psi(r,z)  \lesssim \left(\|\xi\|_{L^1\cap L^2}  +\|r^2\xi\|_{1} \right)\cdot\left(\frac{r^2}{\sqrt A} +
\frac{1}{  A} +
 r^2\left(\frac{A}{|z|}\right)^3  \right),\quad (r,z)\in\Pi
\end{align} provided 
$r\leq \frac{|z|}{A},$ $|z|>0,$ and $A\geq1$. 
\end{lem}
 \begin{proof}

    By replacing $A$ to $A/2$, 
it is equivalent to show  \eqref{est_stream_sym} for $r\leq 2|z|/A$ and $A\geq 2$.
    Let us take $ (r,z)\in\Pi$ satisfying $r\leq 2|z|/A$.    
     We may assume that $z>0$.    
By setting $$t=
|(r,z)-(r',z')|,$$
we split the integral
\begin{equation*}\begin{split}
\psi(r,z)
= 
\int_\Pi
G(r,z,r',z')
\xi(r',z')r'dr'dz'=
\int_{t<r/2}\dots+\int_{t\geq r/2}\dots=:I+II.
\end{split}\end{equation*}
For the term $I$, we estimate, by \eqref{est_F} with $\tau=1/6$ and by H\"older's inequality,
\begin{equation*}\begin{split}
I&\lesssim\int_{t<r/2}\left({\sqrt{rr'}}\,
 \left(\frac{rr'}{t^2}\right)^{1/6}
 (r')^{1/3}\right) (r')^{2/3}\xi(r',z') dr'dz' \\
&\lesssim\left(\int_{t<r/2}  
 \frac{(rr')^{2}}{t}
 r' dr'dz'\right)^{1/3}   \|\xi {1}_{\{t<r/2\}}\|_{3/2}=:   I_1\cdot I_2.
\end{split}\end{equation*}
 For $t<r/2$, we have $r\sim r'$ so we estimate
\begin{equation*}\begin{split}
I_1&\leq r^{5/3}\left(\int_{t<r/2}  
 \frac{1}{t}
  dr'dz'\right)^{1/3} 
   \lesssim r^{2}.
 \end{split}\end{equation*}
We observe that the conditions $t< r/2$ and $r\leq 2z/A$  imply $|z-z'|< z/A$. 
We also observe that, for any    function $g:\mathbb{R}_{>0}\to\mathbb{R}_{\geq0}$ which is non-increasing,
\begin{align}\label{g_est}
\int_{s-(s/A)}^{s+(s/A)}g(\sigma)\dd \sigma\leq \frac{4}{A}||g||_{L^{1}(0,\infty)}\quad s>0,\ A\geq 2,
\end{align}   thanks  to 
 $sg(s)\leq ||g||_{L^1(0,\infty)}$ for any $s>0$.  By the assumption 
\eqref{cond_sym}, we can apply
 \eqref{g_est} into the one-dimensional function $\xi(r',\cdot_{z'})$ of $z'$-variable (by fixing $r'$), which produces 
 \begin{equation*}\begin{split}
I_2&\leq   \left(\|\xi {1}_{\{t<r/2\}}\|_{1} + \|\xi {1}_{\{t<r/2\}}\|_{2} \right) \leq   \left(\|\xi {1}_{\{|z-z'|< z/A\}}\|_{1} + \|\xi {1}_{\{|z-z'|< z/A\}}\|_{2} \right)\\ &\lesssim \left(\frac{1}{A}\|\xi\|_1+\frac{1}{\sqrt{A}}\|\xi\|_2\right)
\lesssim\frac{1}{\sqrt A}  \|\xi\|_{L^1\cap L^2}.
\end{split}\end{equation*} 
Thus we get 
$$
I\lesssim\frac{r^2}{\sqrt A}\|\xi\|_{L^1\cap L^2}.$$
For the term $II$,    by \eqref{est_F} with $\tau=3/2$,
we estimate
\begin{equation*}\begin{split}
II&\lesssim\int_{t\geq r/2}\left( {\sqrt{rr'}} \,\left(\frac{rr'}{t^2}\right)^{3/2}\right)\xi(r',z')r'dr'dz' =  \int_{t\geq r/2}\left(  \frac{r^2(r')^3}{t^3}  \right)\xi(r',z')dr'dz' \\
&=\int\limits_{\substack{t\geq r/2, \\ |z-z'|< z/A}}
+\int\limits_{\substack{t\geq r/2, \\ |z-z'|\geq z/A}}
 =: II_1+II_2.
\end{split}\end{equation*} 
Since
$t\geq r/2$ implies $r'\leq |r'-r|+r\leq 3t$,  we have, by \eqref{g_est},
\begin{equation*}\begin{split}
 II_1 &
 \lesssim \int\limits_{ |z-z'|< z/A}    (r')^2  \xi(r',z')dr'dz'
 \leq  \int\limits_{ |z-z'|< z/A}    \left((r')^1+(r')^3\right)  \xi(r',z')dr'dz'\\
 &\lesssim  \frac{1}{  A}\left(\|\xi\|_{1} + \|r^2\xi\|_{1}\right).
\end{split}\end{equation*} 
  For $II_2$, since
  $|z-z'|\geq  z/A $ implies
$t\geq     z/A $, 
  we have
\begin{equation*}\begin{split}
 II_2 \leq  r^2 \int\limits_{t\geq     z/A}\left(  \frac{ (r')^3}{t^3}  \right)\xi(r',z')dr'dz' 
 \lesssim r^2\left(\frac{A}{z}\right)^3   \|r^2\xi\|_{1}.\\
\end{split}\end{equation*} 
In sum, we obtained
\begin{equation*}\begin{split}
\psi(r,z)&\lesssim \frac{r^2}{\sqrt A} \|\xi\|_{L^1\cap L^2} +
\frac{1}{  A}\left(\|\xi\|_{1} + \|r^2\xi\|_{1}\right)+
r^2\left(\frac{A}{z}\right)^3 \|r^2\xi\|_{1},
\end{split}\end{equation*} which implies \eqref{est_stream_sym}. 
  \end{proof}
\color{black}


Now we are ready to prove Lemma \ref{lem_stream_AR}.
    \begin{proof}[Proof of Lemma \ref{lem_stream_AR}]
      
We decompose

\begin{align*}
\int_{\mathbb{R}^{3}\backslash Q }\psi \xi \dd x
=\int_{r\geq R}+\int_{\substack{r<R, \\ |z|\geq AR}}=:I+II ,
\end{align*}\\
and estimate, by \eqref{est_psi_bdd} with $\delta=1$,
\begin{align*}
I\lesssim 
\left(\|\xi\|_{L^1\cap L^2}  +\|r^2\xi\|_{1} \right) \int_{r\geq R} \frac{r^2}{r^2} \xi\dd x
\lesssim \frac{1}{R^{2}}\left(\|\xi\|_{L^1\cap L^2}  +\|r^2\xi\|_{1} \right)^2.
\end{align*}\\
For $II$, since  $r<R$ and $|z|\geq AR$  imply $r\leq |z|/A$, applying \eqref{est_stream_sym} yields 

\begin{align*}
II
&\lesssim \left(\|\xi\|_{L^1\cap L^2}  +\|r^2\xi\|_{1} \right) \int\limits_{\substack{r<R, \\ |z|\geq AR}} 
\left(\frac{r^2}{\sqrt A} +
\frac{1}{  A} +
 r^2\left(\frac{A}{|z|}\right)^3  \right)
 \xi \dd x \\
 &\lesssim \left(\|\xi\|_{L^1\cap L^2}  +\|r^2\xi\|_{1} \right) \int 
\left(\frac{r^2}{\sqrt A} +
\frac{1}{  A} +
 \frac{r^2}{R^3}   \right)
 \xi \dd x \\
  &\lesssim \left(\frac{1}{\sqrt A}  +
 \frac{1}{R^3}   \right)\cdot \left(\|\xi\|_{L^1\cap L^2}  +\|r^2\xi\|_{1} \right)^2.
\end{align*} Combining the above estimates, we obtain the conclusion \eqref{est_stream_AR}.
 
  \end{proof}


 \section{Proof of Lemma \ref{lem_exist_one}}\label{app_excep}

 \begin{proof}[Proof of Lemma \ref{lem_exist_one}]

 Let $\Omega\subset\mathbb{R}^N$  be a non-empty connected open set and  let $U\subset \Omega$ satisfy  $|U|>0$ and $|\Omega\setminus U|>0$.
Since $|\mathcal{D}_e(U)|=|U|>0$ by Lemma \ref{lem_lebesgue}, we can take some point  $y\in\mathcal{D}_e(U)$. Similarly, from 
$|\mathcal{D}_i(U)|=|\Omega\setminus U|>0$, we   take another point  $z\in \mathcal{D}_i(U)$. 
 We connect $y$ and $z$ by a polygonal line $L$ consisting of a finite number of line segments joined end to end lying in $\Omega$ and set $$r_0=
\mbox{dist}(L,\Omega^c)
>0. 
$$   (For the case $\Omega=\mathbb{R}^N$,   we simply take $r_0=1$.)
Let us show the existence of   an exceptional  point.
For  $r\in(0,r_0)$, we define $f_r:L\to\mathbb{R}$ by $$f_r(x)=\frac{|B_{r}(x)\cap U|}{|B_{r}(x)|},\quad x\in L.$$
Since $y\in\mathcal{D}_e(U)$  and $z\in\mathcal{D}_i(U)$, there exists $r_1\in(0,r_0/2)$ such that 
$$ f_{r_1}(y)\geq 3/4,\quad 
 f_{r_1}(z)\leq 1/4.$$
 Since 
$f_{r_1}$ is continuous on $L$,  there exists  $x_1\in L$ satisfying
 $$\frac{|B_{r_1}( {x_1})\cap U|}{|B_{r_1}({ {x_1})}|}=\frac{1}{2}.$$  We note
 $\overline{B_{r_1}( {x_1})}\subset \Omega$.\\

 By an induction, we can construct a sequence of positive numbers $\{r_n\}_{n=1}^\infty$ and
a sequence $\{ {x_n} \}_{n=1}^\infty$ of points in $\Omega$ such that (by setting $x_0=x_1$) 
$$0<r_{n}<\frac{r_{n-1}}{2},\quad
   B_{r_{n}}({x_{{n}}})\subset  B_{
 r_{n-1}}( {x_{n-1}}), \quad
\frac{|B_{r_{n}}( {x_{n}})\cap U|}{|B_{r_{n}}({ {x_{n}})}|}=\frac{1}{2}
 \quad
 \mbox{ for any  } n\geq 1. $$ 
 Indeed, we assume that there exist
 $r_k, x_k$ satisfying the above conditions up to $k=1,2,\dots, n$. Then, by applying Lemma \ref{lem_lebesgue} to $U\cap B_{r_n}( {x_n})$, we have $ |\mathcal{D}_e[U]\cap B_{r_n}( {x_n})|>0$. Similarly, 
$ |\mathcal{D}_i[U]\cap B_{r_n}( {x_n})|=|\mathcal{D}_e[\Omega\setminus U]\cap B_{r_n}( {x_n})|>0$.
 Therefore we can take $ {y'}, {z'}\in B_{r_n}( {x_n})$ such that
$$ {y'}\in\mathcal{D}_e(U)\quad\mbox{and}\quad
 {z'}\in\mathcal{D}_i(U). $$
We  take sufficiently small $r_{n+1}\in(0, r_n/2)$ such that $B_{r_{n+1}}( {y'}), B_{r_{n+1}}( {z'})\subset B_{r_n}( {x_n}),$ and $$\frac{|B_{r_{n+1}}( {y'})\cap U|}{|B_{r_{n+1}}( {y'})|}\geq \frac{3}{4}\quad\mbox{and}\quad \frac{|B_{r_{n+1}}( {z'})\cap U|}{|B_{r_{n+1}}( {z'})|}\leq \frac{1}{4}.$$ By the same argument in the above, we have a point
$ {x_{n+1}}$ on the line segment connecting
$ {y'}$ and $ {z'}$ with
$$\frac{|B_{r_{n+1}}( {x_{n+1}})\cap U|}{|B_{r_{n+1}}({ {x_{n+1}})}|}=\frac{1}{2}.$$
Clearly, we have  $B_{r_{n+1}}({ {x_{n+1}})}\subset B_{r_{n}}({ {x_{n}})}$. \\

 This construction guarantees  that  $\lim_n {x_n}=:x$ exists and 
 $\{x\}=\cap_{n\geq 1} \overline{B_{r_n}( {x_n})}\subset \Omega$. In particular, we can verify that the limit $x$ is an exceptional point of $U$. Indeed,
  from
 $  B_{r_n}( {x_n})\subset B_{2r_{n}}( {x})$,  we observe 
 \begin{equation*}\begin{split}
{|B_{2 r_{n}}( {x})\cap U|}&\leq {|B_{r_{n}}( {x_n})\cap U|}
+|B_{2 r_n}( {x})\setminus B_{r_n}( {x_n})|
 \leq \frac{1}{2}|B_{r_{n}}( {x_n})|
+ (2^N-1)|B_{r_{n}}( {x_n})|\\ &
 =
\left(2^N-\frac 1 2\right)|B_{r_{n}}( {x_n})|=
\frac{2^N- (1/2)}{2^N}|B_{2r_{n}}({x})|\quad\mbox{for any } n\geq 1.
 \end{split}\end{equation*}
 Similarly, we get
  \begin{equation*}\begin{split}
{|B_{2 r_{n}}( {x})\cap (\Omega\setminus U)|} \leq \frac{2^N- (1/2)}{2^N}|B_{2r_{n}}( {x})|\quad\mbox{for any } n\geq 1.
 \end{split}\end{equation*} 
 It implies
   \begin{equation*}\begin{split}
0<\frac{1/2}{2^N}\leq \frac{{|B_{2 r_{n}}( {x})\cap U|}}{|B_{2r_{n}}( {x})|} \leq \frac{2^N- (1/2)}{2^N}<1\quad\mbox{for any } n\geq 1.
 \end{split}\end{equation*} Thus $x\notin(\mathcal{D}_e(U)\cup \mathcal{D}_i(U))$, which means $x\in\mathcal{E}(U)$.
\ \\

\end{proof}
 
\section{Proof of Lemma \ref{lem_hill_sol}}\label{app_hill_sol}

We begin with the following observations:\\

 For any $\phi\in C^\infty_c(\Pi)$,  
 we have that $\mathcal{T}\phi\in L^{10/3}(\mathbb{R}^5)$  is compactly supported with
 \begin{equation}\label{GNS}
\| \mathcal{T}\phi\|_{L^{10/3}(\mathbb{R}^5)}\lesssim \|\nabla\mathcal{T}\phi\|_{L^{2}(\mathbb{R}^5)}
\end{equation}  by  Gagliardo-Nirenberg-Sobolev inequality
(e.g. see \cite[p277]{Evans_book}). 
  In addition,   
  we have
\begin{equation}\label{poincare}
\int_{\Pi} |\phi|^{10/3}r^{-11/3}drdz\lesssim \|\phi\|^{10/3}_{\mathcal{H}}\end{equation}  by \eqref{iden_norm} and by the computation
  \begin{equation*}\begin{split} \label{comp_10_3}
\int_{\mathbb{R}^5} |\mathcal{T}\phi|^{10/3} d y=2\pi^2\int_\Pi |\phi/r^2|^{10/3}r^3drdz=
2\pi^2\int_{\Pi} |\phi|^{10/3}r^{-11/3}drdz.
\end{split}\end{equation*} 

Let $E$ be the Hilbert space which is the completion of $C^\infty_c(\mathbb{R}^5)$ in the norm $\|\cdot\|_E$ from the inner product
$$\langle f,g\rangle_{E}
=\frac 1 {2\pi^2}\int_{\mathbb{R}^5}\nabla f\cdot\nabla g \,dy.$$
\color{black}
We denote $E_s$ the closed linear subspace of $E$ which is formed by completing  $C_{c,s}^\infty(\mathbb{R}^5)$  (the class\footnote{
Here we use the subscript `s' for cylindrical symmetry and `c' for compact support while the original paper \cite{AF86} used subscript `c' for the symmetry and  `0' for compact support.
} of infinitely smooth and compactly supported functions $f$ in $\mathbb{R}^5$ with the \textit{cylindrical symmetry} $f(y)=f(r,z)$) in the norm $\|\cdot\|_E$. 
 We observe, for $f,g\in E_s$,
$$\langle f,g\rangle_{E}
=\int_\Pi\left((\partial_r f)(\partial_r g) +(\partial_z f)(\partial_z g)\right){r^3}\,drdz.$$
Then,
 \cite[Lemma 2.2]{AF86}  says that the space $\mathcal{H}(\Pi)$ defined from 
\eqref{defn_h_inner} 
  can be identified 
with the   space $E_s$  via the transform $\mathcal{T}$.

\begin{lem}\label{AF86_lem}[Lemma 2.2 in \cite{AF86}]
The spaces $\mathcal{H}(\Pi)$ and $E_s$ are isometrically isomorphic
 under the transformation $f=\mathcal{T}[\phi]$ of any $\phi\in\mathcal{H}(\Pi)$ or $f\in E_s$. 
\end{lem}
The proof follows the computation \eqref{comp_iso} once we note that the space
$C_{c,s}^\infty(\mathbb{R}^5\setminus\{r=0\})$ is dense in $C_{c,s}^\infty(\mathbb{R}^5)$ under the norm $\|\cdot\|_{E}$, hence in $E_s$.
For the detail, we refer to the proof of \cite[Lemma 2.2]{AF86}. \\

By Lemma \ref{AF86_lem} and by 
the definition of  $\mathcal H(\Pi)=\overline{C_c^\infty(\Pi)}$, the identity \eqref{iden_norm}  holds for any $\phi\in\mathcal{H}(\Pi)$. 
Similarly,
the estimates
\eqref{GNS} and \eqref{poincare} hold for   $\phi\in\mathcal{H}(\Pi)$. 
For deep discussions about such homogeneous spaces $E, E_s$, 
we recommend \cite[Sections II.6, II.7]{Gal}. \\

Now we are ready to prove Lemma \ref{lem_hill_sol}.
\begin{proof}[Proof of Lemma \ref{lem_hill_sol}]
Let $\xi\in\mathcal{S}_\mu$ satisfy \eqref{eq_W_no_g_}.
We recall the definition  \eqref{defn_h_inner} of the norm $\|\cdot\|_{\mathcal{H}}$.
By  \eqref{iden_en} of Lemma \ref{lem_en_iden_origin}  and by 
Theorem \ref{thm_exist_max}, we have
$$\pi\|\psi\|_{\mathcal{H}}^2= E[\xi]=\mathcal{I}_\mu\in(0,\infty),$$
where 
$\psi:=\mathcal{G}[\xi]$ is the stream function $ \xi$. 
 Now we show $\psi\in \mathcal{H}(\Pi)$ which  is equivalent to prove 
$\mathcal{T}\psi\in E_s$ due to Lemma \ref{AF86_lem}. 
In the proof below, we denote 
$x\in\mathbb{R}^3$ and $y\in\mathbb{R}^5$.
Since the axi-symmetric function $\xi(\cdot_x)$ lies on $\left( L^1_w\cap L^\infty\right)(\mathbb{R}^3)$,   we get
$$\xi(\cdot_y)\in \left(L^1 \cap L^\infty\right) (\mathbb{R}^5)$$ by 
$$\frac 1 \pi \int_{\mathbb{R}^5} |\xi |d y=2\pi\int_\Pi r^3|\xi| drdz= \int_{\mathbb{R}^3}r^2|\xi |dx.$$
Thanks to Theorem \ref{thm_max_is_ring} and
Lemma \ref{lem_psi_hol_conti}, we know that
$$ \xi(\cdot_y)
\quad\mbox{is compactly supported in}\quad \mathbb{R}^5,$$
\begin{equation}\label{tpsi}
\mathcal{T}\psi\in BUC^\alpha(\overline{\mathbb{R}^5}),\quad 0<\alpha<1,\quad\mbox{and}\quad
(\mathcal{T}\psi)(y)\to 0\quad\mbox{as}\quad |y|\to\infty,\quad y\in\mathbb{R}^5.
\end{equation}
We also observe,  
 by Lemma \ref{lem_en_iden_origin} and
 by the identity \eqref{tran_h_iden},  
\begin{equation*}
\mathcal{T}\psi\in H^2_{loc}(\mathbb{R}^5\setminus\{r=0\})\quad\mbox{and}\quad 
-\Delta_{\mathbb{R}^5}(\mathcal{T}\psi)=-\mathcal{T}[\mathcal{L}\psi]
=\xi\quad
\mbox{a.e.}\quad
\mbox{in}\quad\mathbb{R}^5\setminus\{r=0\}.
\end{equation*}
Moreover, for any cylindrical symmetric bounded subset $U\subset\mathbb{R}^5$ with the corresponding axi-symmetric bounded set $\tilde{U}\subset\mathbb{R}^3$,
we have
\begin{align*}
\int_{U}| \mathcal{T}\psi|^2dy&\lesssim   \int_{\tilde{U}}\Big|\frac{ \psi}{r}\Big|^2dx 
\lesssim |\tilde{U}|_{\mathbb{R}^3}\|{\psi}/r\|^2_{L^\infty(\mathbb{R}^3)}<\infty
\end{align*} by \eqref{est_psi} and
\begin{align*}
\int_{U}|\nabla_{\mathbb{R}^5}\mathcal{T}\psi|^2dy&\lesssim 
\int_{U}\left(\Big|\frac{\partial_r\psi}{r^2}\Big|^2+\Big|\frac{\partial_z\psi}{r^2}\Big|^2+ \Big|\frac{ \psi}{r^3}\Big|^2\right)dy\lesssim 
\int_{\tilde{U}}\left(\Big|\frac{\partial_r\psi}{r}\Big|^2+\Big|\frac{\partial_z\psi}{r}\Big|^2\right)dx+ \int_{\tilde{U}}\Big|\frac{ \psi}{r^2}\Big|^2dx
\\&
\lesssim E[\xi]+|\tilde{U}|_{\mathbb{R}^3}\|\mathcal{T}{\psi}\|^2_{L^\infty(\mathbb{R}^5)}<\infty
\end{align*} by \eqref{iden_en} and \eqref{tpsi}.
It gives
$ 
\mathcal{T}\psi\in H^1_{loc}({\mathbb{R}^5}).
$
Hence the Poisson equation 
 $-\Delta_{\mathbb{R}^5}(\mathcal{T}\psi) 
=\xi$ is satisfied in a weak sense  in any ball in $\mathbb{R}^5$, which gives
the following representation of $\mathcal{T}\psi$ via the fundamental solution 
in $\mathbb{R}^5$:
$$\mathcal{T}\psi=\frac{1}{8\pi^2| \cdot_y|^3} *_{\mathbb{R}^5}\xi(\cdot_y).$$
This representation 
implies (as in the proof of Lemma \ref{lem_en_iden_origin})
 $$\mathcal{T}\psi\in W^{2,p}(\mathbb{R}^5)\cap BUC^{1+\alpha}(\overline{\mathbb{R}^5}), \quad p>5/3, \quad 0<\alpha<1.$$
In particular, we have
$$ \mathcal{T}\psi \in H^1(\mathbb{R}^5).
$$ Since
$\mathcal{T}\psi$ is cylindrical symmetric, we conclude
$\mathcal{T}\psi\in  E_s$, 
which implies $\psi\in \mathcal{H}(\Pi)$  by Lemma \ref{AF86_lem}.\\

It remains to show that the weak formulation \eqref{prob_hill} holds. Due to our assumption \eqref{eq_W_no_g_}, it is equivalent to  show 
\begin{equation}\label{prob_hill_}
 \langle\phi,\psi\rangle_{\mathcal{H}}
 = \int_\Pi\phi\, 
\xi\, 
 r\,dr dz\quad\mbox{for any}\quad \phi\in \mathcal H(\Pi). 
\end{equation}
First, we observe
that the right-hand side of \eqref{prob_hill_} makes sense for any 
$\phi\in\mathcal{H}(\Pi)$ 
thanks to 
\eqref{poincare}
and the fact $\xi\in (L^\infty\cap L^1_w)(\mathbb{R}^3)$. Indeed,   we can estimate, by 
\eqref{poincare},
\begin{equation}\begin{split}\label{comp_5_}
\int_\Pi|\phi| \xi  r dr dz&= \int_{\Pi}|\phi|r^{-11/10} \xi r^{21/10}  dr dz
\leq \left(\int_\Pi  |\phi|^{10/3}r^{-11/3}drdz\right)^{3/10} \left(
\int_\Pi \xi^{10/7} r^3drdz
\right)^{7/10}\\
&\lesssim \| \phi \|_{\mathcal{H}} \left(\|\xi\|^{3/7}_\infty
\int_{\mathbb{R}^3}   r^2\xi dx
\right)^{7/10}
= \| \phi \|_{\mathcal{H}}\|\xi\|^{3/10}_\infty  \|r^2\xi\|_{1}^{7/10} 
\lesssim \mu^{7/10}\| \phi \|_{\mathcal{H}}.
\end{split}\end{equation} 
Second, \eqref{prob_hill_} is clear for any $\phi \in C^\infty_c(\Pi)$ by integration by parts thanks to 
  $$\xi=-\frac{1}{r^2}\mathcal{L}\psi\quad a.e.$$ from Lemma \ref{lem_en_iden_origin}.  Lastly, for a general $\phi\in \mathcal H(\Pi)$, we take a sequence $\{\phi_n\}\subset C^\infty_c(\Pi)$ such that
$\phi_n\to\phi$ in $\mathcal{H}(\Pi)$. For the left-hand side of \eqref{prob_hill_},
we know $\langle\phi_n,\psi\rangle_{\mathcal{H}}\to \langle\phi,\psi\rangle_{\mathcal{H}}$ as $n\to\infty$.
For the right-hand side of \eqref{prob_hill_}, as in the computation
\eqref{comp_5_}, we have the convergence
$$   \int_\Pi|\phi_n-\phi| 
\xi 
 r dr dz \lesssim \mu^{7/10}\|\phi_n-\phi\|_{\mathcal{H}} \to0\quad \mbox{as}\quad n \to \infty.$$ Hence we
   obtain \eqref{prob_hill_} for any $\phi\in\mathcal{H}(\Pi)$.
\color{black}
\end{proof}

 \section*{Acknowledgements}

The work  is partially supported by the NRF-2018R1D1A1B07043065 (National Research Foundation of Korea)  and by the Research Fund 1.200085.01 of UNIST(Ulsan National Institute of Science \& Technology). We thank Ken Abe for many helpful discussions, including discussions on the identity \eqref{def_W_gamma_uniq}.
Without his     encouragement, this research work would not have been possible. We also thank the anonymous referee for valuable comments.
\vspace{15pt}

 \bibliographystyle{abbrv}
\bibliography{ref_01_2022_Hill_vortex}
\end{document}